\documentclass[11pt,reqno]{amsart}

\usepackage{amsmath,amssymb,amsthm,cite}

\usepackage[abbrev]{amsrefs}
\usepackage{graphicx}

\usepackage{mathtools}
\mathtoolsset{showonlyrefs=true}
\usepackage{mathrsfs}
\usepackage{bigints}

\usepackage[top=30mm, bottom=30mm, left=25mm, right=25mm]{geometry}

\usepackage{color}

\newtheorem{theo}{Theorem}[section]
\newtheorem{lemm}[theo]{Lemma}
\newtheorem{cor}[theo]{Corollary}
\newtheorem{prop}[theo]{Proposition}

\numberwithin{equation}{section}

\theoremstyle{definition}
\newtheorem{defi}[theo]{Definition}
\newtheorem{exam}[theo]{Example}
\newtheorem{rema}[theo]{Remark}

\newcommand{\BN}{\mathbb{N}}

\newcommand{\BR}{\mathbb{R}}
\newcommand{\BS}{\mathbb{S}}

\newcommand{\bA}{\boldsymbol{A}}
\newcommand{\bb}{\boldsymbol{b}}

\newcommand{\bff}{\boldsymbol{f}}

\newcommand{\bh}{\boldsymbol{h}}

\newcommand{\bl}{\boldsymbol{l}}

\newcommand{\bn}{\boldsymbol{n}}

\newcommand{\bu}{\boldsymbol{u}}
\newcommand{\bU}{\boldsymbol{U}}
\newcommand{\bv}{\boldsymbol{v}}
\newcommand{\bw}{\boldsymbol{w}}
\newcommand{\bz}{\boldsymbol{z}}

\newcommand{\bfeta}{\boldsymbol{\eta}}
\newcommand{\bfpsi}{\boldsymbol{\psi}}
\newcommand{\bfvarphi}{\boldsymbol{\varphi}}
\newcommand{\bfxi}{\boldsymbol{\xi}}

\newcommand{\CF}{\mathcal{F}}

\newcommand{\Fa}{\mathfrak{a}}
\newcommand{\Fb}{\mathfrak{b}}

\newcommand{\SH}{\mathscr{H}}
\newcommand{\SL}{\mathscr{L}}
\newcommand{\SR}{\mathscr{R}}

\newcommand{\const}{{\mathrm{const}}}
\newcommand{\curl}{\mathrm{curl}\,}
\newcommand{\diam}{\mathop{\rm diam}}
\newcommand{\dist}{\mathrm{dist}\,}
\newcommand{\dv}{\mathrm{div}\,}

\newcommand{\supp}{\mathrm{supp}\,}
\newcommand{\esssup}{\mathop{\rm ess\,sup}}
\newcommand{\dx}{\,\mathrm{d}x}

\newcommand{\be}{\begin{equation}}
\newcommand{\ee}{\end{equation}}
\newcommand{\ba}{\begin{aligned}}
\newcommand{\ea}{\end{aligned}}

\newcommand{\abs}[1]{\lvert#1\rvert}
\newcommand{\norm}[1]{\lVert#1\rVert}

\usepackage{scalerel}
\usepackage[usestackEOL]{stackengine}
\let\oldmathchoice\mathchoice
\usepackage{breqn}
\let\newmathchoice\mathchoice

\def\dashint{\let\mathchoice\oldmathchoice\,\ThisStyle{\ensurestackMath{%
            \stackinset{c}{.2\LMpt}{c}{.5\LMpt}{\SavedStyle-}{%
            \SavedStyle\phantom{\int}}}%
        \setbox0=\hbox{$\SavedStyle\int\,$}\kern-\wd0}\int%
        \let\mathchoice\newmathchoice}

\begin{document}
\title[]{Existence theorems for the steady-state Navier-Stokes equations with nonhomogeneous slip boundary conditions in two-dimensional multiply-connected bounded domains}
\author[]{Giovanni P. Galdi}
\address{Giovanni P. Galdi, Department of Mechanical Engineering and Materials Science, University of Pittsburgh, Pittsburgh, PA 15261, USA}
\email{galdi@pitt.edu}

\author[]{Tatsuki Yamamoto}
\address{Tatsuki Yamamoto, Department of Mathematics, Faculty of Science and Engineering, Waseda University, Tokyo 169--8555, Japan}
\email{tatsu-yamamoto@akane.waseda.jp}
\begin{abstract}
We study the nonhomogeneous boundary value problem for the steady-state Navier-Stokes equations under the slip boundary conditions in two-dimensional multiply-connected bounded domains. Employing the approach of Korobkov-Pileckas-Russo (Ann. Math. 181(2), 769-807, 2015), we prove that this problem has a solution if the friction coefficient is sufficiently large compared with the kinematic viscosity constant and the curvature of the boundary. No additional assumption (other than the necessary requirement of zero total flux through the boundary) is imposed on the boundary data. We also show that such an assumption on the friction coefficient is redundant for the existence of a solution in the case when the fluxes across each connected component of the boundary are sufficiently small, or the domain and the given data satisfy certain symmetry conditions. The crucial ingredient of our proof is the fact that the total head pressure corresponding to the solution to the steady Euler equations takes a constant value on each connected component of the boundary.
\end{abstract}

\maketitle
\date{}

\section{Introduction}\label{Intro}
Let $\Omega$ be a bounded domain of $\BR^2$ with $C^{\infty}$-boundary $\partial\Omega=\cup_{j=0}^N \Gamma_j$ consisting of $N+1$ disjoint components $\Gamma_j$, i.e., 
\be\label{domain}
\Omega=\Omega_0\setminus\bigl(\bigcup\limits_{j=1}^N\overline{\Omega}_j\bigr),
\ee
where $\Omega_j$ ($j=0,1,\dots,N$) are simply-connected bounded domains such that (i) $\overline{\Omega}_j\subseteq\Omega_0$ for all $j=1,2,\dots,N$; (ii) $\overline{\Omega}_i\cap\overline{\Omega}_j=\emptyset$ if $i,j\in\{1,2,\dots,N\}$ and $i\neq j$; (iii) $\Gamma_j=\partial\Omega_j$ ($j=0,1,\dots,N$) are smooth closed curves. Consider in $\Omega$ the steady-state Navier-Stokes equations
\be \left.\ba\label{NS}
\smallskip
-\nu\Delta\bu+\bu\cdot\nabla\bu+\nabla p=\bff\\
\dv\bu=0\ea\right\}\ \ \mbox{in $\Omega$}\,
\ee
under the nonhomogeneous slip boundary conditions
\be \left.\ba\label{slip}
\smallskip
\bu \cdot\bn= a_{\ast}\\
[T(\bu,p)\bn]_{\tau}+\beta \bu_{\tau}=\bb_{\ast}\ea\right\}\ \ \mbox{on $\partial\Omega$.}\,
\ee
Here $\bu=\bu(x)=(u_1(x),u_2(x))$ and $p=p(x)$ denote the unknown velocity field and the unknown pressure at the point $x=(x_1,x_2)\in\Omega$, $\nu>0$ is the given kinematic viscosity constant, $\bff=\bff(x)=(f_1(x), f_2(x))$ is the given external force field defined on $\Omega$, while $\bn$ is the unit outer normal to $\partial\Omega$, and $T(\bu,p)$ is the Cauchy stress tensor given by
\be
T(\bu,p)=-pI+\nu S(\bu),\quad S(\bu)= \nabla \bu+(\nabla \bu)^\top,
\ee
where $I$ is the $2\times2$ identity matrix. Also, $\beta=\beta(x)\ge0$ is the given nonnegative scalar-valued function defined on $\partial\Omega$, which is called the friction coefficient, and the subscript $\tau$ denotes the tangential component of a vector, i.e., $\bu_{\tau}=\bu-\bn(\bn \cdot \bu)$. Finally, $a_{\ast}=a_{\ast}(x)$ and $\bb_{\ast}=\bb_{\ast}(x)=(b_{\ast1}(x),b_{\ast2}(x))$ are the given boundary data. In view of the incompressibility condition $\dv\bu=0$ in $\Omega$, the boundary datum $a_{\ast}$ is required to satisfy the following compatibility condition:
\be\label{flux}
\int_{\partial\Omega} a_{\ast} \, ds=\sum_{j=0}^{N} \int_{\Gamma_{j}} a_{\ast} \, ds=0.
\ee
Moreover, another boundary datum $\bb_{\ast}$ must fulfill the condition $\bb_{\ast}\cdot\bn=0$ on $\partial\Omega$. 

The boundary value problem for the Navier-Stokes equations \eqref{NS} in a domain $\Omega\subseteq\BR^n$, $n=2,3$, with the boundary $\partial\Omega$ has been intensively studied under the no-slip boundary condition
\be\label{no_slip}
\bu=\bu_{\ast} \quad\text{on }\partial \Omega,
\ee
which requires that the particles of the liquid adhere to the boundary $\partial\Omega$ in the case when the region of flow $\Omega$ is surrounded by solid (undeformable) walls. However, there are situations in which the no-slip boundary condition \eqref{no_slip} may not be valid. In such situations, a standard alternative is the slip boundary conditions \eqref{slip}, which was proposed for the first time by Navier \cite{Na1823}. The slip boundary conditions \eqref{slip} are employed as an appropriate model for free surface problems and for flows past chemically reacting walls; see, e.g., \cite{MR884296}. They are also considered in the presence of rough boundaries \cites{MR2585993,MR1813101}. For a further description of the slip boundary conditions \eqref{slip}, we refer the reader to \cite{MR2761078}*{Sections 2 and 3}, \cite{MR2808162}*{p.\,5 and p.\,22}, \cite{MR108116}*{Section 64} and the references cited therein.

The main objective of this paper is to prove the existence of a weak solution $\bu\in W^{1,2}(\Omega)$ to the Navier-Stokes problem \eqref{NS}-\eqref{slip} under no restrictions on $a_{\ast}$ other than \eqref{flux} in two-dimensional multiply-connected bounded domains. The first study on the solvability of the nonhomogeneous boundary value problem for the stationary Navier-Stokes equations goes back to Leray \cite{MR3533002}. In his celebrated paper \cite{MR3533002} published in 1933, Leray proved that the Navier-Stokes problem \eqref{NS} with the nonhomogeneous no-slip boundary condition \eqref{no_slip} has a solution provided
\be\label{zero_flux_no_slip}
\int_{\Gamma_{j}} \bu_{\ast}\cdot\bn \, ds=0,\quad j=0,\dots,N.
\ee
Clearly, if $N\ge1$ (the boundary $\partial\Omega$ consists of more than one connected component $\Gamma_j$), the condition \eqref{zero_flux_no_slip} is stronger than the natural requirement
\be\label{total_flux_no_slip}
\int_{\partial\Omega} \bu_{\ast}\cdot\bn \, ds=\sum_{j=0}^{N} \int_{\Gamma_{j}} \bu_{\ast}\cdot\bn \, ds=0,
\ee
and in particular, it does not allow for the presence of sinks and sources. The question whether the problem \eqref{NS} with the no-slip boundary condition \eqref{no_slip} admits a solution only under the natural restriction \eqref{total_flux_no_slip} was left open by Leray, and it is known as Leray's problem. Since the appearance of the pioneer work by Leray \cite{MR3533002}, several partial answers to this problem have been given by a number of authors. For instance, the existence of a weak solution $\bu\in W^{1,2}(\Omega)$ to the problem \eqref{NS} with \eqref{no_slip} was established under the restriction \eqref{zero_flux_no_slip} (see, e.g., \cites{MR132307,MR5003,MR720205,MR254401}), or for sufficiently small fluxes $\int_{\Gamma_{j}} \bu_{\ast}\cdot\bn \, ds$\footnote{This condition does not assume the norm of the boundary value $\bu_{\ast}$ to be small.} (see, e.g., \cites{MR1308742,MR1228739,MR2808162,MR2491599,MR2550139,MR2679374}\footnote{Concerning the approach employed by Neustupa \cite{MR2679374}, see also \cite{KPR2024}*{Section 3.5}.}), or under certain symmetry assumptions on the planar domain $\Omega\subseteq\BR^2$, the boundary value $\bu_{\ast}$, and the external force $\bff$ (see, e.g., \cites{MR763943,MR1773581,MR2336076,MR1268380}), or for arbitrary large fluxes satisfying the inflow condition\footnote{The flux through the outer boundary $\Gamma_{0}$ is not positive, i.e., $\int_{\Gamma_{0}} \bu_{\ast}\cdot\bn \, ds\le0$.} in a doubly-connected, bounded planar domain $\Omega=\Omega_{0}\setminus\overline{\Omega}_1\subseteq\BR^2$ (see \cite{MR3004771}). Recently, Korobkov-Pileckas-Russo \cite{MR3275850} have successfully proved the solvability of the problem \eqref{NS} with \eqref{no_slip} only under the natural restriction \eqref{total_flux_no_slip} in arbitrary bounded planar or three-dimensional axially symmetric domains. However, Leray's problem for an arbitrary three-dimensional domain still remains open. For a more detailed survey of Leray's problem, we refer the reader to the survey papers \cites{MR3400557,MR3916777} and the book by Korobkov-Pileckas-Russo \cite{KPR2024}.

On the other hand, in contrast to the problem \eqref{NS} with the no-slip boundary condition \eqref{no_slip}, not much is known about the solvability of the stationary Navier-Stokes equations under the nonhomogeneous slip boundary conditions \eqref{slip} in a two-dimensional multiply-connected bounded domain. Under a certain assumption on the shape of the domain, the existence of a weak solution to the problem \eqref{NS}-\eqref{slip} corresponding to $\bb_{\ast}\equiv\boldsymbol{0}$ was shown by Mucha \cite{MR2150351} without any restrictions on $a_{\ast}$ (other than \eqref{flux}). However, it is not clear whether the condition imposed on the domain in \cite{MR2150351} could be fulfilled by a domain with more than one hole. Konieczny-Mucha \cite{MR2256131} proved the existence of a weak solution to the problem \eqref{NS}-\eqref{slip} in a simply-connected bounded domain. Afterward, it was shown in \cite{MR2491326}*{Section 1.1} that the kernel of the operator $\curl$, which is defined on the space of solenoidal vector fields satisfying the slip boundary conditions \eqref{slip} with $a_{\ast}\equiv0$ and $\bb_{\ast}\equiv\boldsymbol{0}$ is trivial in a doubly-connected bounded domain. As noted in \cite{MR2491326}, this fact might be used to extend the result of Konieczny-Mucha \cite{MR2256131} to the case when $\Omega$ is a multiply-connected bounded domain under the proper assumptions on the boundary data. Similar problem with the following nonhomogeneous boundary conditions
\be \left.\ba\label{Hodge}
\smallskip
\bu \cdot\bn= a_{\ast}\\
\curl\bu=b_{\ast}\ea\right\}\ \ \mbox{on $\partial\Omega$}\,
\ee
has been addressed by Neustupa \cite{MR3377379}. Here $\curl\bu\equiv\partial_2u_1-\partial_1u_2$ is the vorticity of a velocity vector $\bu=(u_1,u_2)$. In \cite{MR3377379}, given an external force $\bff$, which belongs to a certain subset of $L^2(\Omega)$, the existence of a strong solution to the problem \eqref{NS} with \eqref{Hodge} has been proved in arbitrary two-dimensional multiply-connected bounded domains only under the natural requirement \eqref{flux}. It should be noted that, in general, the boundary conditions \eqref{Hodge} are different from the slip boundary conditions \eqref{slip} even in the case that the friction coefficient $\beta\equiv0$. To see this, suppose that $\bu\cdot\bn=0$ on $\partial\Omega$ (i.e., $a_{\ast}\equiv0$). Then,  from Lemma \ref{lem_def-curl} and the identity $[T(\bu,p)\bn]_{\tau}=[\nu S(\bu)\bn]_{\tau}$ we deduce that
\be
\frac{1}{\nu}[T(\bu,p)\bn]_{\tau}-\curl\bu(n_2,-n_1)^\top=2W^\top\bu\enskip\text{on }\partial\Omega,
\ee
where $W$ is the Weingarten map of the boundary $\partial\Omega$ (in the direction of $\bn$) representing the curvature of $\partial\Omega$; see Proposition \ref{prop:Weingarten} for its definition and properties. The above identity implies that, if $\bu\cdot\bn=0$ on $\partial\Omega$ (i.e., $a_{\ast}\equiv0$), the boundary conditions \eqref{Hodge} and \eqref{slip} with $\beta\equiv0$ are different from each other by the term $2W^\top\bu$ unless the boundary $\partial\Omega$ is flat.

In this paper we shall show that the Navier-Stokes problem \eqref{NS}-\eqref{slip} admits at least one weak solution in arbitrary bounded planar domains if the friction coefficient $\beta$ is sufficiently large compared with the kinematic viscosity constant $\nu$ and the curvature of the boundary $\partial\Omega$; see Theorem \ref{theo1}. It should be emphasized that, in Theorem \ref{theo1}, no restriction other than the natural requirement \eqref{flux} is imposed on the fluxes of the boundary datum $a_{\ast}$ through each connected component $\Gamma_j$ of $\partial\Omega$. The assumption made on the friction coefficient $\beta$ in Theorem \ref{theo1} is redundant for the solvability of the problem \eqref{NS}-\eqref{slip} in some cases. Indeed, regardless of the value of $\beta$, the existence of a weak solution to the problem \eqref{NS}-\eqref{slip} is established for arbitrary large fluxes satisfying the outflow condition\footnote{The flux through the outer boundary $\Gamma_{0}$ is nonnegative, i.e., $\int_{\Gamma_{0}} \bu_{\ast}\cdot\bn \, ds\ge0$.} in a certain doubly-connected bounded domain (see Theorem \ref{theo2}), or for sufficiently small fluxes $\int_{\Gamma_j}a_{\ast}\,ds$ in arbitrary bounded domains (see Theorem \ref{theo4}), or under certain symmetry assumptions on the domain $\Omega$ and the given data $\bff, \beta, a_{\ast}$ and $\bb_{\ast}$ (see Theorem \ref{theo3}).

The Navier-Stokes problem \eqref{NS}-\eqref{slip} can be reduced to an operator equation in a Hilbert space with a compact operator, and so the existence of a weak solution to the problem \eqref{NS}-\eqref{slip} is readily established by the Leray-Schauder principle once we are able to show an appropriate a priori estimate for solutions to the operator equation. In \cite{MR3533002}, Leray developed two different approaches to prove such an a priori estimate for solutions to the problem \eqref{NS} with \eqref{no_slip}. The first approach in \cite{MR3533002} makes use of a special extension of the boundary datum $\bu_{\ast}$ into $\Omega$ of the form $\bA_{\epsilon}(x)=\curl(\theta(\epsilon,x)\bb(x))$, where $\theta(\epsilon,x)$ is Hopf's cut-off function \cite{MR5003}. For such an extension the estimate
\be\label{Leray-Hopf}
-\int_{\Omega}(\bw\cdot\nabla)\bA_{\epsilon}\cdot\bw\,dx\le\epsilon c\int_{\Omega}\abs{\nabla\bw}^2\,dx
\ee
holds for all solenoidal vector fields $\bw\in W^{1,2}_{0}(\Omega)$, where $\epsilon>0$ can be taken arbitrary small and the constant $c$ is independent of $\epsilon$; see, e.g., \cite{MR1846644}*{Section 1.4 in Chapter 2}. The inequality \eqref{Leray-Hopf} is referred to as Leray-Hopf's inequality.  It is well known that, in multiply-connected domains $\Omega\subseteq\BR^n$, $n=2,3$, some restrictions on the fluxes such as \eqref{zero_flux_no_slip} are indispensable for the validity of Leray-Hopf's inequality \eqref{Leray-Hopf}, that is, one cannot extend the boundary datum $\bu_{\ast}$ in any manner as a solenoidal extension $\bA_{\epsilon}$ satisfying the inequality \eqref{Leray-Hopf} only under the natural requirement \eqref{total_flux_no_slip}; see \cite{MR2957621}, \cite{MR2808162}*{pp.\,605-607}, \cites{MR3208789,MR2824494,MR1197051}. Also, it is worth noting that the classical method due to Hopf \cite{MR5003} for proving the inequality \eqref{Leray-Hopf} relies on the Hardy inequality
\be
\left\lVert{\frac{\bw}{d}}\right\rVert_{L^2(\Omega)}\le c\norm{\nabla\bw}_{L^2(\Omega)}\quad\text{for all }\bw\in W^{1,2}_{0}(\Omega),
\ee
where $d=d(x)=\dist(x,\partial\Omega)$ is the distance function from $\partial\Omega$; see, e.g., \cite{MR1846644}*{Proof of Lemma 1.8 in Chapter 2}. Since the slip boundary conditions \eqref{slip} with the homogeneous boundary data ($a_{\ast}\equiv0$ and $\bb_{\ast}\equiv\boldsymbol{0}$) do not imply $\bu|_{\partial\Omega}=\boldsymbol{0}$ (no information is available about the tangential component of $\bu$ on the boundary), the method of Hopf is not directly applicable to the problem \eqref{NS}-\eqref{slip} even if the domain $\Omega$ is simply-connected. In this respect, we refer the reader to Konieczny-Mucha \cite{MR2256131} for a modification of Hopf's method, which is adopted to the problem \eqref{NS}-\eqref{slip}. 

The second approach, suggested by Leray \cite{MR3533002}, is based on a contradiction argument. This idea has been subsequently used and modified in many papers (see, e.g., \cites{MR763943,MR720205,MR3004771,MR2550139,MR254401,MR2679374}), and it was further developed in the paper by Korobkov-Pileckas-Russo \cite{MR3275850} in which Leray's problem has been solved in arbitrary bounded planar and three-dimensional axially symmetric domains. In the present paper, we shall employ the same contradiction argument to prove the desired a priori estimate for solutions to the problem \eqref{NS}-\eqref{slip}. The crucial ingredient of our proof is the fact that the total head pressure corresponding to the solution to the steady Euler equations takes a constant value on each connected component $\Gamma_j$ of $\partial\Omega$; see Corollary \ref{cor:Bernoulli}. It is worth remarking that this fact fails in the three-dimensional case (see the counterexample given in Remark \ref{rem:Bernoulli} (iii)), and hence our argument is not directly applicable to the three-dimensional problem even in the case when the fluxes of the boundary datum $a_{\ast}$ satisfy the stronger condition
\be\label{zero_flux_slip}
\int_{\Gamma_{j}} a_{\ast}\, ds=0,\quad j=0,\dots,N.
\ee

We wish to end this introductory section with a remark on a priori estimates for solutions to the problem \eqref{NS}-\eqref{slip}. After this work was completed, we learned from Mikhail V. Korobkov that, in the case when $\Omega=\{x\in\BR^2:1<\abs{x}<2\}$, for some particular data we are able to find infinitely many explicit solutions $\{\bu_k\}_{k\in\BR}$ to the problem \eqref{NS}-\eqref{slip} such that $\norm{\nabla\bu_k}_{L^2(\Omega)}\to\infty$ as $k\to\infty$; see Example \ref{exam:Hamel}. This example implies that, in contrast to the problem \eqref{NS} with the no-slip boundary condition \eqref{no_slip} (see \cite{MR3275850}), solutions to the problem \eqref{NS}-\eqref{slip} do not admit a priori estimates in general, and that the additional assumptions (such as the requirement on the friction coefficient and the outflow condition) in our existence theorems are indispensable for deriving the required a priori estimate of solutions; see Remark \ref{rem:Hamel}. On the other hand, the fundamental question of whether the problem \eqref{NS}-\eqref{slip} is solvable for any kinematic viscosity constant $\nu>0$ and any friction coefficient $\beta=\beta(x)\ge0$ in two-dimensional multiply-connected bounded domains only under the natural restriction \eqref{flux} remains open.

The plan of the paper is as follows. In the next section, after formulating the problem \eqref{NS}-\eqref{slip} we shall state our existence theorems. In Section \ref{Prelimi}, we first collect some auxiliary results and then present existence, uniqueness, and global regularity results for the Stokes problem under the slip boundary conditions. For completeness, the proof of the global regularity result (Theorem \ref{thm:grS}) will be reported in the Appendix. In Sections \ref{Proof of Theorem 1} and \ref{Proof of Theorem 2}, we prove our existence theorems for arbitrary large fluxes (Theorems \ref{theo1} and \ref{theo2}) by employing the approach of Korobkov-Pileckas-Russo \cite{MR3275850}. Section \ref{Proof of Theorem 4} is dedicated to the proofs of our existence results for sufficiently small fluxes (Theorem \ref{theo4} and Corollary \ref{cor:harmonic}). The objective of Section \ref{Proof of Theorem 3} is to prove the existence of a symmetric solution to the problem \eqref{NS}-\eqref{slip} under suitable symmetry assumptions on the domain and the given data (Theorem \ref{theo3}). The proof of Theorem \ref{theo3} is based on a generalization of \cite{MR763943}*{Theorem 2.3}. More precisely, the key ingredient of the proof is the fact that the total head pressure corresponding to the solution to the steady Euler equation takes the same constant value on every connected component $\Gamma_j$ of $\partial\Omega$, provided the domain $\Omega$ and the solution to the Euler equation satisfy certain symmetry conditions; see Theorem \ref{thm:Amick_symmetric}.

\section{Results}\label{Results}
We shall begin by giving the definition of weak solutions to the Navier-Stokes problem \eqref{NS}-\eqref{slip}. To this end, we first recall some definitions and notations which will be frequently used throughout the paper. For a given normed space $X$, the corresponding norm is denoted by $\norm{\cdot}_{X}$. The set of all bounded linear functionals on $X$ is called the dual space of $X$ and is denoted by $X'$. The symbol $\langle\cdot,\cdot\rangle$ denotes a generic duality pairing.

By a domain we mean an open connected set. Let $\Omega\subseteq\BR^2$ be a bounded domain with Lipschitz boundary $\partial\Omega$. We use standard notation for function spaces\footnote{We shall use the same font style to denote function spaces for scalar and vector-valued functions.}: $C^m(\Omega)$, $C^m_{0}(\Omega)$, $C^m(\overline{\Omega})$, $C^m(\partial\Omega)$, $L^q(\Omega)$, $L^q(\partial\Omega)$, $W^{k,q}(\Omega)$, $W^{k,q}_{0}(\Omega)$, where  $m\in\BN_{0}$ or $m=\infty$, $q\in[1,\infty]$ and $k\in\BN_{0}$ ($W^{0,q}\equiv W^{0,q}_{0}\equiv L^q$). We denote by $C^{\infty}_{0,\sigma}(\Omega)$ the set of all $C^{\infty}$ vector-valued functions $\bfvarphi=(\varphi_1,\varphi_2)$ with compact supports in $\Omega$ such that $\dv\bfvarphi=0$ in $\Omega$. 

Let $1\le q<\infty$. Then each $f\in W^{1,q}(\Omega)$ has a well defined trace $f|_{\partial\Omega}\in L^q(\partial\Omega)$. More precisely, there exists a bounded linear operator 
\be
\gamma\colon W^{1,q}(\Omega)\to L^q(\partial\Omega)
\ee 
such that 
\be
\gamma (f)=f\enskip\text{on }\partial\Omega 
\ee
for all $f\in W^{1,q}(\Omega)\cap C(\overline{\Omega})$; see \cite{MR3409135}*{Theorem 4.6}, \cite{MR3726909}*{Theorem 18.1}. Let $m\in\BN$. The image space $\gamma(W^{m,q}(\Omega))\subseteq L^q(\partial\Omega)$ is denoted by $W^{m-1/q,q}(\partial\Omega)$ and is called the trace space. The space $W^{m-1/q,q}(\partial\Omega)$ is a Banach space with respect to the norm
\be
\norm{g}_{W^{m-1/q,q}(\partial\Omega)}\equiv\inf_{G\in W^{m,q}(\Omega);\gamma(G)=g}\norm{G}_{W^{m,q}(\Omega)}.
\ee

Let $0<\beta<1$ and $1<q<\infty$. The Sobolev space $W^{-\beta,q}(\partial\Omega)$ of negative order $-\beta$ is defined as the dual space of $W^{\beta,q'}(\partial\Omega)$, $q'=\frac{q}{q-1}$,
\be
W^{-\beta,q}(\partial\Omega)\equiv W^{\beta,q'}(\partial\Omega)'.
\ee
For $F\in W^{-\beta,q}(\partial\Omega)$, the norm of $F$ is defined by 
\be
\norm{F}_{W^{-\beta,q}(\partial\Omega)}\equiv\sup_{0\neq u\in W^{\beta,q'}(\partial\Omega)}\abs{\langle F,u \rangle_{\partial\Omega}}/\norm{u}_{W^{\beta,q'}(\partial\Omega)},
\ee
where $\langle \cdot,\cdot \rangle_{\partial\Omega}$ is the duality pairing between $W^{-\beta,q}(\partial\Omega)$ and $W^{\beta,q'}(\partial\Omega)$.

Other notations will be introduced according to necessity.
\subsection{Variational formulation and weak solutions}\label{subsec:variational_formulation}We next give a variational (or weak) formulation of the Navier-Stokes problem \eqref{NS}-\eqref{slip}. Let $(\bu,p)$ be a classical solution to the problem \eqref{NS}-\eqref{slip}, for example, $\bu\in C^2(\overline{\Omega})$ and $p\in C^1(\overline{\Omega})$, and let $\bfvarphi\in C^{\infty}(\overline{\Omega})$ satisfy $\bfvarphi\cdot\bn=0$ on $\partial\Omega$. Since
\be\label{08091423}
\dv(T(\bu,p)\bfvarphi)=\frac{\nu}{2}S(\bu):S(\bfvarphi)+(\nu\Delta\bu-\nabla p)\cdot\bfvarphi-p\,\dv\bfvarphi\enskip\text{in }\Omega,
\ee
where
\be
S(\bu):S(\bfvarphi)\equiv\sum_{i,j=1}^{2}\left(\frac{\partial u_i}{\partial x_j}+\frac{\partial u_j}{\partial x_i}\right)\left(\frac{\partial \varphi_i}{\partial x_j}+\frac{\partial \varphi_j}{\partial x_i}\right),
\ee
integrating the identity \eqref{08091423} over $\Omega$, we obtain the formula
\be\label{iden:Green}
\int_{\Omega}(-\nu\Delta\bu+\nabla p)\cdot\bfvarphi\,dx+\int_{\Omega}p\,\dv\bfvarphi\,dx=\frac{\nu}{2}\int_{\Omega}S(\bu):S(\bfvarphi)\,dx-\int_{\partial\Omega}[T(\bu,p)\bn]_{\tau}\cdot\bfvarphi\,ds.
\ee
Multiplying $\eqref{NS}_1$ by $\bfvarphi\in C^{\infty}(\overline{\Omega})$ with $\dv\bfvarphi=0$ in $\Omega$ and $\bfvarphi\cdot\bn=0$ on $\partial\Omega$, using the formula \eqref{iden:Green} and the boundary condition $\eqref{slip}_2$, we have the following integral identity:
\be
\frac{\nu}{2}\int_{\Omega}S(\bu):S(\bfvarphi)\,dx+\int_{\partial\Omega}\beta\bu_{\tau}\cdot\bfvarphi\,ds-\langle \bb_{\ast},\bfvarphi \rangle_{\partial\Omega}+\int_{\Omega}(\bu\cdot\nabla)\bu\cdot\bfvarphi\,dx=\langle \bff,\bfvarphi \rangle_{\Omega}.
\ee
We shall need two closed subspaces $H(\Omega)$ and $J(\Omega)$ of $W^{1,2}(\Omega)$ defined by
\be
\ba
H(\Omega)&\equiv\{\bu\in W^{1,2}(\Omega):\gamma(\bu)\cdot\bn=0\enskip\text{on }\partial\Omega\},\\
J(\Omega)&\equiv\{\bu \in H(\Omega):\dv \bu=0\enskip\text{in }\Omega\}.
\ea
\ee
We now give the following definition.
\begin{defi}
Let $\beta\in C(\partial\Omega)$ be nonnegative, and let $\bff\in H(\Omega)'$. Suppose that $a_{\ast}\in W^{1/2,2}(\partial\Omega)$ and $\bb_{\ast}\in W^{-1/2,2}(\partial\Omega)$ satisfy the following compatibility conditions: 
\be\label{theo1_comp}
\int_{\partial\Omega} a_{\ast}\,ds=\sum_{j=0}^{N} \int_{\Gamma_{j}} a_{\ast} \, ds=0\enskip\text{and}\enskip\bb_{\ast}\cdot\bn=0\enskip\text{on}\enskip\partial\Omega.
\ee
Then a vector field $\bu:\Omega\to\BR^2$ is called a weak (or generalized) solution to the Navier-Stokes problem \eqref{NS}-\eqref{slip} if 
\begin{enumerate}
\item $\bu\in W^{1,2}(\Omega)$; 
\item $\bu$ is divergence-free in $\Omega$;
\item $\bu$ satisfies the boundary condition $\eqref{slip}_{1}$ in the trace sense, namely,
\be
\gamma(\bu)\cdot\bn=a_{\ast}\quad\text{in}\enskip W^{1/2,2}(\partial\Omega);
\ee
\item $\bu$ obeys the integral identity
\be\label{def:weakNS_nonhom}
\frac{\nu}{2}\int_{\Omega} S(\bu): S(\bfvarphi) \,dx+\int_{\partial\Omega} \beta\bu_{\tau}\cdot\bfvarphi \,ds-\langle {\bb_{\ast}},\bfvarphi \rangle_{\partial\Omega}+\int_{\Omega}(\bu\cdot\nabla)\bu\cdot\bfvarphi\,dx=\langle \bff,\bfvarphi \rangle_{\Omega}
\ee
for all $\bfvarphi\in J(\Omega)$. 
\end{enumerate}
\end{defi}
\begin{rema}\label{rem:def_weak}(i) By virtue of Lemma \ref{pressure}, for a weak solution $\bu\in W^{1,2}(\Omega)$ to the problem \eqref{NS}-\eqref{slip} we are able to associate a pressure $p\in L^2(\Omega)$ such that $\int_{\Omega}p(x)\,dx=0$ and the identity
\be
\frac{\nu}{2}\int_{\Omega} S(\bu): S(\bfvarphi) \,dx+\int_{\partial\Omega} \beta\bu_{\tau}\cdot\bfvarphi \,ds-\langle {\bb_{\ast}},\bfvarphi \rangle_{\partial\Omega}+\int_{\Omega}(\bu\cdot\nabla)\bu\cdot\bfvarphi\,dx=\langle \bff,\bfvarphi \rangle_{\Omega}+\int_{\Omega}p\,\dv\bfvarphi\,dx
\ee 
holds for all $\bfvarphi\in H(\Omega)$.

(ii) If the given data are appropriately smooth, we can show that every weak solution $\bu$ to the problem \eqref{NS}-\eqref{slip} and the pressure $p$ associated with $\bu$ are more regular, i.e., $\bu\in W^{2,2}(\Omega)$ and $p\in W^{1,2}(\Omega)$. In such a case, the boundary condition $\eqref{slip}_2$ is satisfied in the following sense
\be
[T(\bu,p)\bn]_{\tau}+\beta \bu_{\tau}=\bb_{\ast}\enskip\text{in }L^2(\partial\Omega);
\ee
see Theorem \ref{thm:grNS}. 
\end{rema}

\subsection{Existence results for arbitrary large fluxes} We are now in a position to state our existence theorems. Our first result shows that the Navier-Stokes problem \eqref{NS}-\eqref{slip} is solvable in arbitrary bounded planar domains under no restrictions on fluxes other than the natural requirement \eqref{flux} if the friction coefficient is sufficiently large compared with the kinematic viscosity constant and the curvature of the boundary. More precisely, we shall prove the following.
\begin{theo}\label{theo1}
Let $\Omega\subseteq\BR^2$ be a smooth bounded domain defined by \eqref{domain}, and let the friction coefficient $\beta\in C^1(\partial\Omega)$ be a nonnegative function, which is not identically equal to zero. Suppose that the boundary data $a_{\ast}\in W^{3/2,2}(\partial\Omega)$ and $\bb_{\ast}\in W^{1/2,2}(\partial\Omega)$ satisfy the compatibility conditions \eqref{theo1_comp}, and the external force $\bff$ is in $W^{1,2}(\Omega)$. If the inequality
\be\label{theo1_assumption}
-2\kappa(x)\le\frac{\beta(x)}{\nu}
\ee
holds for all $x\in\partial\Omega$, then the Navier-Stokes problem \eqref{NS}-\eqref{slip} admits at least one weak solution $\bu\in W^{1,2}(\Omega)$. Here $\kappa=\kappa(x)$ is the curvature of the boundary $\partial\Omega$ at the point $x\in\partial\Omega$ in the direction of the unit outer normal to $\partial\Omega$ $($for its definition see Subsection $\ref{subsec:Weingarten})$. 
\end{theo}
\begin{exam}
Suppose that $\Omega$ is the domain defined by \eqref{domain} with $\Omega_0=\{x\in\BR^2:\abs{x}<R_0\}$ for some $R_0>0$ and $N$ convex domains $\Omega_j$ $(j=1,\dots,N)$. In this case, $\kappa(x)=-\frac{1}{R_0}$ for $x\in\Gamma_{0}$ and $\kappa(x)\ge0$ for $x\in\Gamma_j$ $(j=1,\dots,N)$. Then the condition \eqref{theo1_assumption} is necessarily satisfied on the inner boundaries, and hence the problem \eqref{NS}-\eqref{slip} is solvable if the friction coefficient $\beta$ satisfies the condition
\be
\frac{2}{R_0}\le\frac{\beta(x)}{\nu}\enskip\text{for all }x\in\Gamma_{0}=\{x\in\BR^2:\abs{x}=R_0\}.
\ee
\end{exam}
\begin{rema}When the domain $\Omega$ is doubly-connected (i.e., $N=1$), Mucha \cite{MR2150351} proved the existence of a weak solution to the problem \eqref{NS}-\eqref{slip} (with no restriction on $a_{\ast}$ other than \eqref{flux}) under the assumption that
\be\label{202409221057}
\norm{2\chi-\frac{\beta}{\nu}}_{L^{\infty}(\partial\Omega)}<C(\Omega),
\ee
where $\chi$ is the curvature of $\partial\Omega$ in the direction of the inner unit normal to $\partial\Omega$ (i.e., $\chi(x)=-\kappa(x)$, $x\in\partial\Omega$) and $C(\Omega)$ is a constant depending only on $\Omega$. Since the condition \eqref{202409221057} is equivalently rewritten as 
\be
-C(\Omega)<-2\kappa(x)-\frac{\beta(x)}{\nu}<C(\Omega)\enskip\text{for all }x\in\partial\Omega,
\ee
his result is not completely covered by Theorem \ref{theo1}.
\end{rema}
\vspace{0.3cm}
Theorem \ref{theo1} is valid for any bounded domains and any values of the fluxes of $a_{\ast}$ through $\Gamma_j$, $j=0,\dots,N$. On the other hand, the friction coefficient $\beta$ is required to be large (condition \eqref{theo1_assumption}). When $\Omega$ is a doubly-connected bounded domain with a convex inner hole, the assumption on the friction coefficient $\beta$ such as \eqref{theo1_assumption} is redundant for the existence of a weak solution if the flux through the outer boundary $\Gamma_{0}$ is nonnegative (outflow condition). Our second result now reads:
\begin{theo}\label{theo2}
Let $\Omega\subseteq\BR^2$ be a smooth bounded domain defined by \eqref{domain} with $N=1$, and let the friction coefficient $\beta\in C^1(\partial\Omega)$ be nonnegative. Suppose that the boundary data $a_{\ast}\in W^{3/2,2}(\partial\Omega)$ and $\bb_{\ast}\in W^{1/2,2}(\partial\Omega)$ satisfy the following compatibility conditions:
\be
\int_{\partial\Omega} a_{\ast}\,ds=\int_{\Gamma_0} a_{\ast}\,ds+\int_{\Gamma_1} a_{\ast}\,ds=0\enskip\text{and}\enskip\bb_{\ast}\cdot\bn=0\enskip\text{on}\enskip\partial\Omega,
\ee
and the external force $\bff$ is in $W^{1,2}(\Omega)$. Assume further that $\beta\not\equiv0$ in the case when the domain $\Omega$ is an annulus. If the inner hole $\Omega_1$ of the domain $\Omega$ is convex and 
\be\label{outflow_condition}
\int_{\Gamma_0} a_{\ast}\,ds\ge0,
\ee
then the Navier-Stokes problem \eqref{NS}-\eqref{slip} admits at least one weak solution $\bu\in W^{1,2}(\Omega)$.
\end{theo} 
\begin{rema}(i) Note that the assumption on the sign of $\int_{\Gamma_0} a_{\ast}\,ds$ in Theorem \ref{theo2} is opposite to that of \cite{MR3004771}*{Theorem 4.1}. 

(ii) In the case when the domain $\Omega$ is an annulus, the additional assumption that $\beta\not\equiv0$ in Theorem \ref{theo2} (and Theorem \ref{theo4} below) is required for the validity of the Korn-type inequality \eqref{ineq:Korn}. See also Remark \ref{rem:Korn} (ii).
\end{rema}

\begin{rema}\label{rem:theo1_reg} Under the smoothness assumptions on $\Omega$, $\bff$, $a_{\ast}$, $\beta$ and $\bb_{\ast}$ as in Theorems \ref{theo1} and \ref{theo2}, every weak solution $\bu$ and the pressure $p$ associated with $\bu$ become more regular,  i.e.,  $\bu\in W^{2,2}(\Omega)\cap W^{3,2}_{\mathrm{loc}}(\Omega)$ and $p\in W^{1,2}(\Omega)\cap W^{2,2}_{\text{loc}}(\Omega)$; see also Theorem \ref{thm:grNS}. Such regularity assumptions on the data are necessary in order to apply the Morse-Sard theorem for Sobolev functions to each $\Phi_k\in W^{2,2}_{\text{loc}}(\Omega)$ (see the comments below \eqref{202310041719}) and for the validity of the identities \eqref{202309121218} and \eqref{202309121217} in the sense of $L^2(\partial\Omega)$.
\end{rema}

\subsection{A remark on a priori estimates for solutions}The proof of Theorem \ref{theo1} [respectively, Theorem \ref{theo2}] is based on the Leray-Schauder fixed point theorem, and the required a priori estimate for solutions is established under the condition \eqref{theo1_assumption} [respectively, the outflow condition \eqref{outflow_condition}]. Then it is natural to ask if such assumptions are indeed necessary for the solvability of the problem \eqref{NS}-\eqref{slip} when the fluxes of $a_{\ast}$ through each connected component $\Gamma_j$ of $\partial\Omega$ are large. Even though an answer to this question is not yet known, the following example\footnote{The authors learned Example \ref{exam:Hamel} from Mikhail V. Korobkov through personal communication in January 2024.} implies that, in contrast to the problem \eqref{NS} with the no-slip boundary condition \eqref{no_slip} (see \cites{MR3004771,MR3275850}), solutions to the problem \eqref{NS}-\eqref{slip} do not admit a priori estimates in general, and that the assumptions \eqref{theo1_assumption} and \eqref{outflow_condition} are indispensable for deriving the required a priori estimate for solutions.
\begin{exam}[due to Mikhail V. Korobkov and Xiao Ren]\label{exam:Hamel} Let $\Omega=\{x\in\BR^2:1<\abs{x}<2\}$, $\Gamma_0=\{x\in\BR^2:\abs{x}=2\}$, and $\Gamma_1=\{x\in\BR^2:\abs{x}=1\}$. Let $\beta\equiv3/4$ on $\Gamma_0$ and $\beta\equiv0$ on $\Gamma_1$. Suppose that the kinematic viscosity constant $\nu=1$ and $\bff=\bb_{\ast}=\boldsymbol{0}$. Then for the boundary datum $a_{\ast}\in C^{\infty}(\partial\Omega)$ defined by $a_{\ast}\equiv-3/2$ on $\Gamma_0$ and $a_{\ast}\equiv3$ on $\Gamma_1$, the Navier-Stokes problem \eqref{NS}-\eqref{slip} possesses infinitely many explicit solutions $\{\bu_k,p_k\}_{k\in\BR}$ of the form
\be\label{202409241725}
\bu_k(x)=-\frac{3}{\abs{x}^2}x+k\frac{3\abs{x}-2}{\abs{x}^3}x^{\perp},\enskip p_k(x)=-\frac{\abs{\bu_k(x)}^2}{2},\enskip k\in\BR,
\ee
where $(x_1,x_2)^{\perp}=(-x_2,x_1)$. Here it should be noticed that $\norm{\nabla\bu_k}_{L^2(\Omega)}\to\infty$ as $k\to\infty$.
\end{exam}
\begin{rema}\label{rem:Hamel}(i) The explicit solutions given in \eqref{202409241725} are known as the Hamel solutions; see \cite{MR2808162}*{Section X\hspace{-1.2pt}I\hspace{-1.2pt}I.2}, \cite{Hamel1917}*{Section 11}.

(ii) In Example \ref{exam:Hamel}, the condition \eqref{theo1_assumption} is not satisfied. Indeed, since $\kappa(x)=-1/2$ on $\Gamma_0$, it holds that $-2\kappa(x)=1>3/4=\beta/\nu$ on $\Gamma_0$.

(iii) Concerning the boundary datum $a_{\ast}$ in Example \ref{exam:Hamel}, we note that
\be
\int_{\partial\Omega}a_{\ast}\,ds=\int_{\abs{x}=2}-\frac{3}{2}\,ds+\int_{\abs{x}=1}3\,ds=-6\pi+6\pi=0,
\ee
and that the outflow condition \eqref{outflow_condition} is not fulfilled.
\end{rema}

\subsection{Existence results for sufficiently small fluxes} We shall next show that, even if the condition \eqref{theo1_assumption} is not satisfied, the problem \eqref{NS}-\eqref{slip} admits at least one weak solution in arbitrary bounded domains provided that the fluxes $\int_{\Gamma_j} a_{\ast}\,ds$ through $\Gamma_j$, $j=0,\dots,N$, are sufficiently small. In order to state our next result, let us first recall the result due to Kozono-Yanagisawa \cite{MR2550139} concerning the harmonic part of solenoidal extensions of the boundary datum $a_{\ast}$. Denote by $\tilde{V}_{\text{har}}(\Omega)$ the space of the harmonic vector fields defined by
\be
\tilde{V}_{\text{har}}(\Omega)\equiv\{\bh=(h_1,h_2)\in C^{\infty}(\overline{\Omega}):\dv\bh=0, \curl\bh=0\enskip\text{in }\Omega,\,\bh\wedge\bn=0\enskip\text{on }\partial\Omega\},
\ee
where $\curl\bh=\frac{\partial h_1}{\partial x_2}-\frac{\partial h_2}{\partial x_1}$ and $\bh\wedge\bn=h_2n_1-h_1n_2$. It is known that (see \cite{MR2550139}*{Theorem 3.20 (I)}) the space $\tilde{V}_{\text{har}}(\Omega)$ of harmonic vector fields is $N$-dimensional and a basis of $\tilde{V}_{\text{har}}(\Omega)$ is given by $\{\nabla q_1,\dots,\nabla q_N\}$, where $q_{k}$ are solutions of the following Dirichlet boundary value problem for the Laplace equation:
\be\label{eq:Laplace}
\Delta q_{k}=0 \quad \text{in }\Omega\quad\text{and }q_{k}|_{\Gamma_j}=\delta_{jk}\quad \text{for }j=0,\dots,N.
\ee
By applying the Schmidt orthonormalization procedure to the basis $\{\nabla q_1,\dots,\nabla q_N\}$ in $L^2(\Omega)$, we obtain the orthonormal basis $\{\bfpsi_1,\dots,\bfpsi_N\}$ of $\tilde{V}_{\text{har}}(\Omega)$. Then there is an $N\times N$ regular matrix $(\alpha_{ik})_{1\le i,k\le N}$ depending on $\Omega$ such that
\be\label{Schmidt}
\bfpsi_i(x)=\sum_{k=1}^{N}\alpha_{ik}\nabla q_{k}(x),\enskip x\in\Omega,\enskip j=1,\dots,N.
\ee
We shall need the following result.
\begin{prop}[see Lemma \ref{solenoidalextension} and \cite{MR2550139}*{Proposition 1.1}]\label{prop:harmonic} Let $\Omega\subseteq\BR^2$ be a smooth bounded domain defined by \eqref{domain}. Assume that the boundary datum $a_{\ast}\in W^{1/2,2}(\partial\Omega)$ satisfies the compatibility condition \eqref{flux}.
\begin{enumerate}
\item There exists a vector field $\bA\in W^{1,2}(\Omega)$ such that
\be\label{202312151104}
\dv\bA=0\enskip\text{in }\Omega\enskip\text{and }\bA\cdot\bn=a_{\ast}\enskip\text{on }\partial\Omega.
\ee
\item For any vector field $\bA\in W^{1,2}(\Omega)$ satisfying \eqref{202312151104}, there exist $\bh\in \tilde{V}_{\text{har}}(\Omega)$ and $w\in W^{2,2}(\Omega,\BR)$ such that $\bA$ can be represented as
\be\label{202312151105}
\bA=\bh+\nabla^{\perp}w\enskip\text{in }\Omega,
\ee
where $\nabla^{\perp}w=(-\frac{\partial}{\partial x_2}w,\frac{\partial}{\partial x_1}w)$.
\item The harmonic part $\bh$ in \eqref{202312151105} is given explicitly as
\be\label{iden:harmonic}
\bh=\sum_{i,j=1}^{N}\alpha_{ij}\bfpsi_{i}\int_{\Gamma_j}a_{\ast}\,ds=\sum_{k=1}^{N}\nabla q_{k}(x)\sum_{i=1}^{N}\alpha_{ik}\sum_{j=1}^{N}\alpha_{ij}\int_{\Gamma_j}a_{\ast}\,ds.
\ee
Here $(\alpha_{ik})_{1\le i,k \le N}$, $\{\bfpsi_1,\dots,\bfpsi_N\}$ and $\{q_1,\dots,q_N\}$ are the $N\times N$ regular matrix, the orthonormal basis of $\tilde{V}_{\text{har}}(\Omega)$ in the $L^2(\Omega)$-sense, and $N$ harmonic functions appearing in \eqref{eq:Laplace} and \eqref{Schmidt}, respectively.
\end{enumerate}
\end{prop}
Henceforth any vector field $\bA$ satisfying \eqref{202312151104} will be called a solenoidal extension of the boundary datum $a_{\ast}$.
\begin{rema}(i) The regular matrix $(\alpha_{ik})_{1\le i,k\le N}$ in \eqref{Schmidt} and \eqref{iden:harmonic} can be represented by means of the harmonic functions  $\{q_1,\dots,q_N\}$ in \eqref{eq:Laplace}; see \cite{MR2550139}*{Proposition 1.1 (I\hspace{-1.2pt}I)}.

(ii) It follows from \eqref{iden:harmonic} and the above remark that the harmonic part $\bh$ of $\bA$ depends only on the basis $\{\nabla q_1,\dots,\nabla q_N\}$ of $\tilde{V}_{\text{har}}(\Omega)$ and the fluxes $\int_{\Gamma_j}a_{\ast}\,ds$ through $\Gamma_j$, $1\le j\le N$. Thus, in particular, the harmonic part of solenoidal extensions of the boundary datum $a_{\ast}$ is independent of the particular choice of the extension.
\end{rema}
With Proposition \ref{prop:harmonic} in hand, we can prove the following.
\begin{theo}\label{theo4} Let $\Omega$ be a smooth bounded domain defined by \eqref{domain}, and let the friction coefficient $\beta\in C(\partial\Omega)$ be nonnegative. Suppose that the boundary data $a_{\ast}\in W^{1/2,2}(\partial\Omega)$ and $\bb_{\ast}\in W^{-1/2,2}(\partial\Omega)$ satisfy the compatibility conditions \eqref{theo1_comp} and the external force $\bff$ is in $H(\Omega)'$. Assume further that $\beta\not\equiv0$ in the case when the domain $\Omega$ has a circular symmetry around some point.\footnote{Define $\Omega_{\theta}^{g}$ as the image of $\Omega$ by the rotation of angle $\theta$ around the point $g\in\BR^2$. We say that $\Omega$ has a circular symmetry around the point $g$ if the equality $\Omega=\cup_{0\le\theta<2\pi}\Omega_{\theta}^{g}$ holds.\label{def:circular_symmetry}} Let $\bh$ be the harmonic part of the solenoidal extension of $a_{\ast}$ given by \eqref{iden:harmonic}. Then, if the inequality
\be\label{theo4_assumption}
\sup_{\bz\in E(\Omega), \bz\neq\boldsymbol{0}}\frac{-\displaystyle\int_{\Omega}\bh(x)\cdot(\curl\bz)(x)\left(-z_2(x),z_1(x)\right)\,dx}{\norm{\bz}_{W^{1,2}(\Omega)}^2}<\frac{\nu}{2}K(\Omega, \frac{2\beta}{\nu})^{-1}
\ee
holds, the Navier-Stokes problem \eqref{NS}-\eqref{slip} admits at least one weak solution $\bu\in W^{1,2}(\Omega)$. Here
\be
E(\Omega)\equiv\{\bz=\bz(x)=(z_1(x),z_2(x))\in J(\Omega):\int_{\Omega}(\bz\cdot\nabla)\bz\cdot\bfxi\,dx=0\enskip\text{for all }\bfxi\in C^{\infty}_{0,\sigma}(\Omega)\}
\ee
is the set of weak solutions to the steady Euler equation with homogeneous boundary conditions and $K(\Omega,\frac{2\beta}{\nu})$ is the best constant of the following Korn-type inequality:
\be\label{theo4_Korn}
\norm{\bv}_{W^{1,2}(\Omega)}^2\le K(\Omega,\frac{2\beta}{\nu})\left(\int_{\Omega}S(\bv):S(\bv)\,dx+\int_{\partial\Omega}\frac{2\beta}{\nu}\bv_{\tau}\cdot\bv_{\tau}\,ds\right)\enskip\text{for all }\bv\in H(\Omega).
\ee
\end{theo}
As an immediate consequence of Theorem \ref{theo4}, we can prove that the problem \eqref{NS}-\eqref{slip} has a solution provided that the harmonic part of the solenoidal extension of the boundary datum $a_{\ast}$ is sufficiently small. More precisely, we have
\begin{cor}\label{cor:harmonic} Let $\Omega$, $\bff$, $\beta$, $a_{\ast}$, $\bb_{\ast}$ and $\bh$ be as in Theorem $\ref{theo4}$. Let $2<q<\infty$. If the inequality
\be\label{202312201219}
\sqrt{2}C_r\norm{\bh}_{L^q(\Omega)}<\frac{\nu}{2}K(\Omega,\frac{2\beta}{\nu})^{-1}
\ee
holds for $r=\frac{2q}{q-2}$, then there exists at least one weak solution $\bu\in W^{1,2}(\Omega)$ to the problem \eqref{NS}-\eqref{slip}. Here $C_r$ $(2<r<\infty)$ is the best constant of the Sobolev inequality
\be\label{theo4_Sobolev}
\norm{\bv}_{L^r(\Omega)}\le C_r\norm{\bv}_{W^{1,2}(\Omega)}\enskip\text{for all }\bv\in W^{1,2}(\Omega).
\ee
\end{cor}
\begin{rema}\label{rem:small_flux} In view of \eqref{iden:harmonic}, it is not difficult to show that the condition \eqref{202312201219} can be replaced by
\be
\sum_{j=1}^{N}c_j\left|\int_{\Gamma_j}a_{\ast}\,ds\right|<\frac{\nu}{2}K(\Omega,\frac{2\beta}{\nu})^{-1},
\ee
where $c_j$ are certain constants depending only on $\Omega$ and the exponent $q$. It should also be noted that, if the boundary datum $a_{\ast}$ satisfies the restriction \eqref{zero_flux_slip}, then $\bh\equiv\boldsymbol{0}$, and hence the inequality \eqref{theo4_assumption} is necessarily satisfied regardless of the shape of $\Omega$ and the values of $\nu$ and $\beta$. 
 \end{rema}
 With the aid of Corollary \ref{cor:harmonic} and Remark \ref{rem:small_flux}, Theorem \ref{theo2} can be strengthened to the following one:
\begin{cor} Let $\Omega, \bff, \beta, a_{\ast}$ and $\bb_{\ast}$  be as in Theorem $\ref{theo2}$. Suppose that the inner hole $\Omega_1$ of the domain $\Omega$ is convex. Then there exists a positive constant $\CF_{0}=\CF_{0}(\Omega,\nu,\beta)>0$ depending only on $\Omega$, $\nu$ and $\beta$ such that if
\be
\int_{\Gamma_{0}}a_{\ast}\,ds\in(-\CF_{0},+\infty),
\ee
the problem \eqref{NS}-\eqref{slip} admits at least one weak solution.
\end{cor}

\subsection{An existence result in symmetric domains}Finally, we shall show that, under certain symmetry assumptions on the domain $\Omega$ and the given data $\bff$, $\beta$, $a_{\ast}$, and $\bb_{\ast}$, the problem \eqref{NS}-\eqref{slip} admits at least one symmetric solution. Following Amick \cite{MR763943}, we give two definitions.
\begin{defi}\label{def:admissible}A smooth bounded domain $\Omega\subseteq\BR^2$ defined by \eqref{domain} with $N\ge1$ is said to be admissible if 
\begin{enumerate}
\item $\Omega$ is symmetric with respect to the $x_1$-axis; 
\item each component $\Gamma_j$ intersects the $x_1$-axis.
\end{enumerate}
\end{defi}
\begin{defi}A vector field $\bh=(h_1,h_2)$ mapping $\Omega$ or $\partial\Omega$ into $\BR^2$ is said to be symmetric (with respect to the $x_1$-axis) if $h_1$ is an even function of $x_2$, while $h_2$ is an odd function of $x_2$. Likewise, a scalar function $h\colon\partial\Omega\to\BR$ is said to be symmetric (with respect to the $x_1$-axis) if the equality $h(x_1,x_2)=h(x_1,-x_2)$ holds for all $(x_1,x_2)\in\partial\Omega$.
\end{defi}
Our theorem on existence of symmetric solutions now reads: 
\begin{theo}\label{theo3} Let $\Omega\subseteq\BR^2$ be an admissible domain, and let the friction coefficient $\beta^{S}\in C(\partial\Omega)$ be nonnegative and symmetric with respect to the $x_1$-axis. Suppose that $\bff^{S}\in L^2(\Omega)$, $a_{\ast}^{S}\in W^{1/2,2}(\partial\Omega)$ and $\bb_{\ast}^{S}\in L^2(\partial\Omega)$ are symmetric with respect to the $x_1$-axis and the boundary data $a_{\ast}^{S}$ and $\bb_{\ast}^{S}$ enjoy the compatibility conditions
\be\label{flux_theo3}
\int_{\partial\Omega} a_{\ast}^{S}\,ds=\sum_{j=0}^{N} \int_{\Gamma_{j}} a_{\ast}^{S} \, ds=0\enskip\text{and}\enskip\bb_{\ast}^{S}\cdot\bn=0\enskip\text{on}\enskip\partial\Omega.
\ee
Then the Navier-Stokes problem \eqref{NS}-\eqref{slip} admits at least one symmetric solution $\bu^{S}\in W^{1,2}(\Omega)$. 
\end{theo}
\begin{rema} 
In Theorem \ref{theo3}, the friction coefficient $\beta^{S}$ is allowed to be identically equal to zero even in the case when the domain $\Omega$ is an annulus.
\end{rema}

\section{Preliminaries}\label{Prelimi}
For a subset $A$ of $\BR^n$, we write $\SL^n(A)$ for its $n$-dimensional Lebesgue outer measure. Denote by $\SH^1$ and $\SH^1_{\infty}$ the $1$-dimensional Hausdorff measure and the $1$-dimensional Hausdorff content, respectively. Recall that for any subset $A$ of $\BR^n$ we have by definition
\be
\SH^1(A)=\lim_{\delta\to0}\SH^1_{\delta}(A)=\sup_{\delta>0}\SH^1_{\delta}(A),
\ee
where for each $0<\delta\le\infty$,
\be
\SH^1_{\delta}\equiv\inf\Big\{\sum\limits_{j=1}^{\infty} \mathrm{diam}\,C_j:A\subseteq\bigcup\limits_{j=1}^{\infty}C_j,\,\mathrm{diam}\,C_j\le\delta\Big\}.
\ee

\subsection{Properties of Sobolev functions}
Working with locally integrable functions, we shall always assume that their precise representatives are chosen. For $w\in L^1_{\mathrm{loc}}(\Omega)$ the precise representative $w^{\ast}$ is defined for all $x\in\Omega$ by
\begin{align} 
w^{\ast}(x)\equiv\left\{
\ba
\lim_{r\to0}\dashint_{B(x,r)}&w(y)\,dy\enskip&&\text{if the limit exists and is finite,}\\
&0\enskip&&\text{otherwise,}
\ea 
\right. 
\end{align}
where 
\be
\dashint_{B(x,r)}w(y)\,dy\equiv\frac{1}{\SL^n(B(x,r))}\int_{B(x,r)}w(y)\,dy,
\ee
and $B(x,r)\equiv\{y\in \BR^n:\abs{y-x}<r\}$ is the open ball of radius $r$ centered at $x$. The following lemma describes fine properties of Sobolev functions.
\begin{lemm}[see \cite{MR3409135}*{Theorems 4.19 and 4.21}]\label{lem:fine} Suppose that $w\in W^{1,q}(\BR^2)$, $1\le q\le\infty$.
\begin{enumerate}
\item There exists a Borel set $A_w\subseteq\BR^2$ such that $\SH^1(A_w)=0$ and
\be
\lim_{r\to0}\dashint_{B(x,r)}w(y)\,dy
\ee
exists for each $x\in\BR^2\setminus A_w$.
\item In addition,
\be
\lim_{r\to0}\dashint_{B(x,r)}\abs{w(y)-w^{\ast}(x)}^2\,dy=0
\ee
for each $x\in\BR^2\setminus A_w$.
\item For each $\epsilon>0$, there exists an open set $V\subseteq\BR^2$ with $\SH^1_{\infty}(V)\le\epsilon$ and $A_w\subseteq V$ such that the function $w^{\ast}$ is continuous on $\BR^2\setminus V$.
\item For any unit vector $\bl \in\partial B(0,1)$ and almost all straight lines $L$ parallel to $\bl$, the restriction $w^{\ast}|_{L}$ is an absolutely continuous function of one variable.
\end{enumerate}
\end{lemm}
Henceforth we omit special notation for the precise representative writing simply $w^{\ast}=w$. Let us next recall some differentiability properties of Sobolev functions.
\begin{lemm}[see \cite{MR1029687}*{Proposition 1}]\label{lem:Dorr} If $\psi\in W^{2,1}(\BR^2)$, then $\psi$ is a continuous function and there exists a set $A_{\psi}\subseteq\BR^2$ with $\SH^1(A_{\psi})=0$ such that $\psi$ is differentiable in the classical sense at each $x\in\BR^2\setminus A_{\psi}$. Furthermore, its classical derivative at such a point $x$ coincides with the value of $\nabla\psi(x)=\lim\limits_{r\to0}\dashint_{B(x,r)}\nabla\psi(y)\,dy$, and $\lim\limits_{r\to0}\dashint_{B(x,r)}\abs{\nabla\psi(y)-\nabla\psi(x)}^2\,dy=0$.
\end{lemm}
\begin{rema}\label{rem:extension_trace} By the Sobolev extension theorem (see, e.g., \cite{MR3409135}*{Theorems 4.7}) the analogs of Lemmas \ref{lem:fine} and \ref{lem:Dorr} are true for functions $w\in W^{1,q}(\Omega)$, where $\Omega\subseteq\BR^2$ is a bounded Lipschitz domain. By the trace theorem, for each $w\in W^{1,1}(\Omega)$ the value of $w(x)$ is well-defined for $\SH^{1}$-almost all $x\in\partial\Omega$. Therefore, in what follows, we shall assume that every function $w\in W^{1,1}(\Omega)$ is defined on $\overline{\Omega}$.
\end{rema}
A subset of a topological space is called an arc if it is homeomorphic to the unit interval $[0,1]$. The following theorem due to Bourgain, Korobkov and Kristensen \cite{MR3010119} describes Luzin $N$ and Morse-Sard properties for functions from the Sobolev space $W^{2,1}(\BR^2)$. 
\begin{theo}\label{thm:Morse-Sard} Let $\Omega\subseteq\BR^2$ be a bounded domain with Lipschitz boundary and let $\psi\in W^{2,1}(\Omega, \BR)$. Then
\begin{enumerate}
\item $\SH^1\big(\{\psi(x):x\in\overline{\Omega}\setminus A_{\psi}\enskip\text{and}\enskip\nabla\psi(x)=\boldsymbol{0}\}\big)=0$.
\item For every $\epsilon>0$ there exists $\delta>0$ such that for any set $V\subseteq\overline{\Omega}$ with $\SH^1_{\infty}(V)<\delta$ the inequality $\SH^1(\psi(V))<\epsilon$ holds. In particular, $\SH^1(\psi(V))=0$ whenever $\SH^1(V)=0$.
\item For $\SH^1$-almost all $y\in\psi(\overline{\Omega})\subseteq\BR$ the preimage $\psi^{-1}(y)$ is a finite disjoint union of $C^1$-curves $S_j$, $j=1,2,\dots,N(y)$. Each $S_j$ is either a cycle in $\Omega$ $($i.e., $S_j\subseteq\Omega$ is homeomorphic to the unit circle $\BS^1)$ or a simple arc with endpoints on $\partial\Omega$ $($in this case $S_j$ is transversal to $\partial\Omega)$.
\end{enumerate}
\end{theo}

\subsection{The curvature of the boundary and the Weingarten map}\label{subsec:Weingarten}
We recall that a rigid motion $T\colon\BR^2\to\BR^2$ is an affine map given by $T(x)=Rx+b$, $x\in\BR^2$, where $R$ is a $2\times2$ rotation matrix and $b\in\BR^2$ is a constant vector. Let $\Omega\subseteq\BR^2$ be a bounded domain with $C^2$-smooth boundary $\partial\Omega$ and let $\bn=\bn(x)=(n_1(x),n_2(x))$ be the unit outer normal vector at the point $x\in\partial\Omega$. Then for each $x_0\in\partial\Omega$ there exist a rigid motion $T\colon\BR^2\to\BR^2$ with $T(x_0)=0$, a $C^2$ function $\zeta\colon\BR\to\BR$ with $\zeta(0)=0$ and $\frac{d\zeta}{dy_1}(0)=0$, and $r>0$ such that, setting $y=T(x)$, the $y_2$ coordinate axis lies in the direction $\bn(x_0)$ and we have
\be
T(\Omega\cap B(x_0,r))=\{y=(y_1,y_2)\in B(0,r):y_2<\zeta(y_1)\}.
\ee
The coordinates $y$ are called local coordinates. The value $\frac{d^2\zeta}{dy_1^2}(0)$ is called the curvature of $\partial\Omega$ at $x_0$ (in the direction of $\bn(x_0)$) and is denoted by $\kappa(x_0)$.

The distance from $x\in\BR^2$ to a subset $S\subseteq\BR^2$ is indicated by $\dist(x,S)$, where $\dist(x,S)\equiv\inf_{y\in S}\abs{x-y}$. A vector field $N\in C^{1}(U_a,\BR^2)$, where $U_a\equiv\{x\in\BR^2: \dist(x,\partial\Omega)<a\}$, $a>0$, will be called an extended unit outer normal to $\partial\Omega$ if (i) $N=\bn$ on $\partial\Omega$; (ii) $\abs{N}=1$ in $U_a$; (iii) $(\nabla N)N=\boldsymbol{0}$ on $\partial\Omega$.
\begin{exam} Let $d^{\ast}=d^{\ast}(x)$ be the signed distance from $\partial\Omega$ defined by 
\be
\ba
d^{\ast}(x) \equiv \left\{
\begin{array}{ll}
-\dist(x,\partial\Omega)&\text{if}\quad x\in\Omega,\\
\dist(x,\partial\Omega)&\text{if}\quad x\in\BR^2\setminus\Omega.
\end{array}
\right.
\ea
\ee
If $\Omega\subseteq\BR^2$ is a bounded domain of class $C^k$, $k\ge2$, there exists a positive constant $a$ depending on $\Omega$ such that $d^{\ast}\in C^{k}(U_a)$ and for each $x\in U_a$ there is a unique point $\pi(x)\in\partial\Omega$ satisfying
\be
x=\pi(x)+d^{\ast}(x)\bn(\pi(x)),\quad\overline{B_{\abs{d^{\ast}(x)}}(x)}\cap\partial\Omega=\{\pi(x)\}\quad\text{and}\quad\nabla d^{\ast}(x)=\bn(\pi(x)).
\ee
Moreover, the inequality $\abs{\kappa(x)}<1/a$ holds for all $x\in\partial\Omega$; see  \cite{MR1814364}*{Section 14.6}, \cite{MR3524106}*{Chapter 2, Section 2.3}. Then the vector field $\nabla d^{\ast}\in C^{k-1}(U_a)$ gives an extended unit outer normal to $\partial\Omega$.   
\end{exam}

\begin{prop}[see \cite{MR2242962}*{Propositions 3.1 and 3.4}]\label{prop:Weingarten} Let $\Omega$ be a bounded domain with $C^2$-smooth boundary $\partial\Omega$, and let $N\in C^{1}(U_a,\BR^2)$ be an extended unit outer normal to $\partial\Omega$. Then the $2\times2$ matrix-valued function
\be
W(x)\equiv-\nabla N(x)=-(\partial_j N_i)_{1\le i,j\le2},\quad x\in U_a
\ee
satisfies the following properties:
\begin{enumerate}
\item $W^{\top}N=\boldsymbol{0}$ in $U_a$;
\item The value of $W$ on $\partial\Omega$ depends only on $\partial\Omega$ and not on the choice of an extension $N$\footnote{The extension $N$ of the unit outer normal $\bn$ into a neighborhood of $\partial\Omega$ is not unique.};
\item $W=W^\top$ on $\partial\Omega$;
\item For any vector field $\bu\colon\partial\Omega\to\BR^2$, $W\bu$ is a tangential vector field, i.e., $(W\bu)\cdot\bn=0$ on $\partial\Omega$;
\item The eigenvalues of the matrix $W(x)$ at $x\in\partial\Omega$ consist of $0$ and the curvature $\kappa(x)$ of the boundary at $x$ in the direction of $\bn(x)$.
\end{enumerate}
\end{prop}
The function $W$ in the above proposition is called the Weingarten map of $\partial\Omega$ (in the direction of $\bn$).

For a scalar function $f\in C^1(\partial\Omega, \BR)$, we define the tangential gradient $\nabla_{\tau}f(x)$ of $f$ at $x\in\partial\Omega$ by
\be
\nabla_{\tau}f(x)\equiv\nabla\tilde{f}(x)-\bn(x)(\bn(x)\cdot\nabla\tilde{f}(x))=(I-\bn(x)\otimes\bn(x))\nabla\tilde{f}(x),\enskip x\in\partial\Omega,
\ee
where $I$ is the $2\times2$ identity matrix and $\tilde{f}$ is a $C^1$-extension of $f$ into some neighborhood of $\partial\Omega$ with $\tilde{f}|_{\partial\Omega}=f$. The following identity describes the relationship between the boundary conditions \eqref{slip} and \eqref{Hodge}.
\begin{lemm}[see \cite{MR2492825}*{Section 2}]\label{lem_def-curl} Let $\Omega\subseteq\BR^2$ be a bounded domain of class $C^3$, and let $\bu\in W^{2,q}(\Omega)$, $1\le q<\infty$. Then the following identity holds
\be\label{iden:def-curl}
[\gamma(S(\bu))\bn]_{\tau}=\gamma(\curl\bu)(n_2,-n_1)^\top+2\nabla_{\tau}(\gamma(\bu)\cdot\bn)+2W^\top\gamma(\bu)\quad\text{in}\enskip L^q(\partial\Omega).
\ee
Here $\gamma\colon W^{1,q}(\Omega)\to L^q(\partial\Omega)$ is the trace operator.
\end{lemm}

\subsection{The Stokes problem under the slip boundary conditions}
In what follows, we shall denote by $c$ a generic constant, which may change from line to line. If necessary, we shall denote by $c(\ast,\dots,\ast)$ the constant depending only on the quantities appearing in the parenthesis. 

The objective of this subsection is to collect some preliminary results concerning the Stokes problem under the slip boundary conditions, which will be needed in Subsection \ref{Leray_arg}. The results given below are not considered as sharp, but they are sufficient for our purposes. For a more detailed study on the Stokes problem under the slip boundary conditions, we refer the reader to \cites{MR2098066,MR4429780,SoSc73} for the $L^2$-theory and to \cites{MR4231512,MR3145765} for the $L^q$-theory, $1<q<\infty$.

Throughout this subsection, we shall assume that $\Omega\subseteq\BR^2$ is a smooth bounded domain. Consider in $\Omega$ the Stokes equations with the nonhomogeneous slip boundary conditions 
\begin{align} \label{S}
\left\{
\ba
-\nu \Delta \bu+ \nabla p & =\bff \quad & & \text{in} \enskip \Omega,\\
\dv \bu & =0 \quad & & \text{in} \enskip \Omega,\\
\bu \cdot\bn & = a_{\ast} \quad & & \text{on} \enskip \partial \Omega,\\
[T(\bu,p)\bn]_{\tau}+\beta \bu_{\tau} & =\bb_{\ast}  \quad & &\text{on} \enskip \partial \Omega.
\ea 
\right. 
\end{align}
We shall suppose that the friction coefficient $\beta$ is a nonnegative scalar-valued function defined on $\partial\Omega$ and the boundary data $a_{\ast}$ and $\bb_{\ast}$ satisfy the compatibility conditions
\be\label{Stokes_comp}
\int_{\partial\Omega}a_{\ast}\,ds=0\enskip\text{and }\bb_{\ast}\cdot\bn=0\quad\text{on}\enskip\partial\Omega.
\ee 
In order to introduce the notion of $q$-weak solutions to the Stokes problem \eqref{S}, let us define two closed subspaces $H_q(\Omega)$ and $J_q(\Omega)$ of $W^{1,q}(\Omega)$, $1<q<\infty$, by
\be\label{202401011200}
H_q(\Omega)\equiv\{\bu\in W^{1,q}(\Omega):\gamma(\bu)\cdot\bn=0\enskip\text{on}\enskip\partial\Omega\},
\ee
\be\label{202401011201}
J_q(\Omega)\equiv\{\bu\in H_q(\Omega):\dv\bu=0\enskip\text{in}\enskip\Omega\}. 
\ee
For $q=2$ we simply write $H$ and $J$ in place of $H_2$ and $J_2$.
\begin{defi}
Let $1<q<\infty$ and $q'=\frac{q}{q-1}$. Let $\beta\in C(\partial\Omega,\BR)$ be nonnegative, and let $\bff\in H_{q'}(\Omega)'$. Suppose that $a_{\ast}\in W^{1-1/q,q}(\partial\Omega)$ and $\bb_{\ast}\in W^{-1/q,q}(\partial\Omega)$ satisfy the compatibility conditions \eqref{Stokes_comp}. Then a vector field $\bu:\Omega\to\BR^2$ is called a q-weak (or q-generalized) solution to the Stokes problem \eqref{S} if 
\begin{enumerate}
\item $\bu\in W^{1,q}(\Omega)$; 
\item $\bu$ is divergence-free in $\Omega$;
\item $\bu$ satisfies the boundary condition $\eqref{S}_{3}$ in the trace sense, namely,
\be
\gamma(\bu)\cdot\bn=a_{\ast}\quad\text{in}\enskip W^{1-1/q,q}(\partial\Omega);
\ee
\item $\bu$ satisfies the identity
\be\label{def:weakS}
\frac{\nu}{2}\int_{\Omega} S(\bu): S(\bfvarphi) \,dx+\int_{\partial\Omega} \beta\bu_{\tau}\cdot\bfvarphi \,ds-\langle {\bb_{\ast}},\bfvarphi \rangle_{\partial\Omega}=\langle \bff,\bfvarphi \rangle_{\Omega}
\ee
for all $\bfvarphi\in J_{q'}(\Omega)$, $1/q+1/q'=1$.
\end{enumerate}
If $q=2$, $\bu$ will be called a weak (or generalized) solution.
\end{defi}
The basis for this definition is a weak (or variational) formulation of the Stokes problem \eqref{S}, which could be performed in the same way as we did for the Navier-Stokes problem \eqref{NS}-\eqref{slip} in Subsection \ref{subsec:variational_formulation}. For a q-weak solution $\bu$ to the Stokes problem \eqref{S}, one can associate a suitable pressure $p$. 
\begin{lemm}[see \cite{MR2808162}*{Theorem I\hspace{-1.2pt}I\hspace{-1.2pt}I.5.3 and Lemma I\hspace{-1.2pt}V.1.1}]\label{pressure_q} Let $\beta\in C(\partial\Omega)$ be nonnegative and let $1<q<\infty$, $q'=\frac{q}{q-1}$. Suppose that $\bu$ is a $q$-weak solution to the Stokes problem \eqref{S} corresponding to $\bff\in H_{q'}(\Omega)'$, $a_{\ast}\in W^{1-1/q,q}(\partial\Omega)$ and $\bb_{\ast}\in W^{-1/q,q}(\partial\Omega)$. Then there exists a scalar function $p\in L^{q}(\Omega)$ with $\int_{\Omega} p(x)\,dx=0$ such that the integral identity
\be
\frac{\nu}{2}\int_{\Omega} S(\bu): S(\bfvarphi) \,dx+\int_{\partial\Omega} \beta\bu_{\tau}\cdot\bfvarphi \,ds-\langle {\bb_{\ast}},\bfvarphi \rangle_{\partial\Omega}-\langle \bff,\bfvarphi \rangle_{\Omega}=\int_{\Omega} p\,\dv\bfvarphi \,dx 
\ee
holds for all $\bfvarphi\in H_{q'}(\Omega)$. Moreover, $p$ admits the estimate
\be\label{202409300718}
\norm{p}_{L^q(\Omega)}\le c(\norm{\bu}_{W^{1,q}(\Omega)}+\norm{\bff}_{H_{q'}(\Omega)'}+\norm{\bb_{\ast}}_{W^{-1/q,q}(\partial\Omega)})
\ee
with $c=c(q,\Omega,\nu,\beta)$.
\end{lemm}
We begin to consider existence and uniqueness of weak solutions to the problem \eqref{S}. In order to deal with the nonhomogeneous boundary datum $a_{\ast}$ in \eqref{S}, we extend $a_{\ast}$ into $\Omega$ by a solenoidal vector field $\bA$ such that $\bA\cdot\bn=a_{\ast}$ on $\partial\Omega$.
\begin{lemm}
\label{solenoidalextension}
Let $m\ge 1$. If $a_{\ast} \in {W}^{m-1/2,2}(\partial\Omega)$ satisfies
\be
\int_{\partial \Omega} a_{\ast} \, ds=0,
\ee
then there exists a vector field $\bA \in {W}^{m,2}(\Omega)$ such that $\dv \bA=0$ in $\Omega$, $\bA\cdot\bn=a_{\ast}$ on $\partial\Omega$ and $\norm{\bA}_{W^{m,2}(\Omega)} \le c\norm{a_{\ast}}_{{W}^{m-1/2,2}(\partial \Omega)}$.
\end{lemm}
\begin{proof}
Consider the Neumann problem
\begin{align} \label{Neumann}
\left\{
\ba
\Delta q & =0 \quad & & \text{in} \enskip \Omega,\\
\frac{\partial q}{\partial n} & =a_{\ast} & & \text{on} \enskip \partial \Omega.
\ea 
\right. 
\end{align}
Under our assumption, the Neumann problem \eqref{Neumann} has a solution $q\in W^{m+1,2}(\Omega)$ which is unique up to an additive constant. Furthermore, the following estimate holds
\be
\norm{\nabla q}_{W^{m,2}(\Omega)}\le c\norm{a_{\ast}}_{W^{m-1/2,2}(\partial\Omega)},
\ee
with $c=c(\Omega,m)>0$. The desired solenoidal vector field $\bA\in W^{m,2}(\Omega)$ is then given by $\bA=\nabla q$.
\end{proof}

We look for a weak solution to the problem \eqref{S} of the form $\bu=\bw+\bA$, where $\bw\in J(\Omega)$ satisfies the integral identity
\be\label{}
\frac{\nu}{2} \int_{\Omega} S(\bw):S(\bfvarphi) \, dx+\int_{\partial \Omega} \beta \bw_{\tau} \cdot \bfvarphi_{\tau} \, ds=-\frac{\nu}{2} \int_{\Omega} S(\bA):S(\bfvarphi) \, dx-\int_{\partial \Omega} \beta \bA_{\tau} \cdot \bfvarphi \, ds+\langle \bff,\bfvarphi \rangle_{\Omega}+\langle \bb_{\ast},\bfvarphi \rangle_{\partial\Omega}
\ee
for all $\bfvarphi \in J(\Omega)$. Here $\bA\in W^{1,2}(\Omega)$ is a solenoidal extension of $a_{\ast}\in W^{1/2,2}(\partial\Omega)$ from Lemma \ref{solenoidalextension}. The following Korn-type inequality provides the appropriate scalar product on the space $J(\Omega)$ for test functions.
\begin{lemm}[see \cite{MR4231512}*{Proposition 3.13}]\label{lem:Korn}
Let $\alpha\in C(\partial\Omega)$ be a nonnegative function such that $\alpha\not\equiv0$. Then there exists a constant $c=c(\Omega,\alpha)>0$ such that
\be\label{ineq:Korn}
\norm{\bw}_{W^{1,2}(\Omega)}\le c\left(\norm{S(\bw)}_{L^2(\Omega)}+\norm{\sqrt{\alpha}\bw_{\tau}}_{L^2(\partial\Omega)}\right)\quad\forall\bw\in H(\Omega).
\ee
\end{lemm}
\begin{rema}\label{rem:Korn}
(i) In general
\be
\bu\in W^{1,2}(\Omega),\enskip S(\bu)=O\enskip\text{(zero matrix)}\Leftrightarrow\bu\in\SR(\Omega),
\ee
where $\SR(\Omega)\equiv\{\bu:\bu(x)=\boldsymbol{a}+b(-x_2,x_1), \boldsymbol{a}=(a_1,a_2)\in\BR^2, b\in\BR\}$ is the set of rigid displacements. Suppose that $\bu_{0}=\boldsymbol{a}+b(-x_2,x_1)$ is not identically equal to zero. Then $\bu_{0}$ belongs to the space $H(\Omega)$, that is, $\bu_{0}\cdot\bn=0$ on $\partial\Omega$, only if $b\neq0$ and $\Omega$ has a circular symmetry around the point $(-\frac{a_2}{b},\frac{a_1}{b})$.\footnote{See footnote \ref{def:circular_symmetry}.} In such a situation, $\bu_{0}\in J(\Omega)\subseteq H(\Omega)$ and
\be
H(\Omega)\cap\SR(\Omega)=\mathrm{Span}\{\bu_{0}\}.
\ee

(ii) Let $V$ be a closed subspace of $H(\Omega)$ such that $V\cap\SR(\Omega)=\{\boldsymbol{0}\}$. Then there exists a constant $c=c(\Omega)>0$ such that the inequality
\be\label{ineq:Korn2}
\norm{\bu}_{W^{1,2}(\Omega)}\le c\norm{S(\bu)}_{L^2(\Omega)}
\ee
holds for any $\bu\in V$; see \cite{MR1195131}*{Chapter I, Theorem 2.5 and Corollary 2.6}, \cite{SoSc73}*{Lemma 4}. In particular, if $\Omega$ does not have a circular symmetry around any points, then $H(\Omega)\cap\SR(\Omega)=\{\boldsymbol{0}\}$, and hence the inequality \eqref{ineq:Korn2} holds for any $\bu\in H(\Omega)$; see \cite{MR3145765}*{Lemma 3.3}, \cite{MR1932965}*{Theorem 3}.
\end{rema}
\begin{rema}\label{rem:case_symmetry} It is clear from the identity \eqref{def:weakS} that in the case when the friction coefficient $\beta\equiv0$ and $H_{q'}(\Omega)\cap\SR(\Omega)=H(\Omega)\cap\SR(\Omega)=\mathrm{Span}\{\bu_{0}\}$ for some $\bu_{0}\neq\boldsymbol{0}$\footnote{The spaces $H_{q'}(\Omega)\cap\SR(\Omega)$, $1<q'<\infty$, do not depend on the exponent $q'$.\label{foot:kernel}} (the domain $\Omega$ has a circular symmetry around some point), an external force $\bff$ and a boundary datum $\bb_{\ast}$ should satisfy the condition
\be\label{Stokes_comp1}
\langle \bff,\bu_{0} \rangle_{\Omega}+\langle \bb_{\ast},\bu_{0} \rangle_{\partial\Omega}=0.
\ee

\end{rema}
\vspace{0.3cm}
Concerning the domain $\Omega$ and the friction coefficient $\beta$, we suppose that either of the following conditions holds:\footnote{See Solonnikov-\v{S}\v{c}adilov \cite{SoSc73} for the case when $\beta\equiv0$ and the domain $\Omega$ has a circular symmetry around some point.}
\begin{enumerate}
\item[$(\mathrm{a})$] $\beta\not\equiv0$;
\item[$(\mathrm{b})$] $\beta\equiv0$ and the domain $\Omega$ does not possess a circular symmetry around any points.
\end{enumerate}
Then, according to Lemma \ref{lem:Korn} and Remark \ref{rem:Korn} (ii), $J(\Omega)$ is a Hilbert space with respect to the scalar product
\be\label{scalar_product}
(\bw,\bfvarphi)_{J(\Omega)}\equiv\frac{\nu}{2} \int_{\Omega} S(\bw):S(\bfvarphi) \, dx+\int_{\partial \Omega} \beta \bw_{\tau} \cdot \bfvarphi_{\tau} \, ds,
\ee 
where $\nu>0$ is the kinematic viscosity constant. Also, the norms $\norm{\bw}_{J(\Omega)}\equiv(\bw,\bw)^{1/2}_{J(\Omega)}$ and $\norm{\bw}_{W^{1,2}(\Omega)}$ are equivalent, namely, there exists a constant $c=c(\Omega,\nu,\beta)>0$ such that
\be
c^{-1}\norm{\bw}_{W^{1,2}(\Omega)}\le\norm{\bw}_{J(\Omega)}\le c\norm{\bw}_{W^{1,2}(\Omega)}\quad\text{for all}\quad\bw\in J(\Omega).
\ee
By means of standard arguments (see, e.g., \cite{MR2808162}*{Theorem I\hspace{-1.2pt}V.1.1}, \cite{MR254401}*{Chapter 2, Theorems 1 and 2}), we can prove the following existence and uniqueness theorem. See also \cite{SoSc73}*{Theorem $1'$}.
\begin{theo}\label{thm:exist_unique_S}
\noindent\begin{enumerate}\setlength{\parskip}{0cm}\setlength{\itemsep}{0cm}
\item Let the friction coefficient $\beta\in C(\partial\Omega)$ be nonnegative. Suppose that the external force $\bff$ is in $H(\Omega)'$, and the boundary data $a_{\ast}\in W^{1/2,2}(\partial\Omega)$ and $\bb_{\ast}\in W^{-1/2,2}(\partial\Omega)$ satisfy the compatibility conditions \eqref{Stokes_comp}. Assume further that $\beta\not\equiv0$ in the case when the domain $\Omega$ possesses a circular symmetry around some point. Then there exists a unique weak solution $\bu\in W^{1,2}(\Omega)$ to the Stokes problem \eqref{S}. Moreover, if we denote by $p$ the pressure associated with $\bu$ by Lemma $\ref{pressure_q}$, the pair $(\bu,p)$ satisfies the inequality
\be\label{202409281612}
\norm{\bu}_{W^{1,2}(\Omega)}+\norm{p}_{L^2(\Omega)}\le c\Bigl(\norm{\bff}_{H(\Omega)'}+\norm{a_{\ast}}_{W^{1/2,2}(\partial\Omega)}+\norm{\bb_{\ast}}_{W^{-1/2,2}(\partial\Omega)}\Bigr)
\ee
with $c=c(\Omega,\nu,\beta)$.
\item Let the friction coefficient $\beta$ be identically equal to zero. Suppose that the domain $\Omega$ has a circular symmetry around some point and $H(\Omega)\cap\SR(\Omega)=\mathrm{Span}\{\bu_{0}\}$, where $\bu_{0}$ is a nonzero rigid displacement. Then for arbitrary $\bff\in H(\Omega)'$, $a_{\ast}\in W^{1/2,2}(\partial\Omega)$, and $\bb_{\ast}\in W^{-1/2,2}(\partial\Omega)$ satisfying the compatibility conditions \eqref{Stokes_comp} and \eqref{Stokes_comp1}, the Stokes problem \eqref{S} admits a unique weak solution $\bu$ in the class 
\be\label{202410010838}
\{\bu\in W^{1,2}(\Omega):(\bu,\bu_{0})_{L^2(\Omega)}=0\}. 
\ee
Furthermore, if we denote by $p$ the pressure associated with $\bu$ by Lemma $\ref{pressure_q}$, the pair $(\bu,p)$ satisfies the estimate \eqref{202409281612} with $c=c(\Omega,\nu,\bu_0)>0$.
\end{enumerate}
\end{theo}
\begin{rema}\label{rem:Unique_S} In the case when $\beta\equiv0$ and $H(\Omega)\cap\SR(\Omega)=\mathrm{Span}\{\bu_{0}\}$ for some nonzero rigid displacement $\bu_{0}$, the Stokes problem \eqref{S} admits infinitely many weak solutions in the class $W^{1,2}(\Omega)$. Indeed, if $\bu\in W^{1,2}(\Omega)$ is a weak solution to the Stokes problem \eqref{S}, it is easily seen that $\bu+\lambda\bu_0$, $\lambda\in\BR$, are also weak solutions to the problem \eqref{S} corresponding to the same data.\footnote{In particular, the vector field
\be
\tilde{\bu}=\bu-\frac{(\bu,\bu_0)_{L^2(\Omega)}}{\norm{\bu_0}_{L^2(\Omega)}^2}\bu_0
\ee
is a weak solution to the problem \eqref{S} corresponding to the same data such that $(\tilde{\bu},\bu_0)_{L^2(\Omega)}=0$.} Therefore, some restrictions such as $(\bu,\bu_0)_{L^2(\Omega)}=0$ in \eqref{202410010838} are necessary for the uniqueness of weak solutions in Theorem \ref{thm:exist_unique_S} (ii).
\end{rema}

We shall also need the following global regularity result for the Stokes problem \eqref{S} in Subsection \ref{Leray_arg}.
\begin{theo}\label{thm:grS}
Let $1<q<\infty$ and let $\beta\in C^1(\partial\Omega)$ be nonnegative. Suppose that $\bu$ is a q-weak solution to the Stokes problem \eqref{S} corresponding to 
\be
\bff\in L^{q}(\Omega),\;a_{\ast}\in W^{2-1/q,q}(\partial\Omega)\;\text{and }\bb_{\ast}\in W^{1-1/q,q}(\partial\Omega). 
\ee
Assume further that $\int_{\Omega} \bu\cdot\bu_0\,dx=0$ in the case when $\beta\equiv0$ and $H(\Omega)\cap\SR(\Omega)=\mathrm{Span}\{\bu_{0}\}$,\footnote{See footnote \ref{foot:kernel}.} where $\bu_{0}$ is a nonzero rigid displacement.
Then 
\be
\bu\in W^{2,q}(\Omega),\; p\in W^{1,q}(\Omega),
\ee
where $p$ is the pressure associated with $\bu$ by Lemma $\ref{pressure_q}$. Furthermore, the following estimate holds:
\be\label{est:grS}
\norm{\bu}_{W^{2,q}(\Omega)}+\norm{p}_{W^{1,q}(\Omega)}\le c\Bigl(\norm{\bff}_{L^{q}(\Omega)}+\norm{a_{\ast}}_{W^{2-1/q,q}(\partial\Omega)}+\norm{\bb_{\ast}}_{W^{1-1/q,q}(\partial\Omega)}\Bigr)
\ee
with $c=c(q,\Omega,\nu,\beta)$.
\end{theo}
The proof of Theorem \ref{thm:grS} will be given in the Appendix.

\section
{Proof of Theorem \ref{theo1}}\label{Proof of Theorem 1}
In this section we shall prove Theorem \ref{theo1}.
\subsection{Leray's argument reductio ad absurdum}
\label{Leray_arg}
Consider the Navier-Stokes problem \eqref{NS}-\eqref{slip} in a smooth bounded domain $\Omega\subseteq\BR^2$ defined by \eqref{domain} with the external force $\bff\in W^{1,2}(\Omega)$. Without loss of generality, we may assume that $\bff=\nabla^{\perp}g$ with $g\in W^{2,2}(\Omega,\BR)$\footnote{By the Helmholtz-Weyl decomposition of vector fields over smooth bounded planar domains, every $\bff\in W^{1,2}(\Omega)$ can be represented as $\bff=\nabla^{\perp}g+\nabla h$ with $g,h\in W^{2,2}(\Omega,\BR)$; see, e.g., \cite{MR2550139}*{Theorem 3.20}. Then the gradient part is included into the pressure term.}, where $\nabla^{\perp}g=(-\partial_2g,\partial_1g)$. Since the boundary datum $a_{\ast}\in W^{3/2,2}(\partial\Omega)$ satisfies 
\be
\int_{\partial\Omega} a_{\ast} \, ds=\sum_{j=0}^{N} \int_{\Gamma_{j}} a_{\ast} \, ds=0,
\ee
by Lemma \ref{solenoidalextension} there exists a solenoidal extension $\bA\in W^{2,2}(\Omega)$ of $a_{\ast}$ such that
\be
\norm{\bA}_{W^{2,2}(\Omega)} \le c\norm{a_{\ast}}_{{W}^{3/2,2}(\partial \Omega)}.
\ee
Consider the Stokes problem \eqref{S} with the data $\bff$, $\beta$, $a_{\ast}$ and $\bb_{\ast}$ given in the statement of Theorem \ref{theo1}. By Lemma \ref{pressure_q}, Theorem \ref{thm:exist_unique_S} (i) and Theorem \ref{thm:grS}, we can find a weak solution $\bU\in W^{2,2}(\Omega)$ to the Stokes problem \eqref{S} and the pressure $P\in W^{1,2}(\Omega)$ associated with $\bU$ such that $\bU-\bA\in J(\Omega)\cap W^{2,2}(\Omega)$ and $\int_{\Omega}P\,dx=0$. The pair $(\bU,P)$ satisfies the integral identity
\be\label{07040912}
\frac{\nu}{2}\int_{\Omega} S(\bU): S(\bfvarphi) \,dx+\int_{\partial\Omega} \beta\bU_{\tau}\cdot\bfvarphi \,ds-\int_{\partial\Omega}\bb_{\ast}\cdot\bfvarphi\,ds-\int_{\Omega}\bff\cdot\bfvarphi\,dx=\int_{\Omega} P\,\dv\bfvarphi\,dx
\ee
for all $\bfvarphi\in H(\Omega)$ and obeys the inequality
\be\label{07040736}
\norm{\bU}_{W^{2,2}(\Omega)}+\norm{P}_{W^{1,2}(\Omega)}\le c\Bigl(\norm{\bff}_{L^{2}(\Omega)}+\norm{a_{\ast}}_{W^{3/2,2}(\partial\Omega)}+\norm{\bb_{\ast}}_{W^{1/2,2}(\partial\Omega)}\Bigr)
\ee
with $c=(\Omega,\nu,\beta)>0$. We take $\bU\in W^{2,2}(\Omega)$ as a solenoidal extension of $a_{\ast}$ and look for a weak solution to the Navier-Stokes problem \eqref{NS}-\eqref{slip} of the form $\bu=\bw+\bU$ where $\bw\in J(\Omega)$ satisfies the integral identity
\be\label{def:weakNShomo}
\ba
\frac{\nu}{2} \int_{\Omega} S(\bw):S(\bfvarphi) \, dx+\int_{\partial \Omega} \beta \bw_{\tau} \cdot \bfvarphi \, ds=&- \int_{\Omega} (\bU \cdot \nabla)\bU \cdot \bfvarphi \,dx- \int_{\Omega} (\bU \cdot \nabla)\bw \cdot \bfvarphi \, dx\\
&- \int_{\Omega} (\bw \cdot \nabla)\bw \cdot \bfvarphi \, dx- \int_{\Omega} (\bw \cdot \nabla)\bU \cdot \bfvarphi \,dx
\ea
\ee
for all $\bfvarphi\in J(\Omega)$. We recall that the space $J(\Omega)$ is a Hilbert space with respect to the scalar product \eqref{scalar_product}. For fixed $\bw\in J(\Omega)$, the right-hand side of the integral identity \eqref{def:weakNShomo} defines a bounded linear functional on $\bfvarphi\in J(\Omega)$. According to the Riesz representation theorem, there exists one and only one element ${\mathcal{K}}\bw\in J(\Omega)$ such that
\be
\ba
({\mathcal{K}}\bw,\bfvarphi)_{J(\Omega)}=- \int_{\Omega} (\bU \cdot \nabla)\bU \cdot \bfvarphi \,dx- \int_{\Omega} (\bU \cdot \nabla)\bw \cdot \bfvarphi \, dx- \int_{\Omega} (\bw \cdot \nabla)\bw \cdot \bfvarphi \, dx- \int_{\Omega} (\bw \cdot \nabla)\bU \cdot \bfvarphi \,dx
\ea
\ee
for all $\bfvarphi\in J(\Omega)$. Thus the integral identity \eqref{def:weakNShomo} is equivalent to an operator equation in the Hilbert space $J(\Omega)$: 
\be\label{eq:08091216}
\bw={\mathcal{K}}\bw.
\ee 
Moreover, in the same way as in the proof of \cite{MR254401}*{Chapter 5, Theorem 1}, we can prove that the operator $\mathcal{K}$ is a compact mapping from $J(\Omega)$ into itself. Hence we are able to apply the Leray-Schauder theorem (see, e.g., \cite{MR254401}*{Chapter 1, section 3}) to prove the solvability of the equation \eqref{eq:08091216}. To this end, we need to show that all possible solutions $\bw^{(\lambda)}$ to the integral identity
\be\label{eq:Leray-Schauder_lambda}
\ba
\frac{\nu}{2} \int_{\Omega} S(\bw^{(\lambda)}):&S(\bfvarphi)\, dx+\int_{\partial \Omega} \beta \bw^{(\lambda)} \cdot \bfvarphi \, ds=-\lambda\int_{\Omega} (\bU \cdot \nabla)\bU \cdot \bfvarphi \,dx-\lambda\int_{\Omega} (\bU \cdot \nabla)\bw^{(\lambda)} \cdot \bfvarphi \, dx\\
&-\lambda\int_{\Omega} (\bw^{(\lambda)} \cdot \nabla)\bw^{(\lambda)} \cdot \bfvarphi \, dx-\lambda\int_{\Omega} (\bw^{(\lambda)} \cdot \nabla)\bU \cdot \bfvarphi \,dx \quad\forall\bfvarphi\in J(\Omega)
\ea
\ee
are uniformly bounded in $J(\Omega)$ with respect to $\lambda\in[0,1]$. Assume that this is false. Then there exist sequences $\{ \lambda_{k} \}_{k\in\BN} \subset [0,1]$ and $\{ \widehat{\bw}_{k}=\widehat{\bw}^{(\lambda_k)} \}_{k\in\BN} \subset J(\Omega)$ such that
\be\label{eq:weaklambda}
\ba
\frac{\nu}{2} \int_{\Omega} S(\widehat{\bw}_{k}):&S(\bfvarphi) \, dx+\int_{\partial \Omega} \beta\widehat{\bw}_{k} \cdot \bfvarphi\,ds=-\lambda_{k}\int_{\Omega} (\bU \cdot \nabla)\bU \cdot \bfvarphi \,dx-\lambda_{k}\int_{\Omega} (\bU \cdot \nabla)\widehat{\bw}_{k} \cdot \bfvarphi \, dx\\
&-\lambda_{k}\int_{\Omega} (\widehat{\bw}_{k} \cdot \nabla)\widehat{\bw}_{k}\cdot \bfvarphi \, dx-\lambda_{k}\int_{\Omega} (\widehat{\bw}_{k} \cdot \nabla)\bU \cdot \bfvarphi \,dx \quad\forall\bfvarphi\in J(\Omega)
\ea
\ee
and
\be
\lim_{k \to \infty} \lambda_{k} =\lambda_{0} \in [0,1], \quad \lim_{k \to \infty} J_k = \lim_{k \to \infty} \| \widehat{\bw}_{k} \|_{J(\Omega)} = \infty.
\ee
Choose $\bfvarphi = J^{-2}_k \widehat{\bw}_{k}$ in \eqref{eq:weaklambda} and set $\bw_{k}= J^{-1}_k \widehat{\bw}_{k}$, we have
\be\label{07031746}
\ba
&\frac{\nu}{2} \int_{\Omega} S(\bw_{k}):S(\bw_{k}) \, dx+\int_{\partial \Omega} \beta \bw_{k} \cdot \bw_{k} \, ds\\
=&-\frac{\lambda_{k}}{J_k}\int_{\Omega} (\bU \cdot \nabla)\bU \cdot \bw_{k} \,dx- \lambda_{k}\int_{\partial\Omega}  \frac{\abs{\bw_k}^{2}}{2} a_{\ast} \,ds- \lambda_{k}\int_{\Omega} (\bw_{k} \cdot \nabla)\bU \cdot \bw_{k} \,dx,
\ea
\ee
because
\be
\int_{\Omega} (\bU \cdot \nabla) \bw_k \cdot \bw_k \, dx=\int_{\Omega} \bU \cdot \nabla\frac{\abs{\bw_k}^{2}}{2} \,dx=\int_{\Omega} \dv\Bigr(\frac{\abs{\bw_k}^{2}}{2} \bU \Bigl) \,dx=\int_{\partial\Omega} \frac{\abs{\bw_k}^{2}}{2} \bU\cdot\bn \,ds=\int_{\partial\Omega} \frac{\abs{\bw_k}^{2}}{2} a_{\ast} \,ds
\ee
and 
\be\label{202312281626}
\int_{\Omega} (\widehat{\bw}_{k} \cdot \nabla)\widehat{\bw}_{k} \cdot \widehat{\bw}_{k} \, dx=0.
\ee
Since $\norm{\bw_{k}}_{J(\Omega)}=1$, extracting a subsequence if necessary, we may assume that the sequence $\{ \bw_{k} \}_{k\in\BN}$ converges weakly in $J(\Omega)$ to some vector field $\bv\in J(\Omega)$. By the compact embeddings
\be
J(\Omega)\hookrightarrow L^r(\Omega), \quad J(\Omega)\hookrightarrow L^r(\partial\Omega) \quad \forall r\in[1,\infty),
\ee
the subsequence $\{ \bw_{k} \}_{k\in\BN}$ converges strongly in $L^r(\Omega)$ and $L^r(\partial\Omega)$. Therefore, letting $k\to\infty$ in the equality \eqref{07031746}, we deduce
\be\label{08251220}
\ba
1&=- \lambda_{0}\int_{\partial\Omega}  \frac{\abs{\bv}^{2}}{2} a_{\ast} \,ds-\lambda_{0}\int_{\Omega} (\bv\cdot \nabla)\bU \cdot \bv \,dx\\
&=- \lambda_{0}\int_{\partial\Omega}  \frac{\abs{\bv}^{2}}{2} a_{\ast} \,ds+\lambda_{0}\int_{\Omega} (\bv\cdot \nabla)\bv\cdot\bU\,dx.
\ea
\ee
In particular $\lambda_{0}>0$, and so $\lambda_{k}>0$ for all sufficiently large $k$.

Let us go back to the integral identity \eqref{eq:weaklambda}. Define the functional $l_{k}(\bfvarphi)$ on $\bfvarphi\in H(\Omega)$ by
\be\label{07040911}
\ba
l_{k}(\bfvarphi)\equiv&\frac{\nu}{2} \int_{\Omega} S(\widehat{\bw}_{k}):S(\bfvarphi) \, dx+\int_{\partial \Omega} \beta \widehat{\bw}_{k} \cdot \bfvarphi \, ds+\lambda_{k}\int_{\Omega} (\bU \cdot \nabla)\bU \cdot \bfvarphi \,dx\\
&+\lambda_{k}\int_{\Omega} (\bU \cdot \nabla)\widehat{\bw}_{k} \cdot \bfvarphi \, dx+\lambda_{k}\int_{\Omega} (\widehat{\bw}_{k} \cdot \nabla)\widehat{\bw}_{k}\cdot \bfvarphi \, dx+\lambda_{k}\int_{\Omega} (\widehat{\bw}_{k} \cdot \nabla)\bU \cdot \bfvarphi \,dx.
\ea
\ee
Obviously, $l_{k}(\bfvarphi)$ is a linear functional on $H(\Omega)$. By using the Sobolev embedding theorem and \eqref{07040736}, we obtain
\be
\ba
\abs{l(\bfvarphi)}&\le c\big(\norm{\widehat{\bw}_{k}}_{J(\Omega)}+\norm{\widehat{\bw}_{k}}^{2}_{J(\Omega)}+\norm{\bU}^{2}_{W^{1,2}(\Omega)}\big)\norm{\bfvarphi}_{H(\Omega)}\\
&\le c\big(\norm{\widehat{\bw}_{k}}_{J(\Omega)}+\norm{\widehat{\bw}_{k}}^{2}_{J(\Omega)}+\norm{\bff}^{2}_{L^{2}(\Omega)}+\norm{a_{\ast}}^{2}_{W^{3/2,2}(\partial\Omega)}+\norm{\bb_{\ast}}^{2}_{W^{1/2,2}(\partial\Omega)}\big)\norm{\bfvarphi}_{H(\Omega)}
\ea
\ee
with constant $c=c(\Omega,\nu,\beta)$ independent of $k$. Moreover, it follows from \eqref{eq:weaklambda} that $l_{k}(\bfvarphi)=0$ for all $\bfvarphi\in J(\Omega)$. Thus, in view of Lemma \ref{pressure} stated below, there exist scalar functions $\widetilde{p_{k}}\in L^{2}(\Omega)$ with $\int_{\Omega} \widetilde{p_{k}}\,dx=0$ such that
\be\label{07040800}
l_{k}(\bfvarphi)=\int_{\Omega} \widetilde{p_{k}}\dv\bfvarphi\,dx\quad\forall\bfvarphi\in H(\Omega)
\ee
and
\be
\norm{\widetilde{p_{k}}}_{L^2(\Omega)}\le c\big(\norm{\widehat{\bw}_{k}}_{J(\Omega)}+\norm{\widehat{\bw}_{k}}^{2}_{J(\Omega)}+\norm{\bff}^{2}_{L^{2}(\Omega)}+\norm{a_{\ast}}^{2}_{W^{3/2,2}(\partial\Omega)}+\norm{\bb_{\ast}}^{2}_{W^{1/2,2}(\partial\Omega)}\big)
\ee
with constant $c$ independent of $k$. 
\begin{lemm}[see \cite{SoSc73}]\label{pressure} Let $\Omega\subseteq\BR^2$ be a bounded domain with Lipschitz boundary, and let $l(\bfvarphi)$ be a bounded linear functional on $\bfvarphi\in H(\Omega)$. If $l(\bfvarphi)=0$ for all $\bfvarphi \in J(\Omega)$, then there exists a scalar function $p\in L^2(\Omega)$ with $\int_{\Omega} p(x)\,dx=0$ such that
\be
l(\bfvarphi)=\int_{\Omega} p\,\dv\bfvarphi \,dx 
\ee
for all $\bfvarphi\in H(\Omega)$. Moreover, $\norm{p}_{L^{2}(\Omega)}$ is equivalent to $\norm{l}_{H(\Omega)'}$.
\end{lemm}
\vspace{0.3cm}Putting $\widehat{\bu}_k=\widehat{\bf w}_k+\bU$ and $\widehat p_k=\widetilde{p_{k}}+P$. From \eqref{07040912}, \eqref{07040911} and \eqref{07040800}, we find that the pairs $\big(\widehat{\bu}_k, \, \widehat p_k\big)$ satisfy the integral identity
\be\label{07041623}
\ba
&\frac{\nu}{2} \int_{\Omega} S(\widehat{\bu}_{k}):S(\bfvarphi) \, dx+\int_{\partial \Omega} \beta \widehat{\bu}_{k} \cdot \bfvarphi \, ds\\
&=-\lambda_{k}\int_{\Omega} (\widehat{\bu}_{k} \cdot \nabla)\widehat{\bu}_{k}\cdot \bfvarphi \, dx+\int_{\Omega}\bff\cdot\bfvarphi\,dx+\int_{\partial\Omega}\bb_{\ast}\cdot\bfvarphi\,ds+\int_{\Omega} {\widehat p_k}\dv\bfvarphi\,dx
\ea
\ee
for all $\bfvarphi\in H(\Omega)$. Since $\widehat{\bu}_{k}\in W^{1,2}(\Omega)$, we deduce $(\widehat{\bu}_{k} \cdot \nabla)\widehat{\bu}_{k}\in L^{q}(\Omega)$ for $q\in [1,2)$ with the estimate
\be
\ba
\norm{(\widehat{\bu}_{k} \cdot \nabla)\widehat{\bu}_{k}}_{L^{q}(\Omega)}&\le\norm{\widehat{\bu}_{k}}_{L^{2q/(2-q)}(\Omega)}\norm{\nabla\widehat{\bu}_{k}}_{L^{2}(\Omega)}\le c\bigl(\norm{\widehat{\bw}_{k}}_{W^{1,2}(\Omega)}+\norm{\bU}_{W^{1,2}(\Omega)}\bigr)^2\\
&\le c\Bigl(\norm{\widehat{\bw}_{k}}^{2}_{J(\Omega)}+\norm{\bff}^{2}_{L^{2}(\Omega)}+\norm{a_{\ast}}^{2}_{W^{3/2,2}(\partial\Omega)}+\norm{\bb_{\ast}}^{2}_{W^{1/2,2}(\partial\Omega)}\bigr),
\ea
\ee
where $c=c(q,\Omega,\nu,\beta)$. Thus the pairs $\big(\widehat{\bu}_k=\widehat{\bf w}_k+\bU, \, \widehat p_k=\widetilde{p_{k}}+P\big)$ might be considered as $q$-weak solutions ($1<q<2$) to the Stokes problem
\be
\ba\label{}
\left\{
\ba
-\nu \Delta \widehat{\bu}_k+\nabla \widehat p_k & =\bff-\lambda_k(\widehat{\bu}_k \cdot \nabla)\widehat{\bu}_k \quad & & \text{in} \enskip \Omega,\\
\dv \widehat{\bu}_k & =0 \quad & & \text{in} \enskip \Omega,\\
\widehat{\bu}_k \cdot\bn & = a_{\ast} \quad & & \text{on} \enskip \partial \Omega,\\
[T(\widehat{\bu}_k, \widehat p_k)\bn]_{\tau}+\beta \widehat {\bu}_{k\tau} & =\mathbf b_{\ast} \quad & &\text{on} \enskip \partial \Omega.
\ea 
\right. 
\ea
\ee
Here $T(\widehat{\bu}_k, \widehat p_k)=-\widehat p_kI+\nu S(\widehat{\bu}_k)$. By Theorem \ref{thm:grS}, we have $\widehat{\bu}_k\in W^{2,q}(\Omega)$, $\widehat p_k\in W^{1,q}(\Omega)$, $1<q<2$, along with the estimate
\be\label{07041635}
\ba
&\norm{\widehat{\bu}_k}_{W^{2,q}(\Omega)}+\norm{\widehat p_k}_{W^{1,q}(\Omega)}\\
\le c\Bigl(&\norm{\bff}_{L^{q}(\Omega)}+\norm{\lambda_k(\widehat{\bu}_k \cdot \nabla)\widehat{\bu}_k}_{L^{q}(\Omega)}+\norm{a_{\ast}}_{W^{2-1/q,q}(\partial\Omega)}+\norm{\bb_{\ast}}_{W^{1-1/q,q}(\partial\Omega)}\Bigr)\\
\le c\Bigl(&\norm{\bff}_{L^{2}(\Omega)}+\norm{\widehat{\bw}_{k}}^{2}_{J(\Omega)}+\norm{\bff}^{2}_{L^{2}(\Omega)}+\norm{a_{\ast}}^{2}_{W^{3/2,2}(\partial\Omega)}+\norm{\bb_{\ast}}^{2}_{W^{1/2,2}(\partial\Omega)}\\
&+\norm{a_{\ast}}_{W^{3/2,2}(\partial\Omega)}+\norm{\bb_{\ast}}_{W^{1/2,2}(\partial\Omega)}\Bigr)
\ea
\ee
with $c=c(q,\Omega,\nu,\beta)$. We now show by a simple boot-strap argument that, in fact, $\widehat{\bu}_k\in W^{2,2}(\Omega)\cap W^{3,2}_{\text{loc}}(\Omega)$ and $\widehat{p}_k\in W^{1,2}(\Omega)\cap W^{2,2}_{\text{loc}}(\Omega)$. Since $\widehat{\bu}_k\in W^{2,3/2}(\Omega)$, by the embeddings
\be
W^{2,3/2}(\Omega)\hookrightarrow W^{1,6}(\Omega)\hookrightarrow L^{\infty}(\Omega),
\ee
we have $\bff-\lambda_k(\widehat{\bu}_{k} \cdot \nabla)\widehat{\bu}_{k}\in W^{1,3/2}(\Omega)$. Again by Theorem \ref{thm:grS}, we deduce $\widehat{\bu}_k\in W^{2,2}(\Omega)$ and $\widehat{p}_k\in W^{1,2}(\Omega)$, which yields $\bff-\lambda_k(\widehat{\bu}_{k} \cdot \nabla)\widehat{\bu}_{k}\in W^{1,2}(\Omega)$. Then, in view of the interior regularity result for the Stokes equations\footnote{The interior regularity of the solution does not depend on the regularity of the boundary $\partial\Omega$ and of the boundary data $a_{\ast}$ and $\bb_{\ast}$.} (see \cite{MR2808162}*{Theorem I\hspace{-1.2pt}V.4.1}), we conclude that $\widehat{\bu}_k\in W^{2,2}(\Omega)\cap W^{3,2}_{\text{loc}}(\Omega)$ and $\widehat{p}_k\in W^{1,2}(\Omega)\cap W^{2,2}_{\text{loc}}(\Omega)$.

Put $\nu_k\equiv(\lambda_kJ_k)^{-1}\nu$. Multiplying \eqref{07041623} by $\frac{1}{\lambda_kJ^2_k}=\frac{\lambda_k\nu^2_k}{\nu^2}$, we find that the pair $\big(\bu_k\equiv\frac{1}{J_k}\widehat{\bu}_k, \, p_k\equiv\frac{1}{\lambda_kJ^2_k}\widehat p_k\big)$ satisfies the following identity for all $\bfvarphi\in H(\Omega)$:
\be\label{08191350}
\ba
\frac{\nu_k}{2} \int_{\Omega} S(\bu_{k}):S(\bfvarphi) \, dx+\int_{\partial \Omega} \beta_k {\bu_k}_{\tau} \cdot \bfvarphi \, ds=&-\int_{\Omega} (\bu_{k} \cdot \nabla)\bu_k\cdot \bfvarphi \, dx+\int_{\Omega} \bff_k\cdot\bfvarphi\,dx\\
&+\int_{\partial\Omega} \bb_k\cdot\bfvarphi \,ds+\int_{\Omega} p_k\,\dv\bfvarphi\,dx.
\ea
\ee
Here $\beta_k\equiv\frac{\nu_k}{\nu}\beta$, $\bff_k\equiv\frac{\lambda_k\nu^2_k}{\nu^2}\bff$ and $\bb_k\equiv\frac{\lambda_k\nu^2_k}{\nu^2}\bb_{\ast}$. Thus the pair $\big(\bu_k, \, p_k\big)$ is a weak solution to the Navier-Stokes system
\begin{align} \label{eq:NS_k}
\left\{
\ba
-\nu_k \Delta \bu_k+(\bu_k \cdot \nabla)\bu_k+ \nabla p_k & =\bff_k \quad & & \text{in} \enskip \Omega,\\
\dv \bu_k & =0 \quad & & \text{in} \enskip \Omega,\\
\bu_k \cdot\bn & = a_k \quad & & \text{on} \enskip \partial \Omega,\\
[T(\bu_k,p_k)\bn]_{\tau}+\beta_k \bu_{k\tau} & =\bb_k  \quad & &\text{on} \enskip \partial \Omega,
\ea 
\right. 
\end{align}
where $a_k\equiv\frac{\lambda_k\nu_k}{\nu}a_{\ast}$ and $T(\bu_k, p_k)=-p_kI+\nu_kS(\bu_k)$. From \eqref{07041635} we find that the sequence $\{p_k\}$ is uniformly bounded  in $W^{1,q}(\Omega)$ for each $q\in [1,2)$. Thus, extracting a subsequence (if necessary), we may assume that $p_k\rightharpoonup p$ in $W^{1,q}(\Omega)$ for each $q\in [1,2)$. Also, it is immediately verified that the sequence $\{\bu_{k}\}$ is uniformly bounded in $W^{1,2}(\Omega)$, $\bu_{k}\rightharpoonup \bv$ in $W^{1,2}(\Omega)$ and ${\bu}_k\in W^{2,2}(\Omega)\cap W^{3,2}_{\text{loc}}(\Omega)$, ${p}_k\in W^{1,2}(\Omega)\cap W^{2,2}_{\text{loc}}(\Omega)$. Let $\bfxi$ be an arbitrary vector field from $C^{\infty}_{0}(\Omega)$. Taking $\bfvarphi=\bfxi$ in \eqref{08191350} and then letting $k\to\infty$, we have
\be
0=-\int_{\Omega} (\bv\cdot\nabla)\bv\cdot\bfxi\,dx+\int_{\Omega} p\,\dv\bfxi\,dx \quad\forall\bfxi\in C^{\infty}_{0}(\Omega).
\ee
Integrating by parts gives the identity
\be
\int_{\Omega} (\bv\cdot\nabla)\bv\cdot\bfxi\,dx+\int_{\Omega} \nabla p\cdot\bfxi\,dx=0 \quad\forall\bfxi\in C^{\infty}_{0}(\Omega).
\ee
Hence the pair $(\bv,p)$ satisfies, for almost all $x\in\Omega$, the Euler system (recall that $\bv\in J(\Omega)$)
\begin{align} \label{Euler}
\left\{
\ba
(\bv\cdot\nabla)\bv+\nabla p & =\boldsymbol{0} \quad & & \text{in} \enskip \Omega,\\
\dv \bv & =0 \quad & & \text{in} \enskip \Omega,\\
\bv \cdot\bn & = 0 \quad & & \text{on} \enskip \partial \Omega.
\ea 
\right. 
\end{align}
Since $\dv\bU=0$ in $\Omega$ and $\bU\cdot\bn=a_{\ast}$ on $\partial\Omega$, it follows from \eqref{08251220} and $\eqref{Euler}_1$ that
\be\label{08251507}
\ba
\frac{1}{\lambda_{0}}&=-\int_{\partial\Omega}  \frac{\abs{\bv}^{2}}{2} a_{\ast} \,ds+\int_{\Omega} (\bv\cdot \nabla)\bv\cdot\bU\,dx=-\int_{\partial\Omega}  \frac{\abs{\bv}^{2}}{2} a_{\ast} \,ds-\int_{\Omega} \nabla p\cdot\bU\,dx\\
&=-\int_{\partial\Omega}  \frac{\abs{\bv}^{2}}{2} a_{\ast} \,ds-\int_{\Omega} \dv(p\bU)\,dx=-\int_{\partial\Omega}  \frac{\abs{\bv}^{2}}{2} a_{\ast} \,ds-\int_{\partial\Omega} p\,\bU\cdot\bn\,ds\\
&=-\int_{\partial\Omega} (p+\frac{\abs{\bv}^{2}}{2}) a_{\ast}\,ds=-\int_{\partial\Omega}\Phi a_{\ast}\,ds,
\ea
\ee
where $\Phi\equiv p+\frac{\abs{\bv}^2}{2}$ is the total head pressure corresponding to the solution $(\bv,p)$ to the steady Euler equations \eqref{Euler}. Obviously, $\Phi\in W^{1,q}(\Omega)$ for all $q\in[1,2)$.

The results obtained so far can be summarized in the following.
\begin{lemm}\label{lem_Leray} Let $\Omega$ be a smooth bounded domain defined by \eqref{domain} and let $\beta\in C^1(\partial\Omega)$ be nonnegative and $\beta\not\equiv0$. Suppose that $a_{\ast}\in W^{3/2,2}(\partial\Omega)$ and $\bb_{\ast}\in W^{1/2,2}(\partial\Omega)$ satisfy \eqref{theo1_comp} and $\bff=\nabla^{\perp}g$, $g\in W^{2,2}(\Omega,\BR)$. If there are no weak solutions to the Navier-Stokes problem \eqref{NS}-\eqref{slip}, then there exist $\bv$, $p$ with the following properties:
\begin{itemize}
\item[$(\mathrm{E})$] $\bv\in J(\Omega)$, $p\in W^{1,q}(\Omega)$, $q\in[1,2)$, and the pair $(\bv,p)$ satisfies the steady Euler equations \eqref{Euler}.
\item[$(\mathrm{E\mathchar`-NS})$] There exist sequences of functions ${\bu}_k\in W^{2,2}(\Omega)\cap W^{3,2}_{\text{loc}}(\Omega)$, ${p}_k\in W^{1,2}(\Omega)\cap W^{2,2}_{\text{loc}}(\Omega)$ and numbers $\nu_k\to+0$, $\lambda_k\to\lambda_0>0$ such that the norms $\norm{\bu_k}_{W^{1,2}(\Omega)}$, $\norm{p_k}_{W^{1,q}(\Omega)}$ are uniformly bounded for each $q\in[1,2)$, the pairs $(\bu_k,p_k)$ satisfy the Navier-Stokes system \eqref{eq:NS_k} with $\beta_k=\frac{\nu_k}{\nu}\beta$, $\bff_k=\frac{\lambda_k\nu^2_k}{\nu^2}\bff$, $a_k=\frac{\lambda_k\nu_k}{\nu}a_{\ast}$, $\bb_k=\frac{\lambda_k\nu^2_k}{\nu^2}\bb_{\ast}$, and
\be
\ba
\bu_{k}\rightharpoonup \bv\enskip\text{in }W^{1,2}&(\Omega),\quad p_k\rightharpoonup p\enskip\text{in }W^{1,q}(\Omega)\quad\forall q\in[1,2),\\
&\frac{1}{\lambda_{0}}=-\int_{\partial\Omega}\Phi a_{\ast}\,ds.
\ea
\ee
Here $\Phi=p+\frac{\abs{\bv}^2}{2}\in W^{1,q}(\Omega)$, $q\in[1,2)$.
\end{itemize}

\end{lemm}
\subsection{Some properties of solutions to the steady Euler equations}\label{Leray_arg_Euler} Let us next investigate the values of the total head pressure $\Phi$ on the boundary $\partial\Omega$. From $\eqref{Euler}_{2,3}$ there exists a stream function $\psi\in W^{2,2}(\Omega)$ such that
\be\label{iden:stream}
\nabla\psi=(-v_2,v_1)\enskip\text{in }\overline{\Omega}.
\ee
By the Sobolev embedding theorem, $\psi$ is continuous on $\overline{\Omega}$. The level sets of $\psi$ are called streamlines. Applying Lemmas \ref{lem:fine}, \ref{lem:Dorr} and Remark \ref{rem:extension_trace} to functions $\bv$, $\psi$ and $\Phi$, we have the following.
\begin{lemm}\label{lem:fine_Euler}The stream function $\psi$ is continuous on $\overline{\Omega}$ and there exists a set $A_{\bv}\subseteq\BR^2$ satisfying the following properties:

$(\mathrm{i})$ $\SH^1(A_{\bv})=0$. 

$(\mathrm{ii})$ For all $x\in\Omega\setminus A_{\bv}$,
\be
\lim_{r\to0}\dashint_{B(x,r)}\abs{\bv(y)-\bv(x)}^2\,dy=\lim_{r\to0}\dashint_{B(x,r)}\abs{\Phi(y)-\Phi(x)}^2\,dy=0.
\ee
Moreover, the function $\psi$ is differentiable in the classical sense at $x$ and $\nabla\psi(x)=(-v_2(x),v_1(x))$.

$(\mathrm{iii})$ For each $\epsilon>0$, there exists an open set $V\subseteq\BR^2$ such that $\SH^1_{\infty}(V)\le\epsilon$, $A_{\bv}\subseteq V$ and the functions $\bv$ and $\Phi$ are continuous on $\overline{\Omega}\setminus V$.
\end{lemm}

Taking into account $\eqref{Euler}_1$, by a direct calculation one finds that
\be\label{iden:Bernoulli}
\nabla\Phi=(\partial_2v_1-\partial_1v_2)(-v_2,v_1)=\omega\nabla\psi\enskip\text{in }\Omega,
\ee
where $\omega$ denotes the vorticity $\omega=\partial_2v_1-\partial_1v_2=\Delta\psi$. The identity \eqref{iden:Bernoulli} implies that the total head pressure $\Phi$ is  constant on each streamline. The following Bernoulli's law was established in \cite{MR2848783}*{Theorem 1}. (See also \cite{MR3004771}*{Theorem 3.2} for a more detailed proof.)
\begin{theo}[Bernoulli's law]\label{thm:Bernoulli} Let $A_{\bv}$ be the set from Lemma $\ref{lem:fine_Euler}$. For every compact connected set $K\subseteq\overline\Omega$, the following property holds: if
\be
\psi\big|_{K}=\const,
\ee
then
\be
\Phi(x_1)=\Phi(x_2)\enskip\text{for all }x_1,x_2\in K\setminus A_{\bv}.
\ee
\end{theo}
\begin{lemm}\label{lem:stream}There exist constants $\xi_0, \xi_1,\dots,\xi_N\in\BR$ such that
\be\label{202307222007}
\psi(x)\equiv\xi_j\quad\text{for all }x\in\Gamma_j,\enskip j=0,1,\dots,N.
\ee
\end{lemm}
\begin{proof} We shall show that $\psi$ takes a constant value on the outer boundary $\Gamma_0$. Since $\Gamma_0$ is smooth, for every $x_0\in\Gamma_0$ there exist a $2\times2$ rotation matrix $R=(r_{ij})_{1\le i,j\le2}$, a smooth function $\zeta\colon\BR\to\BR$ with $\zeta(0)=0$ and $\frac{d\zeta}{dy_1}(0)=0$, and $r>0$ such that, setting $y=T(x)=R(x-x_0)$, we have $T(\Omega\cap B(x_0,r))=\{y\in B(0,r):y_2<\zeta(y_1)\}$. By taking a sufficiently small $0<\epsilon<r$, in the original coordinate system $x$, a boundary portion $\sigma$ of $\Gamma_0$ around the point $x_0$ is represented as 
\be
\sigma(y_1)=R^{\top}(y_1,\zeta(y_1))+x_0,\enskip-\epsilon<y_1<\epsilon,
\ee 
and the unit outer normal $\bn(x)$ at the point $x\in \sigma\left((-\epsilon,\epsilon)\right)\subseteq\Gamma_0$ is given by
\be
\bn(x)=(1+\frac{d\zeta}{dy_1}(y_1)^2)^{-\frac{1}{2}}R^{\top}(-\frac{d\zeta}{dy_1}(y_1),1),\enskip(y_1,\zeta(y_1))=T(x).
\ee
Since $\psi\in W^{2,2}(\Omega)$ is continuous on $\overline{\Omega}$, the composition $\psi\circ \sigma$ is continuous on $(-\epsilon,\epsilon)$. Using the fact that the restriction to $\Omega$ of functions from $C^{\infty}_{0}(\BR^2)$ is dense in $W^{2,2}(\Omega)$ (see \cite{MR3726909}*{Theorem 11.35}), we can show that $\psi\circ \sigma$ admits a weak derivative in $L^1_{\mathrm{loc}}(-\epsilon,\epsilon)$ and 
\be
\frac{d(\psi\circ \sigma)}{dy_1}(y_1)=\gamma(\nabla\psi)(\sigma(y_1))\cdot \sigma^{\prime}(y_1)\text{ in }L^1_{\mathrm{loc}}(-\epsilon,\epsilon),
\ee
where $\gamma\colon W^{1,2}(\Omega)\to L^2(\partial\Omega)$ is the trace operator. In view of \eqref{iden:stream} and $\eqref{Euler}_3$, we observe that
\be
\ba
\frac{d(\psi\circ \sigma)}{dy_1}(y_1)&=\gamma(\nabla\psi)(\sigma(y_1))\cdot \sigma^{\prime}(y_1)\\
&=\Big(-v_2(\sigma(y_1)),v_1(\sigma(y_1))\Big)\cdot\Big(r_{11}+r_{21}\frac{d\zeta}{dy_1}(y_1),r_{12}+r_{22}\frac{d\zeta}{dy_1}(y_1)\Big)\\
&=-\bv(\sigma(y_1))\cdot(-r_{12}-r_{22}\frac{d\zeta}{dy_1}(y_1), r_{11}+r_{21}\frac{d\zeta}{dy_1}(y_1))\\
&=-\bv(\sigma(y_1))\cdot R^{\top}(-\frac{d\zeta}{dy_1}(y_1),1)\\
&=-(1+\frac{d\zeta}{dy_1}(y_1)^2)^{\frac{1}{2}}\bv(\sigma(y_1))\cdot\bn(\sigma(y_1))=0
\ea
\ee
for almost all $y_1\in (-\epsilon,\epsilon)$, where we have used the fact that $R=(r_{ij})_{1\le i,j\le2}$ is a rotation matrix (i.e., $r_{11}=r_{22}$ and $r_{12}=-r_{21}$). As a consequence, there exists a constant $\xi_0\in\BR$ such that $\psi\circ \sigma(y_1)\equiv\xi_0$ for all $y_1\in (-\epsilon,\epsilon)$ (recall that $\psi\circ \sigma$ is continuous on $(-\epsilon,\epsilon)$). Since $\Gamma_0$ is connected and compact, we conclude $\psi(x)\equiv\xi_0$ for all $x\in\Gamma_0$. In the same way, it can be shown that there exist constants $\xi_1, \xi_2, \dots, \xi_N\in\BR$ such that $\psi(x)\equiv\xi_j$ for all $x\in\Gamma_j$, $j=1,2,\dots,N$.
\end{proof}
From Theorem \ref{thm:Bernoulli} and Lemma \ref{lem:stream}, we obtain the following corollary, which can be regarded as a generalization of \cite{MR720205}*{Lemma 4} and \cite{MR763943}*{Theorem 2.2}. See also \cite{MR3004771}*{Remark 3.2}.
\begin{cor}\label{cor:Bernoulli} Let $A_{\bv}$ be the set in Lemma $\ref{lem:fine_Euler}$. Then there exist constants $\widehat\Phi_0, \widehat\Phi_1,\dots,\widehat\Phi_N\in\BR$ such that 
\be\label{08251509}
\Phi(x)\equiv\widehat\Phi_j\enskip\text{for all}\enskip x\in\Gamma_j\setminus A_{\bv},\enskip j=0,\dots,N.
\ee
\end{cor}
\vspace{0.3cm}
In view of \eqref{08251507} and \eqref{08251509}, we have
\be\label{202309101143}
\frac{1}{\lambda_0}=-\int_{\partial\Omega}\Phi a_{\ast}\,ds=-\sum_{j=0}^{N}\widehat\Phi_j\int_{\Gamma_j}a_{\ast}\,ds.
\ee
\begin{rema}\label{rem:Bernoulli}
(i) The identity \eqref{202309101143} gives a contradiction if the fluxes of $a_{\ast}$ satisfy the restriction \eqref{zero_flux_slip}.

(ii) In general, one cannot claim that all constants $\widehat\Phi_j$ in Corollary \ref{cor:Bernoulli} are equal; $\widehat\Phi_0=\widehat\Phi_1=\dots=\widehat\Phi_N$. See \cite{MR763943}*{Example 3.1}.

(iii) Unlike the result in \cite{MR720205}*{Lemma 4} and in \cite{MR763943}*{Theorem 2.2}, Corollary \ref{cor:Bernoulli} fails in the three-dimensional case. To see this, let us suppose that $\Omega=B(0,1)=\{x=(x_1,x_2,x_3)\in\BR^3:\abs{x}<1\}$ and set $\bu_{0}\equiv e_3\times x=(-x_2,x_1,0)$, where $e_3=(0,0,1)$. Then we find that $\bu_{0}$ satisfies the Euler system \eqref{Euler} together with $p=\frac{\abs{\bu_{0}}^2}{2}=\frac{x_1^2+x_2^2}{2}$, and the corresponding total head pressure $\Phi(x)=p+\frac{\abs{\bu_{0}}^2}{2}=x_1^2+x_2^2$ does not take a constant value on the unit sphere $\abs{x}=1$.
\end{rema}
\subsection{Arriving at a contradiction}\label{theo1_contradiction}

In this subsection, we shall derive the desired contradiction by employing the approach of Korobkov-Pileckas-Russo \cite{MR3275850}*{Subsection 3.3}. We consider separately the two possible cases.

(a) The maximum of $\Phi$ is attained on the boundary $\partial\Omega$:
\be\label{case_a}
\max\limits_{j=0,\dots,N} \widehat\Phi_{j} \ge \esssup\limits_{x\in\Omega}\Phi(x).
\ee

(b) The  maximum of $\Phi$ is not attained on $\partial\Omega$:\footnote{The case $\esssup\limits_{x\in\Omega}\Phi(x)=+\infty$ is not excluded. }
\be\label{case_b}
\max\limits_{j=0,\dots,N}\widehat\Phi_{j} <\esssup\limits_{x\in\Omega}\Phi(x).
\ee
\subsubsection{The maximum of $\Phi$ is attained on the boundary $\partial\Omega$}\label{subsub1}
Let us first consider the case \eqref{case_a}. Adding a constant to the pressure (if necessary) we may assume that
\be
0=\max\limits_{j=0,\dots,N}\widehat\Phi_j\ge\esssup\limits_{x\in\Omega}\Phi(x).
\ee
In particular,
\be
\Phi(x)\le0\quad\text{in }\Omega.
\ee
If $\widehat{\Phi}_0=\widehat{\Phi}_1=\dots=\widehat{\Phi}_N$, then the flux condition \eqref{flux} and the identity \eqref{202309101143} give a contradiction. Thus, assume that
\be
\min\limits_{j=0,\dots,N}\widehat\Phi_j<0.
\ee

Change (if necessary) the numbering of the boundary components $\Gamma_0,\Gamma_1,\dots,\Gamma_N$ in such a way that
\be\label{202309121048}
\widehat{\Phi}_j<0,\enskip j=0,\dots,M,
\ee
\be\label{202309121049}
\widehat{\Phi}_{M+1}=\dots=\widehat{\Phi}_{N}=0.
\ee

We denote by $\Phi_k\equiv p_k+\frac{\abs{\bu_k}^2}{2}$ the total head pressures corresponding to the solutions $(\bu_k, p_k)$ to the Navier-Stokes system \eqref{eq:NS_k}. Our first goal is to separate the boundary component $\Gamma_{N}$ (where $\Phi=0$) from the boundary components $\Gamma_{j}$, $j=0,\dots,M$ (where $\Phi<0$), by level sets of $\Phi$ compactly contained in $\Omega$. More precisely, for any $j=0,\dots,M$ and for some sequence of positive values $\{t_{i}\}_{i\in\BN}$ we construct a compact connected set $A^{j}_{i}\subset\subset\Omega$ satisfying the following properties:
\begin{enumerate}
\item $\psi|_{A^{j}_{i}}\equiv\const$ and $\Phi(A^{j}_{i})=\{-t_{i}\}$;
\item Each $A^{j}_{i}$ is homeomorphic to the unit circle $\BS^{1}$, and the boundary components $\Gamma_{j}$ and $\Gamma_{N}$ lie in different connected components of $\BR^2\setminus A^{j}_{i}$;
\item There exists a subsequence $\Phi_{k_l}$ such that the restrictions $\Phi_{k_l}|_{A^{j}_{i}}$ and $\Phi|_{A^{j}_{i}}$ are continuous and $\Phi_{k_l}|_{A^{j}_{i}}$ converges to $\Phi|_{A^{j}_{i}}$ uniformly: $\Phi_{k_l}|_{A^{j}_{i}}\rightrightarrows\Phi|_{A^{j}_{i}}$. 
\end{enumerate}

For the construction of the set $A^{j}_{i}$, we shall need some information on the behavior of the total head pressure $\Phi$ on streamlines. To this end we first recall the classical result due to Kronrod \cite{MR34826} concerning level sets of continuous functions.
\begin{lemm}[see \cite{MR34826}*{Lemma 13}]\label{lem:Kro} Let $Q=[0,1]\times[0,1]$ be a square in $\BR^2$ and let $f\in C(Q,\BR)$. Denote by $T_{f}$ a family of all connected components of level sets of $f$. Then for any two different points $A,B\in T_{f}$ there exists a continuous injective function $\phi\colon[0,1]\to T_{f}$\footnote{The convergence in $T_{f}$ is defined as follows: $T_{f}\ni C_i\to C$ if and only if $\sup\limits_{x\in C_i}\dist(x,C)\to 0$.} satisfying the following properties:
\begin{enumerate}
\item $\phi(0)=A$ and $\phi(1)=B$. 
\item For any $t_{0}\in [0,1]$,
\be
\lim_{[0,1]\ni t\to t_{0}}\sup_{x\in \phi(t)}\dist(x,\phi(t_{0}))=0.
\ee
\item For each $t\in (0,1)$, the sets $A$ and $B$ lie in different connected components of the set $Q\setminus \phi(t)$.
\end{enumerate}
\end{lemm} 
The range $\{\phi(t):0\le t\le1\}$ will be called the Kronrod arc joining $A$ to $B$ and denoted by $[A,B]$.
\begin{rema}\label{rem:Kro}
The assertion of Lemma \ref{lem:Kro} remains valid for level sets of continuous functions $f\colon\overline{\Omega}\to\BR$, where $\Omega$ is a multiply-connected bounded domain of type \eqref{domain}, provided $f(x)\equiv\xi_j=\const$ on each inner boundary $\Gamma_j$, $j=1,\dots,N$. Indeed, we can extend $f$ to the whole $\overline{\Omega}_{0}$ by setting $f(x)\equiv\xi_j$ for $x\in\overline{\Omega}_j$, $j=1,\dots,N$. The extended function $f$ will be continuous on $\overline{\Omega}_{0}$, which is homeomorphic to the unit square $Q=[0,1]\times[0,1]$.
\end{rema}
From Lemma \ref{lem:stream} and Remark \ref{rem:Kro} one can apply Lemma \ref{lem:Kro} to the stream function $\psi$. For any element $K\in T_{\psi}$ with $\diam K>0$, we define the value $\Phi(K)$ as $\Phi(K)=\Phi(x)$, where $x\in K\setminus A_{\bv}$ (notice that $K\setminus A_{\bv}\neq\emptyset$ if $\diam K>0$). By Theorem \ref{thm:Bernoulli} this is well-defined. 

The following lemma describes the continuity property of $\Phi$ on the Kronrod arc $[A,B]$. Note that the total head pressure $\Phi(x)$ itself is not necessarily continuous on $\overline{\Omega}$.
\begin{lemm}[see \cite{MR3275850}*{Lemma 3.5}]\label{lem:head_conti}Let $A,B\in T_{\psi}$ with $\diam A, \diam B>0$, and let $\phi\colon[0,1]\to T_{\psi}$ be the  function joining $A$ to $B$ in Lemma $\ref{lem:Kro}$. Then the restriction $\Phi|_{[A,B]}\equiv\Phi(\phi(t))\colon[0,1]\to\BR$ is a continuous function of $t\in[0,1]$.
\end{lemm}

For $x\in\overline{\Omega}$, denote by $K_{x}$ the connected component of the level set $\{z\in\overline{\Omega}:\psi(z)=\psi(x)\}$ containing the point $x$. By Lemma \ref{lem:stream}, $K_{x}\cap\partial\Omega=\emptyset$ for any $y\in\psi(\overline{\Omega})\setminus\{\xi_0, \xi_1,\dots,\xi_N\}$ and for any $x\in\psi^{-1}(y)$. According to Theorem \ref{thm:Morse-Sard} (ii) and (iii), for almost all $y\in\psi(\overline{\Omega})$ and for all $x\in\psi^{-1}(y)$, the equality $K_{x}\cap A_{\bv}=\emptyset$ holds and the component $K_{x}\subseteq\Omega$ is a $C^{1}$-curve homeomorphic to the unit circle $\BS^{1}$. We call such a $K_x$ an admissible cycle. The next uniform convergence result was obtained in \cite{MR3004771}*{Lemma 3.3}. 
\begin{lemm}\label{lem:regular-cycles} There exists a subsequence $\Phi_{k_l}$ such that for almost all $y\in\psi(\overline{\Omega})$ and for any admissible cycle $S\subseteq\psi^{-1}(y)$, the restrictions $\Phi_{k_l}|_{S}$ and $\Phi|_{S}$ are continuous and the sequence $\{\Phi_{k_l}|_{S}\}$ converges to $\Phi|_{S}$ uniformly: $\Phi_{k_l}|_{S}\rightrightarrows\Phi|_{S}$.
\end{lemm}

Below we assume (without loss of generality) that the subsequence $\Phi_{k_l}$ coincides with $\Phi_{k}$. Admissible cycles $S$ satisfying the statement of Lemma \ref{lem:regular-cycles} will be called regular cycles.

We are now ready to construct a compact connected set $A^{j}_{i}\subset\subset\Omega$ (see the comments below \eqref{202309121049}). Denote by $B_{0},\dots,B_{N}$ the elements from $T_{\psi}$ such that $B_j\supseteq\Gamma_j$, $j=0,\dots,N$\footnote{Let $\Omega_j$ be one of the inner holes of the domain $\Omega$. Then it holds that $\overline{\Omega}_j\subseteq B_j$, because we extended $\psi$ into $\Omega_j$ by setting $\psi|_{\Omega_j}\equiv\xi_j$, where $\xi_j=\psi|_{\Gamma_j}=\const$ (see Lemma \ref{lem:stream}).}. From Lemma \ref{lem:Kro} each element $C\in[B_i,B_j]\setminus\{B_i,B_j\}$ is a connected component of some level set of $\psi$ such that the sets $B_i$ and $B_j$ lie in different connected components of $\BR^2\setminus C$. Put
\be
\alpha\equiv\max_{j=0,\dots,M}\min_{C\in[B_j,B_N]}\Phi(C).
\ee
By Lemma \ref{lem:head_conti} and \eqref{202309121048}, the minimum of $\Phi$ on the arc $[B_j,B_N]$ exists and $\alpha<0$. Take a sequence of positive values $t_{i}\in(0,-\alpha)$ such that $t_{i+1}\le\frac{1}{2}t_i$ and the assertion
\be\label{202309161453}
\Phi(C)=-t_{i}\Rightarrow C\text{ is a regular cycle}
\ee 
holds for all $j=0,\dots,M$ and for any $C\in[B_j,B_N]$. The existence of such a sequence $\{t_i\}_{i\in\BN}$ follows from the fact that (see \cite{MR3275850}*{Corollary 3.2})
\be
\SH^{1}\big(\{\Phi(C):C\in[B_j,B_N]\text{ and }C\text{ is not a regular cycle}\}\big)=0,\enskip j=0,\dots,M.
\ee

Define the order on the arc $[B_j,B_N]$ as follows. For some different elements $C',C^{''}\in[B_j,B_N]$, we shall write $C'<C^{''}$ if and only if $C^{''}$ is closer to $B_N$ than $C'$, that is, $C'$ and $B_N$ belong to different connected components of the set $T_{\psi}\setminus\{C^{''}\}$. For $j=0,\dots,M$ and $i\in\BN$, put
\be
A^{j}_{i}\equiv\max\{C\in[B_j,B_N]:\Phi(C)=-t_{i}\}.
\ee
In other words, $A^{j}_{i}$ is the element from the set $\{C\in[B_j,B_N]:\Phi(C)=-t_{i}\}$ which is closest to $B_N$. Since each $A^{j}_{i}$ is a connected component of some level set of $\psi$, the sets $A^{j_1}_{i}$ and $A^{j_2}_{i}$ are either disjoint or identical, that is, if $A^{j_1}_{i}\neq A^{j_2}_{i}$, then $A^{j_1}_{i}\cap A^{j_2}_{i}=\emptyset$. By \eqref{202309161453} and the definition of regular cycles (see the comments below Lemma \ref{lem:regular-cycles}), each $A^{j}_{i}$ is a $C^1$-curve homeomorphic to the unit circle $\BS^{1}$ and $A^{j}_{i}\subset\subset\Omega$. In particular, for any $i\in\BN$, the compact set $\bigcup_{j=0}^{M}A^j_i$ is separated from $\partial\Omega$ and $\dist(\bigcup_{j=0}^{M}A^j_i,\partial\Omega)>0$. For each $i\in\BN$, denote by $V_i$ the connected component of the open set $\Omega\setminus\bigcup_{j=0}^{M}A^j_i$ such that $\Gamma_N\subseteq\partial V_i$. By construction, the sequence of domains $V_i$ is decreasing: $V_{i}\supseteq V_{i+1}$ (since $t_{i+1}\le\frac{1}{2}t_i$) and $\partial V_i$ consists of $A^{j}_{i}$, $j=0,\dots,M$, and several components $\Gamma_{k}$ with $k>M$. Changing the numbering of components $\Gamma_{M+1},\dots,\Gamma_{N}$ if necessary, we can find some $K\in\{M+1,\dots,N\}$ such that $\partial V_{i}=A^0_i\cup\dots\cup A^M_i\cup\Gamma_{K}\cup\dots\cup\Gamma_{N}$ for sufficiently large $i$\footnote{See \cite{MR3275850}*{pp.\,783-784} for more details.}. 

We shall next construct a domain $W_{ik}(t)$ and a level line $S_{ik}(t)\subseteq\{x\in\Omega:\Phi_{k}(x)=-t\}$ for all $t\in[\frac{5}{8}t_i,\frac{7}{8}t_i]$ and for sufficiently large $k\in\BN$ such that (i) $S_{ik}(t)\subseteq\partial W_{ik}(t)$; (ii) The level line $S_{ik}(t)$ separates $\Gamma_{N}$ from $\Gamma_{j}$, $j=0,\dots,M$; (iii) The minimum of $\Phi_k$ in $\overline{W_{ik}(t)}$ is achieved on $S_{ik}(t)$, i.e., $\Phi_k(x)>-t$ for all $x\in \overline{W_{ik}(t)}\setminus S_{ik}(t)$ and $\Phi_k\equiv-t$ on $S_{ik}(t)$. Since $A^j_i$ is a regular cycle, from Lemma \ref{lem:regular-cycles} we have the uniform convergence $\Phi_{k}|_{A^j_i}\rightrightarrows\Phi|_{A^j_i}=-t_i$ as $k\to\infty$. Therefore for any $i\in\BN$ there exists $k_i\in\BN$ such that the inequalities
\be
\Phi_{k}|_{A^j_i}<-\frac{7}{8}t_i,\enskip\Phi_{k}|_{A^j_{i+1}}>-\frac{5}{8}t_i,\enskip j=0,\dots,M
\ee
hold for all $k\ge k_i$ (recall the relation $t_{i+1}\le\frac{1}{2}t_i$). Thus for all $k\ge k_i$ and for all $t\in[\frac{5}{8}t_i,\frac{7}{8}t_i]$ we have
\be
\Phi_{k}|_{A^j_i}<-t,\enskip\Phi_{k}|_{A^j_{i+1}}>-t,\enskip j=0,\dots,M.
\ee

For $k\ge k_i$, $j=0,\dots,M$ and $t\in[\frac{5}{8}t_i,\frac{7}{8}t_i]$, denote by $W^j_{ik}(t)$ the connected component of the open set $\{x\in V_{i}\setminus\overline{V}_{i+1}:\Phi_{k}(x)>-t\}$ such that $\partial W^j_{ik}(t)\supseteq A^j_{i+1}$. Put
\be
W_{ik}(t)\equiv\bigcup_{j=0}^{M}W^j_{ik}(t),\quad S_{ik}(t)\equiv(\partial W_{ik}(t))\cap V_i\setminus\overline{V}_{i+1}.
\ee
By construction, $\Phi_k(x)\equiv-t$ on $S_{ik}(t)$ and
\be\label{202310041719}
\partial W_{ik}(t)=S_{ik}(t)\cup A^0_{i+1}\cup\dots\cup A^{M}_{i+1}.
\ee
Since $\bu_k\in W^{3,2}_{\text{loc}}(\Omega)$ and $p_k\in W^{2,2}_{\text{loc}}(\Omega)$, each $\Phi_k\equiv p_k+\frac{1}{2}\abs{\bu_k}^2$ belongs to $W^{2,2}_{\text{loc}}(\Omega)$. According to the Morse-Sard theorem for Sobolev functions (see Theorem \ref{thm:Morse-Sard}), for almost all $t\in[\frac{5}{8}t_i,\frac{7}{8}t_i]$ the level line $S_{ik}(t)$ consists of finitely many $C^1$-curves homeomorphic to the unit circle $\BS^1$ and $\Phi_k$ is differentiable in the classical sense at each point $x\in S_{ik}(t)$ with $\nabla\Phi_k(x)\neq\boldsymbol{0}$. The values $t\in[\frac{5}{8}t_i,\frac{7}{8}t_i]$ enjoying the above property will be called $(k,i)$-regular. For $(k,i)$-regular value $t\in[\frac{5}{8}t_i,\frac{7}{8}t_i]$, it holds that
\be\label{Hopf}
\int_{S_{ik}(t)}\nabla\Phi_k\cdot\widetilde{\bn}\,ds=-\int_{S_{ik}(t)}\abs{\nabla\Phi_k}\,ds<0,
\ee
where $\widetilde{\bn}$ is the outward (with respect to $W_{ik}(t)$) unit normal to $\partial W_{ik}(t)$. Indeed, the nonzero gradient $\nabla\Phi_k(x)$ at the point $x\in S_{ik}(t)$ is directed inside the domain $W_{ik}(t)$, and hence
\be
\frac{\nabla\Phi_k(x)}{\abs{\nabla\Phi_k(x)}}=-\widetilde{\bn}(x),\enskip x\in S_{ik}(t);
\ee
see Figue 1. Putting $\widetilde{\Gamma}\equiv\Gamma_K\cup\dots\cup\Gamma_N$. For $h>0$, denote $\Gamma_h\equiv\{x\in\Omega:\dist(x,\widetilde{\Gamma})=h\}$ and $\Omega_h\equiv\{x\in\Omega:\dist(x,\widetilde{\Gamma})<h\}$. Since $\Omega$ is of class $C^{\infty}$, there exists a positive constant
\be
\delta_{0}<\frac{1}{2}\min\{\abs{x-y}:x\in\Gamma_i,\,y\in\Gamma_j,\,i,j\in\{0,\dots,N\},\,i\neq j\}
\ee
such that $\dist(\cdot,\widetilde{\Gamma})\in C^{\infty}(\Omega_{\delta_0})$ and $\abs{\nabla\dist(x,\widetilde{\Gamma})}\equiv 1$ on $\Omega_{\delta_0}$; see \cite{MR1814364}*{Lemma 14.16}. By the implicit function theorem, for each $h\le\delta_{0}$ the set $\Gamma_{h}$ is a union of $N-K+1$ $C^{\infty}$-curves homeomorphic to the unit circle $\BS^1$.
\begin{center}
\includegraphics[scale=0.45]{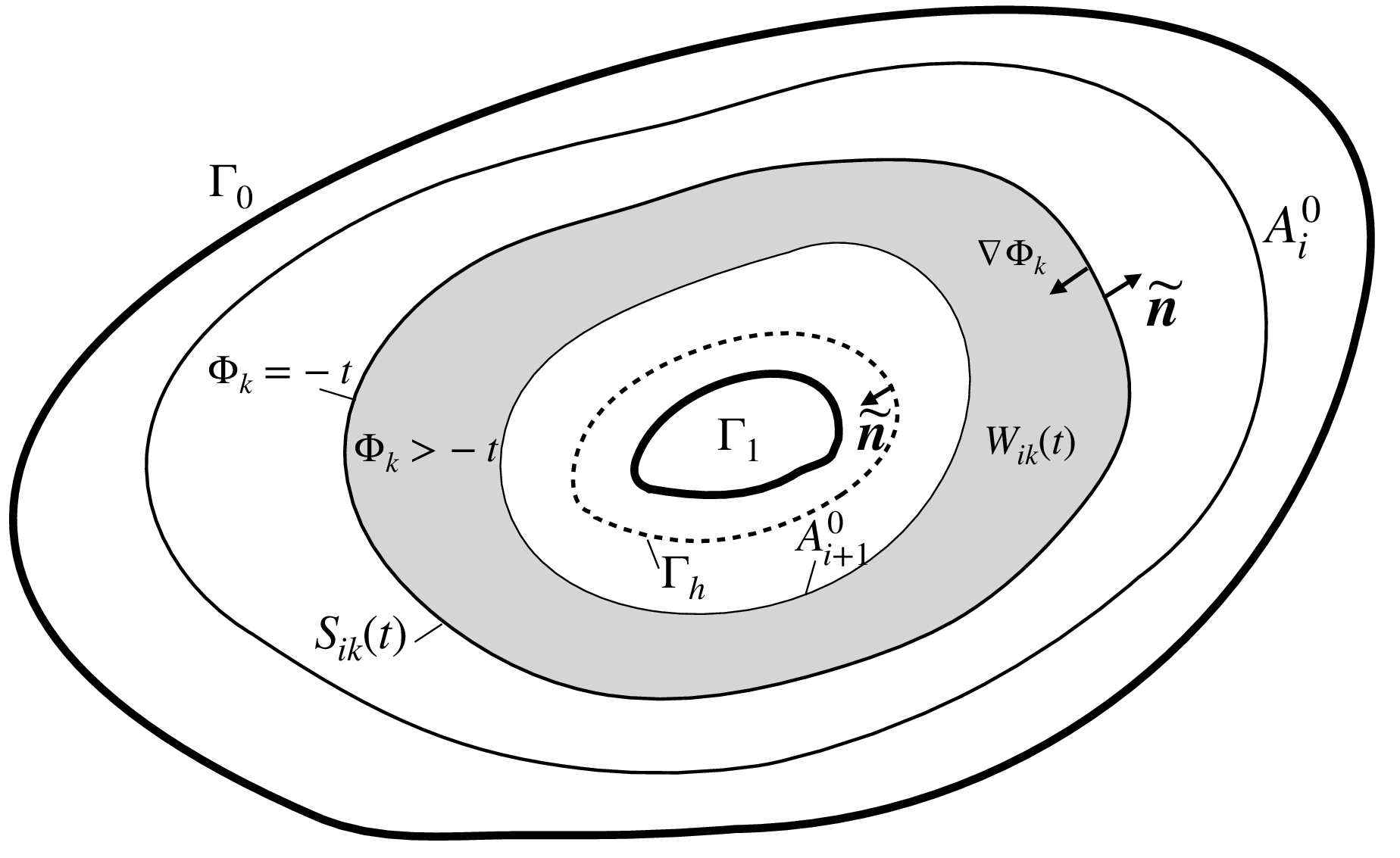}
\end{center}
\begin{center}
Figure 1. The case when $N=1$ and $\widetilde{\Gamma}=\Gamma_1$.
\end{center}
By construction, $\Phi\not\equiv\const$ on $V_i$, and hence the identity \eqref{iden:Bernoulli} implies $\int_{V_i}\omega^2\,dx>0$ for each $i$. Let $\omega_{k}$ be the vorticity of a velocity vector $\bu_k=(u_k^1,u_k^2)$ defined by $\omega_{k}\equiv\partial_2u_k^1-\partial_1u_k^2$. Taking into account the weak convergence $\omega_k\rightharpoonup\omega$ in $L^2(\Omega)$, we have
\begin{lemm}\label{est:vorticity_k} For any $i\in\BN$, there exist some constants $\sigma_i>0$, $\delta_i\in(0,\delta_0)$ and $\tilde{k}_i\in\BN$ such that $\overline{\Omega}_{\delta_i}\subseteq V_{i+1}$ and $\int_{V_{i+1}\setminus\overline{\Omega}_{\delta_i}}\omega_k^2\,dx>\sigma_i$ for all $k\ge\tilde{k}_{i}$.
\end{lemm}
The crucial step in the proof is to establish the following estimate.
\begin{lemm}\label{lem:keyestimate} Assume that the inequality \eqref{theo1_assumption} holds for all $x\in\partial\Omega$. Then for any $i\in\BN$, there exists $k(i)\in\BN$ such that the inequality
\be\label{keyestimate}
\int_{S_{ik}(t)}\abs{\nabla\Phi_k}\,ds<\CF t
\ee
holds for every $k\ge k(i)$ and for almost all $t\in[\frac{5}{8}t_i,\frac{7}{8}t_i]$. Here the constant $\CF$ is independent of $i, k$ and $t$.
\end{lemm}
\begin{proof}
First of all, it is immediately verified that the total head pressures $\Phi_k$ corresponding to the solutions $(\bu_k, p_k)$ to the Navier-Stokes system \eqref{eq:NS_k} satisfy the linear elliptic equation
\be\label{eq:202309111045}
\Delta \Phi_k = \omega_{k}^2+\frac{1}{\nu_k} \dv (\Phi_k \bu_k)-\frac{1}{\nu_k}  \bff_k \cdot \bu_k
\ee
almost everywhere in $\Omega$. For $(k,i)$-regular value $t\in [\frac{5}{8}t_i,\frac{7}{8}t_i]$ and $h\in (0,\delta_i]$ (see Lemma \ref{est:vorticity_k}), consider the domain
\be
\Omega^{h}_{ik}(t)\equiv W_{ik}(t)\cup\overline{V}_{i+1}\setminus \overline{\Omega}_h.
\ee
By construction, $\partial\Omega^{h}_{ik}(t)=S_{ik}(t)\cup\Gamma_{h}$, and denote by $\widetilde{\bn}$ the unit outward (with respect to $\Omega^{h}_{ik}(t)$) normal to $\partial\Omega^{h}_{ik}(t)$. Integrating the equation \eqref{eq:202309111045} over the domain $\Omega^{h}_{ik}(t)$ with $h\in(0,\delta_i]$, we have
\be\label{iden:202309111050}
\ba
&\int_{S_{ik}(t)}\nabla\Phi_k\cdot\widetilde{\bn}\,ds+\int_{\Gamma_{h}}\nabla\Phi_k\cdot\widetilde{\bn}\,ds\\
=&\int_{\Omega^{h}_{ik}(t)}\omega_{k}^2 \,dx+\frac{1}{\nu_k}\int_{S_{ik}(t)}\Phi_k\bu_k\cdot\widetilde{\bn}\,ds+\frac{1}{\nu_k}\int_{\Gamma_{h}}\Phi_k\bu_k\cdot\widetilde{\bn}\,ds-\frac{1}{\nu_k}\int_{\Omega^{h}_{ik}(t)}\bff_k \cdot \bu_k \,dx\\
=&\int_{\Omega^{h}_{ik}(t)}\omega_{k}^2 \,dx-\frac{\lambda_{k}t}{\nu}\int_{\partial\Omega\setminus\widetilde{\Gamma}}a_{\ast}\,ds+\frac{1}{\nu_k}\int_{\Gamma_{h}}\Phi_k\bu_k\cdot\widetilde{\bn}\,ds-\frac{1}{\nu_k}\int_{\Omega^{h}_{ik}(t)}\bff_k \cdot \bu_k \,dx,
\ea
\ee
where we have used the identity $\Phi_k(x)\equiv-t$ on $S_{ik}(t)$. It follows from \eqref{Hopf}, \eqref{iden:202309111050} and the inclusions $V_{i+1}\setminus\overline{\Omega}_{\delta_i}\subseteq\Omega^{h}_{ik}(t)\subseteq\Omega$ that
\be\label{202309111100}
\ba
\int_{S_{ik}(t)}\abs{\nabla\Phi_k}\,ds=&\int_{\Gamma_{h}}\nabla\Phi_k\cdot\widetilde{\bn}\,ds+\frac{\lambda_{k}t}{\nu}\int_{\partial\Omega\setminus\widetilde{\Gamma}}a_{\ast}\,ds-\frac{1}{\nu_k}\int_{\Gamma_{h}}\Phi_k\bu_k\cdot\widetilde{\bn}\,ds\\
&+\frac{1}{\nu_k}\int_{\Omega^{h}_{ik}(t)}\bff_k \cdot \bu_k \,dx-\int_{\Omega^{h}_{ik}(t)}\omega_{k}^2 \,dx\\
\le&\int_{\Gamma_{h}}\nabla\Phi_k\cdot\widetilde{\bn}\,ds+\CF t-\frac{1}{\nu_k}\int_{\Gamma_{h}}\Phi_k\bu_k\cdot\widetilde{\bn}\,ds\\
&+\frac{1}{\nu_k}\norm{\bff_k}_{L^2(\Omega)}\norm{\bu_k}_{L^2(\Omega)}-\int_{V_{i+1}\setminus\overline{\Omega}_{\delta_i}}\omega_{k}^2 \,dx,
\ea
\ee
where $\CF\equiv\frac{1}{\nu}\left|\int_{\partial\Omega\setminus\widetilde{\Gamma}}a_{\ast}\,ds\right|$.
We now show that
\be\label{202309111105}
\lim_{h\to 0}\int_{\Gamma_{h}}\nabla\Phi_k\cdot\widetilde{\bn}\,ds=\int_{\widetilde{\Gamma}}\omega_k\bu_k^{\perp}\cdot{\bn}\,ds,
\ee
where $\bn$ is the unit outward (with respect to $\Omega$) normal to $\partial\Omega$ and $(x_1,x_2)^{\perp}=(-x_2,x_1)$. From $\eqref{eq:NS_k}_{1,2}$ we deduce
\be
\nabla\Phi_k=-\nu_k\nabla^{\perp}\omega_k+\omega_k\bu_k^{\perp}+\bff_k=-\nu_k\nabla^{\perp}\omega_k+\omega_k\bu_k^{\perp}+\frac{\lambda_k\nu_k^2}{\nu^2}\nabla^{\perp}g\quad \text{in }\Omega.
\ee
By the Stokes theorem, for a function $g\in W^{2,2}(\Omega^{h}_{ik}(t),\BR)$ it holds that
\be
\int_{\Gamma_{h}}\nabla^{\perp}g\cdot\widetilde{\bn}=0.
\ee
Consequently, we have 
\be\label{202309111127}
\int_{\Gamma_{h}}\nabla\Phi_k\cdot\widetilde{\bn}\,ds=\int_{\Gamma_{h}}\omega_k\bu_k^{\perp}\cdot\widetilde{\bn}\,ds.
\ee
Since $-\widetilde{\bn}$ gives the unit outward (with respect to $\Omega_{h}$) normal to $\Gamma_h$ (see Figure 1 for the case when $N=1$ and $\widetilde{\Gamma}=\Gamma_{1}$), from \eqref{202309111127} and the Gauss divergence theorem, we observe that
\be\label{}
\ba
\biggr|\int_{\Gamma_{h}}\nabla\Phi_k\cdot\widetilde{\bn}\,ds-\int_{\widetilde{\Gamma}}\omega_k\bu_k^{\perp}\cdot{\bn}\,ds\biggr|&=\biggr|\int_{\Gamma_{h}}\omega_k\bu_k^{\perp}\cdot(-\widetilde{\bn})\,ds+\int_{\widetilde{\Gamma}}\omega_k\bu_k^{\perp}\cdot{\bn}\,ds\biggr|\\
&\le\int_{\Omega_{h}}\abs{\nabla\omega_k\cdot\bu_k^{\perp}}\,dx+\int_{\Omega_{h}}\omega_k^2\,dx\to0\enskip\text{as }h\to0,
\ea
\ee
which proves \eqref{202309111105}. Similarly we have
\be\label{202309111110}
\lim_{h\to 0}\int_{\Gamma_{h}}\Phi_k\bu_k\cdot\widetilde{\bn}\,ds=\int_{\widetilde{\Gamma}}\Phi_k\bu_k\cdot{\bn}\,ds,
\ee
and so letting $h\to 0$ in \eqref{202309111100}\footnote{Since we only know $\Phi_k\in W^{1,2}(\Omega)$, the value of $\nabla\Phi_k$ on $\widetilde{\Gamma}$ cannot be always well-defined. This is the reason why we first integrate the equation \eqref{eq:202309111045} over the domain $\Omega^{h}_{ik}(t)$ with $h\in(0,\delta_i]$, which strictly lies inside $\Omega$. Note that $\Phi_k\in W^{2,2}(\Omega^{h}_{ik}(t))$ by the interior regularity result, and that the value of $\omega_k\bu_k^{\perp}$ on $\widetilde{\Gamma}$ in \eqref{202309111105} is well-defined.}, we deduce
\be\label{202309111115}
\ba
\int_{S_{ik}(t)}\abs{\nabla\Phi_k}\,ds\le&\int_{\widetilde{\Gamma}}\omega_k\bu_k^{\perp}\cdot{\bn}\,ds+\CF t-\frac{1}{\nu_k}\int_{\widetilde{\Gamma}}\Phi_k\bu_k\cdot{\bn}\,ds\\
&+\frac{1}{\nu_k}\norm{\bff_k}_{L^2(\Omega)}\norm{\bu_k}_{L^2(\Omega)}-\int_{V_{i+1}\setminus\overline{\Omega}_{\delta_i}}\omega_{k}^2 \,dx.
\ea
\ee
Let us next show that for any $\epsilon_{0}>0$ there exists $k_0=k_0(\epsilon_{0})\in\BN$ such that the estimates
\be\label{202309111528}
\biggr|\frac{1}{\nu_k}\int_{\widetilde{\Gamma}}\Phi_k\bu_k\cdot{\bn}\,ds\biggr|<\epsilon_{0},
\ee
\be\label{202309111529}
\int_{\widetilde{\Gamma}}\omega_k\bu_k^{\perp}\cdot{\bn}\,ds<\epsilon_{0}
\ee
hold for all $k\ge k_{0}$. The weak convergence $\Phi_k\rightharpoonup \Phi$ in ${W}^{1,q}(\Omega)$, $q\in[1,2)$, implies 
\be\label{202309121054}
\gamma(\Phi_k)\to\gamma(\Phi)\enskip\text{in }L^{q}(\partial\Omega);
\ee
see, e.g., \cite{MR2808162}*{Theorem I\hspace{-1.2pt}I.4.1}, \cite{MR3726909}*{Corollary 18.4 and Exercise 18.6}. Recall that $\bu_k=\frac{\lambda_k\nu_k}{\nu}(\widehat{\bw}_{k}+\bU)$, where $\widehat{\bw}_{k}\in J(\Omega)$ satisfies $\gamma(\widehat{\bw}_{k})\cdot\bn=0$ on $\partial\Omega$ and $\bU\in W^{2,2}(\Omega)$ is a weak solution to the Stokes problem \eqref{S}. In particular, 
\be\label{202309121055}
\bu_k\cdot\bn=\frac{\lambda_k\nu_k}{\nu}\bU\cdot\bn\enskip\text{on }\partial\Omega.
\ee
Since $\gamma(\Phi)=0$ on $\widetilde{\Gamma}$ (see \eqref{202309121049}) and $\norm{\bU}_{C(\overline{\Omega})}\le c\norm{\bU}_{W^{2,2}(\Omega)}<\infty$, it follows from \eqref{202309121054} and \eqref{202309121055} that
\be\label{202409292013}
\frac{1}{\nu_k}\biggr|\int_{\widetilde{\Gamma}}\Phi_k\bu_k\cdot\bn \,ds\biggr|=\frac{\lambda_k}{\nu}\biggr|\int_{\widetilde{\Gamma}}\Phi_k \bU\cdot\bn \,ds\biggr|\le\frac{\norm{\bU}_{C(\overline{\Omega})}}{\nu}\int_{\widetilde{\Gamma}}\abs{\Phi_k} \,ds\to0 \enskip \text{as }k\to\infty,
\ee
which implies \eqref{202309111528}.

In order to show the estimate \eqref{202309111529}, we begin to observe that the pairs $(\bu_k,p_k)\in W^{2,2}(\Omega)\times W^{1,2}(\Omega)$ satisfy the boundary condition $\eqref{eq:NS_k}_4$ in the following sense:
\be\label{202309121218}
[T(\bu_k,p_k)\bn]_{\tau}+\beta_k {\bu_k}_{\tau}=\bb_k\enskip\text{in }L^2(\partial\Omega).
\ee
Indeed, the pairs $(\bu_k,p_k)$ satisfy the integral identity
\be\label{08191037}
\frac{\nu_k}{2}\int_{\Omega}S(\bu_k):S(\bfvarphi)\,dx-\int_{\Omega}p_k\,\dv\bfvarphi\,dx=\int_{\Omega}(-\nu_k\Delta\bu_k+\nabla p_k)\cdot\bfvarphi\,dx+\int_{\partial\Omega}[T(\bu_k,p_k)\bn]_{\tau}\cdot\bfvarphi\,ds
\ee
for every $\bfvarphi\in W^{1,2}(\Omega)$ with $\gamma(\bfvarphi)\cdot\bn=0$ on $\partial\Omega$. Combining the above identity with \eqref{08191350}, and then choosing $\bfvarphi$ from $C^{\infty}_{0}(\Omega)$, we find that $-\nu_k \Delta \bu_k+(\bu_k \cdot \nabla)\bu_k+ \nabla p_k=\bff_k$ for almost all $x\in\Omega$. As a consequence, the identity
\be\label{202308191054}
\int_{\partial\Omega}[T(\bu_k,p_k)\bn]_{\tau}\cdot\bfvarphi\,ds+\int_{\partial \Omega} \beta_k\bu_{k\tau} \cdot \bfvarphi \, ds-\int_{\partial \Omega} \bb_k\cdot \bfvarphi \, ds=0
\ee
holds for all $\bfvarphi\in W^{1,2}(\Omega)$ with $\gamma(\bfvarphi)\cdot\bn=0$ on $\partial\Omega$. In fact, \eqref{202308191054} is valid for any $\bfvarphi\in W^{1,2}(\Omega)$ because $\bb_k\cdot\bn=0$ on $\partial\Omega$. Since $W^{1/2,2}(\partial\Omega)\equiv\gamma(W^{1,2}(\Omega))$ is dense in $L^{2}(\partial\Omega)$, we conclude \eqref{202309121218}. Also, we shall make use of the following identity (see Lemma \ref{lem_def-curl})
\be\label{202309121217}
\curl\bu_k(n_2,-n_1)^\top=[S(\bu_k)\bn]_{\tau}-2\nabla_{\tau}(\bu_k\cdot\bn)-2W^\top\bu_k\enskip\text{in }L^2(\partial\Omega),
\ee
where $W$ is the Weingarten map of $\partial\Omega$ (see Proposition \ref{prop:Weingarten}). Using the identity \eqref{202309121217} along with the slip boundary conditions $\eqref{eq:NS_k}_{3,4}$, we have
\be\label{202309120753}
\ba
&\int_{\widetilde{\Gamma}}\omega_k\bu_k^{\perp}\cdot{\bn}\,ds=\int_{\widetilde{\Gamma}}\omega_k(n_2,-n_1)^\top\cdot\bu_k\,ds\\
=&\int_{\widetilde{\Gamma}}[S(\bu_k)\bn]_{\tau}\cdot\bu_k\,ds-2\int_{\widetilde{\Gamma}}\nabla_{\tau}(\bu_k\cdot\bn)\cdot\bu_k\,ds-2\int_{\widetilde{\Gamma}}W^\top\bu_k\cdot\bu_k\,ds\\
=&\frac{1}{\nu_k}\int_{\widetilde{\Gamma}}\bb_k\cdot\bu_k\,ds-\frac{1}{\nu_k}\int_{\widetilde{\Gamma}}\beta_k{\bu_k}_{\tau}\cdot{\bu_k}_{\tau}\,ds-2\int_{\widetilde{\Gamma}}\nabla_{\tau}a_k\cdot\bu_k\,ds-2\int_{\widetilde{\Gamma}}W^\top\bu_k\cdot\bu_k\,ds.
\ea
\ee
Since $W=W^\top$ on $\partial\Omega$ and the eigenvalues of $W=W(x)$ consist of the curvature $\kappa(x)$ of the boundary at $x\in\partial\Omega$ and $0$, there exists a $2\times2$ orthogonal matrix $P=P(x)$ such that
\be\label{iden:diagonal}
P(x)W(x)P^\top(x)=\mathrm{diag}[\kappa(x),0],\enskip x\in\partial\Omega.
\ee
Putting ${\bv}_k=(v_{k,1},v_{k,2})=P{\bu_k}_{\tau}$. It follows from Proposition \ref{prop:Weingarten} and \eqref{iden:diagonal} that
\be\label{202309120754}
\ba
W^\top(x)\bu_k(x)\cdot\bu_k(x)&={\bu_k}^\top (W\bu_k)={\bu_k}_{\tau}^{\top}W{\bu_k}_{\tau}={\bu_k}_{\tau}^{\top}P^{\top}(PWP^{\top})P{\bu_k}_{\tau}\\
&={\bv_k}^{\top}\mathrm{diag}[\kappa(x),0]{\bv}_k=\kappa(x){v}_{k,1}(x)^2\enskip\text{at }x\in\partial\Omega.
\ea
\ee
Recall that $\beta_k\equiv\frac{\nu_k}{\nu}\beta$, $a_k\equiv\frac{\lambda_k\nu_k}{\nu}a_{\ast}$, $\bb_k\equiv\frac{\lambda_k\nu^2_k}{\nu^2}\bb_{\ast}$ and the sequence $\{\bu_k\}$ is uniformly bounded in $W^{1,2}(\Omega)$. Using the equalities $\abs{\bv_k}^{2}=\abs{P{\bu_k}_{\tau}}^{2}=\abs{{\bu_k}_{\tau}}^{2}$ and \eqref{202309120754} in \eqref{202309120753}, and then recalling the inequality \eqref{theo1_assumption}, we deduce that
\be\label{202309120755}
\ba
\int_{\widetilde{\Gamma}}\omega_k\bu_k^{\perp}\cdot{\bn}\,ds=&\frac{\lambda_k\nu_k}{\nu^2}\int_{\widetilde{\Gamma}}\bb_{\ast}\cdot\bu_{k}\,ds-\frac{1}{\nu}\int_{\widetilde{\Gamma}}\beta\abs{{\bv_k}}^2\,ds\\
&-\frac{2\lambda_k\nu_k}{\nu}\int_{\widetilde{\Gamma}}(\nabla_{\tau}a_{\ast})\cdot\bu_k\,ds-2\int_{\widetilde{\Gamma}}\kappa (v_{k,1})^2\,ds\\
\le&\frac{\lambda_k\nu_k}{\nu^2}\int_{\widetilde{\Gamma}}\bb_{\ast}\cdot\bu_{k}\,ds-\frac{2\lambda_k\nu_k}{\nu}\int_{\widetilde{\Gamma}}(\nabla_{\tau}a_{\ast})\cdot\bu_k\,ds-\int_{\widetilde{\Gamma}}\left(\frac{\beta}{\nu}+2\kappa\right)(v_{k,1})^2\,ds\\
\le&\frac{\nu_k}{\nu^2}\norm{\bb_{\ast}}_{L^2(\partial\Omega)}\norm{\bu_k}_{L^2(\partial\Omega)}+\frac{2\nu_k}{\nu}\norm{\nabla_{\tau}a_{\ast}}_{L^2(\partial\Omega)}\norm{\bu_k}_{L^2(\partial\Omega)}\\
\le&c(\Omega,\nu)\sup_{k}\norm{\bu_k}_{W^{1,2}(\Omega)}(\norm{\bb_{\ast}}_{L^2(\partial\Omega)}+\norm{a_{\ast}}_{W^{3/2,2}(\partial\Omega)})\nu_k\to 0\enskip\text{as }k\to\infty,
\ea
\ee
which proves \eqref{202309111529}.

By definition, $\frac{1}{\nu_k}\norm{\bff_k}_{L^2(\Omega)}=\frac{\lambda_k\nu_k}{\nu^2}\norm{\bff}_{L^2(\Omega)}\to 0$ as $k \to \infty$. Since the sequence $\{\bu_k\}$ is uniformly bounded in $W^{1,2}(\Omega)$, for any $\epsilon_{1}>0$ there exists a constant $k_1=k_1(\epsilon_{1})\in\BN$ such that the inequality
\be\label{202309121447}
\frac{1}{\nu_k}\norm{\bff_k}_{L^2(\Omega)}\norm{\bu_k}_{L^2(\Omega)}<\epsilon_{1}
\ee
holds for all $k\ge k_1$. Collecting \eqref{202309111115}, \eqref{202309111528}, \eqref{202309111529}, \eqref{202309121447} and the estimate in Lemma \ref{est:vorticity_k} we find that the inequality
\be
\int_{S_{ik}(t)}\abs{\nabla\Phi_k}\,ds\le\CF t+2\epsilon_{0}+\epsilon_{1}-\sigma_{i}
\ee
holds for all $k\ge\max\{k_0,k_1,\tilde{k}_{i}\}$. Thus, by taking $\epsilon_{0}=\frac{\sigma_{i}}{8}$ and $\epsilon_{1}=\frac{\sigma_{i}}{4}$, we obtain the desired estimate \eqref{keyestimate}.
\end{proof}
Finally, with the estimate \eqref{keyestimate} in hand, we can derive the desired contradiction by using exactly the same argument employed in \cite{MR3275850}*{Lemma 3.9}. For the sake of completeness, we reproduce the argument here. For $i\in\BN$ and $k\ge k(i)$ (see Lemma \ref{lem:keyestimate}), put
\be
E_{ik}\equiv\bigcup\limits_{t\in [\frac{5}{8}t_i,\frac{7}{8}t_i]} S_{ik}(t).
\ee
By the co-area formula (see \cite{MR1932709}) for each integrable function $g\colon\Omega\to\BR$, the equality
\be\label{co-area}
\int_{E_{ik}}g\abs{\nabla\Phi_{k}}\,dx=\int_{\frac{5}{8}t_i}^{\frac{7}{8}t_i}\int_{S_{ik}(t)}g(x)\,d\SH^{1}(x)\,dt
\ee
holds. In particular, choosing $g=\abs{\nabla\Phi_k}$ and using \eqref{keyestimate}, we have
\be
\int_{E_{ik}}\abs{\nabla\Phi_{k}}^2\,dx=\int_{\frac{5}{8}t_i}^{\frac{7}{8}t_i}\int_{S_{ik}(t)}\abs{\nabla\Phi_{k}}\,d\SH^{1}(x)\,dt\le\CF\int_{\frac{5}{8}t_i}^{\frac{7}{8}t_i}t\,dt=\CF't_{i}^2,
\ee
where $\CF'\equiv\frac{3}{16}\CF$ is independent of $i$. Next, taking $g=1$ in \eqref{co-area} and employing the Schwarz inequality, we obtain
\be\label{202309131619}
\ba
\int_{\frac{5}{8}t_i}^{\frac{7}{8}t_i}\SH^{1}(S_{ik}(t))\,dt=\int_{E_{ik}}\abs{\nabla\Phi_k}\,dx&\le\biggl(\int_{E_{ik}}\abs{\nabla\Phi_k}^2\,dx\biggr)^{\frac{1}{2}}\SL^{2}(E_{ik})^{\frac{1}{2}}\\
&\le\sqrt{\CF'}t_i\SL^{2}(E_{ik})^{\frac{1}{2}}.
\ea
\ee
By construction, for almost all $t\in [\frac{5}{8}t_i,\frac{7}{8}t_i]$, the set $S_{ik}(t)$ consists of finitely many $C^{1}$-cycles and $S_{ik}(t)$ separates $A^j_i$ from $A^j_{i+1}$ , $j=0,\dots,M$. Thus, each set $S_{ik}(t)$ separates $\Gamma_j$ from $\Gamma_N$, and we have
\be
\SH^{1}(S_{ik}(t))\ge\min\bigl(\diam(\Gamma_j),\diam(\Gamma_N)\bigr)\equiv C^{\ast}.
\ee
Then \eqref{202309131619} yields
\be\label{202309131625}
\frac{1}{4}C^{\ast}\le\sqrt{\CF'}\SL^{2}(E_{ik})^{\frac{1}{2}}.
\ee
Recall that $S_{ik}(t)\subseteq V_{i}\setminus\overline{V}_{i+1}$ and that the sequence of domains $V_{i}$ is decreasing, i.e., $V_{i}\supseteq V_{i+1}$. Thus we observe that $\SL^{2}(E_{ik})\le\SL^{2}(V_{i}\setminus\overline{V}_{i+1})\to0$ as $i\to\infty$, and hence \eqref{202309131625} gives the desired contradiction. 
\subsubsection{The maximum of $\Phi$ is not attained on the boundary $\partial\Omega$}\label{subsub2}
Let us next consider the case (b), when \eqref{case_b} holds. In this case, we can find the regular cycle $F\subseteq\Omega$ such that $\Phi(F)>\max\limits_{j=0,\dots,N}\widehat\Phi_{j}$ (\cite{MR3275850}*{Lemma 3.10}). Adding a constant to the pressure, we may assume that $\Phi(F)=0$. The rest of the proof for the case (b) proceeds in a similar way to that presented in Subsection \ref{subsub1}. The differences are as follows: now $M=N$, the set $F$ plays the role of the component $\Gamma_{N}$ in the above argument and the calculations become much simpler since $F$ lies strictly inside $\Omega$. See \cite{MR3275850}*{Subsection 3.3.2} for a detailed procedure which leads to a contradiction.

\vspace{0.3cm}
The contradictions in both cases (a) and (b) are due to the assumption that the norms of all possible solutions $\bw^{(\lambda)}$ to the integral identity \eqref{eq:Leray-Schauder_lambda} are not uniformly bounded in the space $J(\Omega)$ with respect to $\lambda\in[0,1]$. Thus the desired uniform estimate is proved and the Leray-Schauder theorem gives the existence of a solution $\bw\in J(\Omega)$ to the equation \eqref{def:weakNShomo}. The proof of Theorem \ref{theo1} is therefore completed.
\subsection{Uniqueness and regularity of weak solutions}\label{subsec:rem_uni_and_reg}
We end this section with some results on uniqueness and regularity of weak solutions to the Navier-Stokes problem \eqref{NS}-\eqref{slip}. Example \ref{exam:Hamel} implies that weak solutions to the problem \eqref{NS}-\eqref{slip} are not unique in general. On the other hand, we can prove the uniqueness of weak solutions for small data.
\begin{theo}
Let $\Omega\subseteq\BR^2$ be a smooth bounded domain, and let $\beta\in C(\partial\Omega)$ be nonnegative. Suppose that $a_{\ast}\in W^{1/2,2}(\partial\Omega)$ and $\bb_{\ast}\in W^{-1/2,2}(\partial\Omega)$ satisfy the compatibility conditions \eqref{theo1_comp} and $\bff\in H(\Omega)'$. Assume further that $\beta\not\equiv0$ in the case when $\Omega$ possesses a circular symmetry around some point. Then there exists a constant $c=c(\Omega,\nu,\beta)>0$ such that if
\be
\norm{a_{\ast}}_{W^{1/2,2}(\partial\Omega)}+\norm{a_{\ast}}_{W^{1/2,2}(\partial\Omega)}^2+\norm{\bff}_{H(\Omega)'}+\norm{\bb_{\ast}}_{W^{-1/2,2}(\partial\Omega)}<c,
\ee
the Navier-Stokes problem \eqref{NS}-\eqref{slip} has no more than one weak solution in the class $W^{1,2}(\Omega)$.
\end{theo} 
The proof of this theorem is entirely analogous to \cite{MR4231512}*{Proposition 8.4} or \cite{MR2808162}*{Theorems I\hspace{-1.2pt}X.2.1, I\hspace{-1.2pt}X.4.1 and I\hspace{-1.2pt}X.4.2}, and it is therefore omitted.
\begin{rema} Suppose that $\beta\equiv0$ and $H(\Omega)\cap\SR(\Omega)$\footnote{Recall that $\SR(\Omega)$ is the set of rigid displacements (see Remark \ref{rem:Korn} (i)).} is nontrivial (the domain $\Omega$ has a circular symmetry around some point). In this case, weak solutions to the Navier-Stokes problem \eqref{NS}-\eqref{slip} are not unique in the class $W^{1,2}(\Omega)$ even if the given data are small. Indeed, any vector field $\bv$ from $H(\Omega)\cap\SR(\Omega)$ and the scalar function $\pi=\frac{\abs{\bv}^2}{2}$ corresponding to $\bv$ satisfy the Navier-Stokes equations \eqref{NS}-\eqref{slip} with $\bff\equiv\bb_{\ast}\equiv\boldsymbol{0}$ and $a_{\ast}\equiv\beta\equiv0$.
\end{rema}
Finally, as in the case of the Stokes problem, we have the following regularity result for the Navier-Stokes problem.
\begin{theo}\label{thm:grNS} Let $\Omega\subseteq\BR^2$ be a smooth bounded domain, and let $\beta\in C^1(\partial\Omega)$ be nonnegative. Suppose that $\bu$ is a weak solution to the Navier-Stokes problem \eqref{NS}-\eqref{slip} corresponding to
\be
\bff\in L^{2}(\Omega),\;a_{\ast}\in W^{3/2,2}(\partial\Omega)\;\text{and }\bb_{\ast}\in W^{1/2,2}(\partial\Omega).
\ee
Denote by $p$ the pressure associated with $\bu$ in Remark $\ref{rem:def_weak}$ $(\mathrm{i})$. Then, $(\bu,p)\in W^{2,2}(\Omega)\times W^{1,2}(\Omega)$. Furthermore, the pair $(\bu,p)$ satisfies the second boundary condition $\eqref{slip}_2$ in the following sense:
\be
[T(\bu,p)\bn]_{\tau}+\beta \bu_{\tau}=\bb_{\ast}\enskip\text{in }L^2(\partial\Omega).
\ee
\begin{proof} The first part of the theorem can be shown in exactly the same way as in the proof of the regularity properties for $\widehat{\bu}_k$ and $\widehat{p}_k$ given in Subsection \ref{Leray_arg}. The second part has already been proved in the proof of Lemma \ref{lem:keyestimate}; see the arguments below \eqref{202409292013}.  
\end{proof}
\end{theo}
\section
{Proof of Theorem \ref{theo2}}\label{Proof of Theorem 2}
According to the proof of Theorem \ref{theo1}, it suffices to show that Lemma \ref{lem:keyestimate}, more specifically, the estimate \eqref{202309111529} holds under the assumption of Theorem \ref{theo2}. From \eqref{202309101143} and the identity $\int_{\Gamma_1}a_{\ast}\,ds=-\int_{\Gamma_0}a_{\ast}\,ds$, we have
\be\label{202309221134}
\ba
\frac{1}{\lambda_0}&=-\widehat\Phi_0\int_{\Gamma_0}a_{\ast}\,ds-\widehat\Phi_1\int_{\Gamma_1}a_{\ast}\,ds\\
&=(\widehat\Phi_1-\widehat\Phi_0)\int_{\Gamma_0}a_{\ast}\,ds,
\ea
\ee
and hence the assumption $\int_{\Gamma_0}a_{\ast}\,ds>0$\footnote{If $\int_{\Gamma_0}a_{\ast}\,ds=0$, then we derive the desired contradiction immediately from the identity \eqref{202309221134}. Thus we consider only the case when $\int_{\Gamma_0}a_{\ast}\,ds>0$.} implies $\widehat\Phi_1>\widehat\Phi_0$ (recall that $0<\lambda_0$). Suppose now that \eqref{case_a} holds, i.e., $\widehat\Phi_1=\max\limits_{j=0,1} \widehat\Phi_{j} \ge \esssup\limits_{x\in\Omega}\Phi(x)$ (the maximum of $\Phi$ is attained on the inner boundary $\Gamma_1$). Proceeding exactly as in the proof of Lemma \ref{lem:keyestimate}, we arrive at the following identity (see \eqref{202309120755}):
\be\label{202311131230}
\int_{\widetilde{\Gamma}}\omega_k\bu_k^{\perp}\cdot{\bn}\,ds=\frac{\lambda_k\nu_k}{\nu^2}\int_{\widetilde{\Gamma}}\bb_{\ast}\cdot\bu_{k}\,ds-\frac{2\lambda_k\nu_k}{\nu}\int_{\widetilde{\Gamma}}(\nabla_{\tau}a_{\ast})\cdot\bu_k\,ds-\frac{1}{\nu}\int_{\widetilde{\Gamma}}\beta\abs{{\bu_k}_{\tau}}^2\,ds-2\int_{\widetilde{\Gamma}}\kappa (v_{k,1})^2\,ds.
\ee
Here $\widetilde{\Gamma}=\Gamma_{1}$ (the inner boundary). By assumption, the inner hole $\Omega_1$ of the domain $\Omega$ is convex, and so the curvature $\kappa$ of the inner boundary $\Gamma_1$ is nonnegative. Moreover, since the friction coefficient $\beta\ge0$, we deduce 
\be
\int_{\Gamma_1}\omega_k\bu_k^{\perp}\cdot{\bn}\,ds\le\frac{\lambda_k\nu_k}{\nu^2}\int_{\widetilde{\Gamma}}\bb_{\ast}\cdot\bu_{k}\,ds-\frac{2\lambda_k\nu_k}{\nu}\int_{\widetilde{\Gamma}}(\nabla_{\tau}a_{\ast})\cdot\bu_k\,ds\to0\enskip\text{as}\enskip k\to\infty,
\ee
which implies \eqref{202309111529}. The rest of the proof is completely the same as that of Theorem \ref{theo1}, and so the proof of Theorem \ref{theo2} is complete.

\begin{rema}\label{rem:inflow}In the case when \eqref{case_a} holds with $\widehat\Phi_1<\widehat\Phi_0$ ($\Phi$ attains its maximum on the outer boundary $\Gamma_0$), one cannot prove the estimate \eqref{202309111529} in the same way as above. To fix the ideas, suppose that $\Omega=\{x=(x_1,x_2)\in\BR^2:R_1<\abs{x}<R_0\}$, $R_1>0$, $\Gamma_{j}=\{x\in\BR^2:\abs{x}=R_j\}$, $j=0,1$. In such a situation, $\kappa(x)=\frac{1}{R_1}$ if $x\in\Gamma_1$ and $\kappa(x)=-\frac{1}{R_0}$ if $x\in\Gamma_0$, and one need to evaluate the right-hand side of \eqref{202311131230} with $\widetilde{\Gamma}=\Gamma_{0}$ (the outer boundary). However, the last two terms
\be
-\frac{1}{\nu}\int_{\Gamma_0}\beta\abs{{\bu_k}_{\tau}}^2\,ds+\frac{2}{R_0}\int_{\Gamma_0}({v}_{k,1})^2\,ds
\ee
may not be handled without restrictions on the friction coefficient $\beta$.
\end{rema}

\section
{Proof of Theorem \ref{theo4} and Corollary \ref{cor:harmonic}}\label{Proof of Theorem 4}
In this section we prove Theorem \ref{theo4} and Corollary \ref{cor:harmonic}.
\begin{proof}[Proof of Theorem $\ref{theo4}$] Denote by $\widetilde{\bU}\in W^{1,2}(\Omega)$ the weak solution to the Stokes problem \eqref{S} corresponding to the data $\bff, \beta, a_{\ast}$ and $\bb_{\ast}$ given in the statement of Theorem \ref{theo4}. Repeating the same contradiction argument presented in Subsection \ref{Leray_arg} with $\widetilde{\bU}$ in place of $\bU$, we can find the pair $(\bv,p)\in J(\Omega)\times W^{1,3/2}(\Omega)$ of a solution to the Euler system \eqref{Euler} and some positive number $\lambda_{0}\in (0,1]$ satisfying the identity \eqref{08251507}. Recalling that the vector field $\bv$ is a weak limit of some sequence $\{\bw_k\}_{k\in\BN}\subseteq J(\Omega)$ with $\norm{\bw_k}_{J(\Omega)}=1$ (see the comment below \eqref{202312281626}), we have
\be\label{202312291110}
\norm{\bv}_{J(\Omega)}\equiv\left(\frac{\nu}{2}\int_{\Omega}S(\bv):S(\bv)\,dx+\int_{\partial\Omega}\beta\bv_{\tau}\cdot\bv_{\tau}\,ds\right)^{\frac{1}{2}}\le1.
\ee
Since $\dv\widetilde{\bU}=0$ in $\Omega$ and $\widetilde{\bU}\cdot\bn=a_{\ast}$ on $\partial\Omega$, from \eqref{08251507} we have
\be\label{202410181029}
\frac{1}{\lambda_0}=-\int_{\partial\Omega}\Phi\widetilde{\bU}\cdot\bn\,ds=-\int_{\Omega}\dv(\Phi\widetilde{\bU})\,dx=-\int_{\Omega}\widetilde{\bU}\cdot\nabla\Phi\,dx.
\ee
According to Proposition \ref{prop:harmonic} (ii), the solenoidal extension $\widetilde{\bU}\in W^{1,2}(\Omega)$ of $a_{\ast}$ is decomposed as $\widetilde{\bU}=\bh+\nabla^{\perp}w$ in $\Omega$, where $\bh\in\tilde{V}_{\text{har}}(\Omega)$ and $w\in W^{2,2}(\Omega,\BR)$. By integration by parts\footnote{Since $\nabla\Phi\in L^{3/2}(\Omega)$ and $\curl\nabla\Phi(=0)\in L^{3/2}(\Omega)$, the tangential component $\nabla\Phi\wedge\bn|_{\partial\Omega}$ is well defined as a member of $W^{1-1/3,3}(\partial\Omega)'=W^{-2/3,3/2}(\partial\Omega)$, and the integration by parts is justified in the sense of the generalized Stokes formula; see, e.g., \cite{MR521262}*{Chapter V\hspace{-1.2pt}I\hspace{-1.2pt}I, Lemma 4.2}, \cite{MR2542982}*{(2.2)}.}, we observe that
\be\label{202312291113}
\ba
\int_{\Omega}\widetilde{\bU}\cdot\nabla\Phi\,dx&=\int_{\Omega}\bh\cdot\nabla\Phi\,dx+\int_{\Omega}\nabla^{\perp}w\cdot\nabla\Phi\,dx\\
&=\int_{\Omega}\bh\cdot\nabla\Phi\,dx+\int_{\Omega}(\curl\nabla\Phi)w\,dx+\int_{\partial\Omega}\left(\nabla\Phi\wedge\bn\right)w\,ds\\
&=\int_{\Omega}\bh\cdot\nabla\Phi\,dx,
\ea
\ee
because $\nabla\Phi$ is parallel to $\bn$ on $\partial\Omega$ (see Corollary \ref{cor:Bernoulli}). Consequently, under the assumption \eqref{theo4_assumption}, from \eqref{202410181029}, \eqref{202312291113}, \eqref{iden:Bernoulli} and the Korn-type inequality \eqref{theo4_Korn} we deduce (notice that $\bv\in E(\Omega)$)
\be
1=-\lambda_{0}\int_{\Omega}\widetilde{\bU}\cdot\nabla\Phi\,dx=-\lambda_{0}\int_{\Omega}\bh\cdot\curl\bv(-v_2,v_1)\,dx<\lambda_{0}\norm{\bv}_{W^{1,2}(\Omega)}^2\frac{\nu}{2}K(\Omega,\frac{2\beta}{\nu})^{-1}\le\norm{\bv}_{J(\Omega)}^2,
\ee
which contradicts \eqref{202312291110}. The proof of Theorem \ref{theo4} is therefore completed.
\end{proof}
\begin{proof}[Proof of Corollary $\ref{cor:harmonic}$] By the H\"{o}lder inequality and the Sobolev embedding theorem, we have
\be
-\int_{\Omega}\bh\cdot\curl\bz(-z_2,z_1)\,dx\le\norm{\bh}_{L^q(\Omega)}\norm{\curl\bz}_{L^2(\Omega)}\norm{\bz}_{L^{2q/(q-2)}(\Omega)}\le\sqrt{2}C_{2q/(q-2)}\norm{\bh}_{L^q(\Omega)}\norm{\bz}_{W^{1,2}(\Omega)}^2
\ee
for all $\bz\in E(\Omega)$. Here $C_{2q/(q-2)}$ is the best constant of the Sobolev inequality \eqref{theo4_Sobolev} for $r=2q/(q-2)$. Then the above inequality and the assumption \eqref{202312201219} yield the estimate \eqref{theo4_assumption}.
\end{proof}

\section
{Proof of Theorem \ref{theo3}}\label{Proof of Theorem 3}
Let $\Omega\subseteq\BR^2$ be an admissible domain (see Definition \ref{def:admissible}). Since the boundary datum $a_{\ast}^{S}\in W^{1/2,2}(\partial\Omega)$ satisfies
\be\label{202401051721}
\int_{\partial\Omega} a_{\ast}^{S}\,ds=\sum_{j=0}^{N} \int_{\Gamma_{j}} a_{\ast}^{S} \, ds=0,
\ee
it follows from Lemma \ref{solenoidalextension} that there exists a solenoidal extension $\bA=(A_1,A_2)\in W^{1,2}(\Omega)$ of $a_{\ast}^{S}$ such that $\bA\cdot\bn=a_{\ast}^{S}$ on $\partial\Omega$ and $\norm{\bA}_{W^{1,2}(\Omega)}\le c\norm{a_{\ast}^{S}}_{W^{1/2,2}(\partial\Omega)}$. Note that $\bA$ is not necessarily symmetric. Define the symmetric vector field $\bA^{S}=(A_1^{S},A_2^{S})$ with components
\be\label{sym_part_A}
A_1^{S}(x_1,x_2)\equiv\frac{1}{2}(A_1(x_1,x_2)+A_1(x_1,-x_2)),\enskip A_2^{S}(x_1,x_2)\equiv\frac{1}{2}(A_2(x_1,x_2)-A_2(x_1,-x_2)).
\ee
Since the boundary datum $a_{\ast}^{S}$ and the unit outer normal $\bn$ to $\partial\Omega$ are symmetric with respect to the $x_1$-axis, we find that
\be
\dv\bA^S=0\enskip\text{in }\Omega\enskip\text{and}\enskip\bA^S\cdot\bn=a_{\ast}^{S}\enskip\text{on}\enskip\partial\Omega.
\ee
We then look for a symmetric solution $\bu^S=\bw+\bA^S$ where $\bw\in J^S(\Omega)$ satisfies the integral identity
\be\label{def:weakNS_symmetric}
\ba
& \frac{\nu}{2} \int_{\Omega} S(\bw):S(\bfvarphi) \, dx+\int_{\partial \Omega} \beta^{S} \bw_{\tau} \cdot \bfvarphi_{\tau} \, ds=-\frac{\nu}{2} \int_{\Omega} S(\bA^S):S(\bfvarphi) \, dx-\int_{\partial \Omega} \beta \bA^S \cdot \bfvarphi \, ds\\
&- \int_{\Omega} (\bA^S \cdot \nabla)\bA^S \cdot \bfvarphi \,dx- \int_{\Omega} (\bA^S \cdot \nabla)\bw \cdot \bfvarphi \, dx - \int_{\Omega} (\bw \cdot \nabla)\bw \cdot \bfvarphi \, dx- \int_{\Omega} (\bw \cdot \nabla)\bA^S \cdot \bfvarphi \,dx\\
&+\int_{\Omega} \bff^{S}\cdot\bfvarphi \, dx+\int_{\partial\Omega} \bb_{\ast}^{S}\cdot\bfvarphi \, ds
\ea
\ee
for all $\bfvarphi \in J^S(\Omega)$. Here $J^S(\Omega)\equiv\{\bfvarphi\in J(\Omega):\bfvarphi\colon\Omega\to\BR^2\enskip\text{is symmetric with respect to the $x_1$-axis}\}$.\footnote{If the integral identity \eqref{def:weakNS_symmetric} holds for all $\bfvarphi\in J^{S}(\Omega)$, then it is also valid for all $\bfvarphi\in J(\Omega)$. This can be verified by decomposing $\bfvarphi\in J(\Omega)$ into the sum of its symmetric and anti-symmetric parts as $\bfvarphi=\bfvarphi^{S}+\bfvarphi^{A}$, and then using the fact that the $L^2(\Omega)$- and $L^2(\partial\Omega)$-scalar products of a symmetric element and an anti-symmetric one always vanish. Here $\bfvarphi^{S}$ is the symmetric part of $\bfvarphi$ defined as in \eqref{sym_part_A}, and $\bfvarphi^{A}\equiv\bfvarphi-\bfvarphi^{S}$ is the anti-symmetric part of $\bfvarphi$.} Note that, since $J^S(\Omega)\cap\SR(\Omega)=\{\boldsymbol{0}\}$ for any admissible domain $\Omega$, the bilinear form
\be
(\bw,\bfvarphi)_{J^S(\Omega)}=\frac{\nu}{2} \int_{\Omega} S(\bw):S(\bfvarphi) \, dx+\int_{\partial \Omega} \beta^{S} \bw_{\tau} \cdot \bfvarphi_{\tau} \, ds,\enskip\bw,\bfvarphi\in J^S(\Omega),
\ee
defines a scalar product on the space $J^S(\Omega)$ even in the case when $\beta^S\equiv0$ and the domain $\Omega$ is an annulus; see Remark \ref{rem:Korn} (ii). Repeating the arguments employed in Subsections \ref{Leray_arg} and \ref{Leray_arg_Euler}, one can show the existence of $\bw\in J^S(\Omega)$ satisfying \eqref{def:weakNS_symmetric} provided that the following requirements on the pair $(\bv,p)\in J^S(\Omega)\times W^{1,q}(\Omega)$, $q\in [1,2)$, are incompatible (see \eqref{Euler} and \eqref{202309101143})
\begin{align} \label{Euler_symmetric}
\left\{
\ba
(\bv\cdot\nabla)\bv+\nabla p & =\boldsymbol{0} \quad & & \text{in} \enskip \Omega,\\
\dv \bv & =0 \quad & & \text{in} \enskip \Omega,\\
\bv \cdot\bn & = 0 \quad & & \text{on} \enskip \partial \Omega,
\ea 
\right. 
\end{align}
and
\be\label{202309101143_symmetric}
\frac{1}{\lambda_0}=-\int_{\partial\Omega}\Phi a_{\ast}^{S}\,ds=-\sum_{j=0}^{N}\widehat\Phi_j\int_{\Gamma_j}a_{\ast}^{S}\,ds,
\ee
where $\lambda_0\in (0,1]$ is some constant, $\Phi\equiv p+\frac{\abs{\bv}^2}{2}$ is the total head pressure and $\widehat\Phi_j=\gamma(\Phi)|_{\Gamma_j}=\const$ (see Corollary \ref{cor:Bernoulli}). The following theorem can be regarded as a generalization of \cite{MR763943}*{Theorem 2.3}.
\begin{theo}\label{thm:Amick_symmetric} Let $\Omega\subseteq\BR^2$ be an admissible domain and let the pair $(\bv,p)\in J^S(\Omega)\times W^{1,q}(\Omega)$, $q\in[1,2)$, satisfy the Euler system \eqref{Euler_symmetric}. Then
\be
\widehat\Phi_0=\widehat\Phi_1=\cdots=\widehat\Phi_N,
\ee
where $\widehat\Phi_j=\gamma(\Phi)|_{\Gamma_j}=\const$.
\end{theo}
\begin{proof}[Proof of Theorem $\ref{thm:Amick_symmetric}$]
The proof goes exactly as in \cite{MR763943}*{Theorem 2.3}\footnote{A slight modification would be made for an extension of the function $v_2\in W^{1,2}(\Omega,\BR)$ to the whole of $\BR^2$.}. Recall that $\Gamma_{0}$ is the outer boundary. By the assumption, the set $\{x_{2}=0\}\cap\Gamma_{j}$ consists of two points $(\Fa_{j},0)$ and $(\Fb_{j},0)$ with $\Fa_{j}<\Fb_{j}$. We may label the components so that $\Fa_{0}<\Fa_{1}<\Fb_{1}<\cdots<\Fa_{N}<\Fb_{N}<\Fb_{0}$. Now we shall prove that $\widehat\Phi_0=\widehat\Phi_1$. Near the point $(\Fa_{0},0)$, the component $\Gamma_{0}$ can be represented by an equation $x_{1}=\zeta_{0}(x_{2})\in C^{\infty}$. Similarly, $\Gamma_{1}$ is given by an equation $x_{1}=\zeta_{1}(x_{2})\in C^{\infty}$ near the point $(\Fa_{1},0)$. We take $\delta>0$ sufficiently small so that $A(\delta)\equiv\{(x_1,x_2); x_1\in(\zeta_{0}(x_2), \zeta_{1}(x_2)), x_2\in(0,\delta)\}\subseteq\Omega$. Recall the identity
\be
\nabla\Phi=(\partial_2 v_1-\partial_1 v_2)(-v_2,v_1)=\omega\nabla\psi.
\ee
Integrating the equation $\partial_{1}\Phi=-\omega v_2$ over $A(\delta)$, we obtain
\be
\int_{0}^{\delta} \{\Phi(\zeta_{1}(x_2),x_2)-\Phi(\zeta_{0}(x_2),x_2)\} \,dx_2=-\int_{A(\delta)} \omega v_2 \, dx,
\ee 
which implies
\be\label{est:05090909}
\abs{\widehat\Phi_1-\widehat\Phi_0}\le\frac{1}{\delta}\biggr|\int_{A(\delta)} \omega v_2 \, dx\biggr|.
\ee
According to the Sobolev extension theorem (see, e.g., \cite{MR3409135}*{Theorems 4.7}), there exists a bounded linear operator $\mathfrak{E}:{W}^{1,2}(\Omega)\to{W}^{1,2}(\BR^2)$ such that $\mathfrak{E}f=f$ on $\Omega$ for all $f\in{W}^{1,2}(\Omega)$. Set $V_{2}(x_1, x_2)\equiv\frac{1}{2}(\mathfrak{E}v_2(x_1, x_2)-\mathfrak{E}v_2(x_1, -x_2))$. Since $v_2(x_1, x_2)=-v_2(x_1, -x_2)$ for all $(x_1,x_2)\in\Omega$, we notice that $V_2=v_2$ on $\Omega$. Also, by construction $V_{2}(x_1, x_2)=-V_{2}(x_1, -x_2)$ for all $(x_1,x_2)\in\BR^2$. In particular, $V_{2}(x_1, 0)=0$ for $\SL^1$-almost all $x_1\in\BR$. Hence we have
\be\label{est:05090910}
\ba
\int_{0}^{\delta}\int_{\zeta_{0}(x_2)}^{\zeta_{1}(x_2)} \frac{\abs{v_2(x_1, x_2)}^2}{x_2^2} \,dx_{1}\,dx_{2} & \le \int_{0}^{\delta}\int_{-\infty}^{\infty} \frac{\abs{V_2(x_1, x_2)}^2}{x_2^2} \,dx_{1}\,dx_{2}\\
& =\int_{-\infty}^{\infty}\int_{0}^{\delta} \frac{\abs{V_2(x_1, x_2)}^2}{x_2^2} \,dx_{2}\,dx_{1}\\
& \le4\int_{-\infty}^{\infty}\int_{0}^{\delta} \biggr|\frac{\partial}{\partial x_{2}}V_{2}(x_1, x_2)\biggr|^2 \,dx_{2}\,dx_{1}\\
& \le 4 \norm{\nabla V_2}_{L^2(\BR^2)}.
\ea
\ee
Here we have used the following one-dimensional Hardy inequality:
\be
\bigg\| \frac{u(x)}{x}\bigg\|_{L^2(0, \delta)} \le 2 \norm{u^{\prime}}_{L^2(0, \delta)} \quad \text{for all} \, u\in{W}^{1,2}(0,\delta) \enskip \text{with }u(0)=0.
\ee
By \eqref{est:05090909} and \eqref{est:05090910}, we deduce
\be
\ba
\abs{\widehat\Phi_1-\widehat\Phi_0}^2 & \le \frac{1}{\delta^2}\biggr|\int_{A(\delta)} \omega v_2 \, dx\biggr|^2\\
& \le \frac{1}{\delta^2}\biggr(\int_{0}^{\delta}\int_{\zeta_0(x_2)}^{\zeta_1(x_2)} x_2^2\frac{\abs{v_2(x_1, x_2)}^2}{x_2^2} \,dx_{1}\,dx_{2}\biggr)\biggr(\int_{A(\delta)}\omega^2\dx\biggr)\\
& \le \frac{1}{\delta^2}\biggr(4\delta^2\norm{\nabla V_2}_{L^2(\BR^2)}^2\biggr)\biggr(4\int_{A(\delta)}\abs{\nabla V}^2\,dx\biggr)=16\norm{\nabla V_2}_{L^2(\BR^2)}^2\norm{\nabla v_2}_{A(\delta)}^2.
\ea
\ee
Letting $\delta\to0$ in the above inequality gives $\widehat\Phi_0=\widehat\Phi_1$. We can prove similarly that
\be
\widehat\Phi_0=\widehat\Phi_1=\cdots=\widehat\Phi_N.
\ee
This completes the proof of Theorem \ref{thm:Amick_symmetric}.
\end{proof}
The use of Theorem \ref{thm:Amick_symmetric} and the flux condition \eqref{202401051721} in the identity \eqref{202309101143_symmetric} leads to a contradiction, and the proof of Theorem \ref{theo3} is achieved.
\section{Appendix (Proof of Theorem  \ref{thm:grS})}\label{Appendix}
We shall closely follow the methods of Amrouche-Rejaiba \cite{MR3145765}*{Theorem 4.1} and Acevedo Tapia-Amrouche-Conca-Ghosh \cite{MR4231512}*{Theorem 5.10} in which the same result was established in the case where the space dimension $n=3$. We recall that the spaces $H_q(\Omega)$ and $J_q(\Omega)$ are defined by \eqref{202401011200} and \eqref{202401011201}, respectively. Let
\be
\dot{H}(\Omega)\equiv\{\bfvarphi=(\varphi_1,\varphi_2)\in{C}^{\infty}(\overline{\Omega}): \bfvarphi\cdot\bn=0\enskip\text{on }\partial\Omega\}.
\ee
We begin to show the following density result.
\begin{lemm}\label{lem:density} Let $\Omega\subseteq\BR^2$ be a smooth bounded domain. Then $\dot{H}(\Omega)$ is dense in $H_q(\Omega)$, $1<q<\infty$. 
\end{lemm}
\begin{proof} We follow the argument of Duvaut-Lions  \cite{MR521262}*{Chapter V\hspace{-1.2pt}I\hspace{-1.2pt}I, Lemma 6.1}. Since $\Omega$ is of class $C^{\infty}$, for every point $x_{0}\in\partial\Omega$ there exist a $2\times2$ rotation matrix $R$, a $C^{\infty}$-function $\zeta\colon\BR\to\BR$ with $\zeta(x_0)=0$ and $\frac{d\zeta}{dx_1}(0)=0$ and $r>0$ such that, setting $y=T(x)=R(x-x_0)$, we have
\be
T(\Omega\cap B(x_0,r))=\{y=(y_1,y_2)\in B(0,r):y_2>\zeta(y_1)\}.
\ee
Take a sufficiently small $d>0$ so that the set
\be
U^{2d}\equiv\{y=(y_1,y_2)\in\BR^2:\abs{y_1}<2d, -2d+\zeta(y_1)<y_2<\zeta(y_1)+2d\}
\ee
is contained in $B(0,r)$. Put
\be
\ba
U^{d}&\equiv\{y=(y_1,y_2)\in\BR^2:\abs{y_1}<d, -d+\zeta(y_1)<y_2<\zeta(y_1)+d\},\\
\BR^2_{+}&\equiv\{z=(z_1,z_2)\in\BR^2:z_2>0\},\\
Q^{d}\enskip&\equiv\{z=(z_1,z_2)\in\BR^2:\abs{z_1}<d\enskip\text{and }\abs{z_2}<d\},\\
Q^{d}_{+}&\equiv Q^{d}\cap\BR^2_{+},\\
Q^{d}_{0}&\equiv\{z=(z_1,0)\in\BR^2:\abs{z_1}<d\},
\ea
\ee
and define the bijective map $\Psi\colon U^{d}\to Q^{d}$ by
\be
\Psi(y_1,y_2)=(y_1,y_2-\zeta(y_1))=(z_1,z_2).
\ee
Then we observe that $\Psi\in C^{\infty}(\overline{U^{d}})$, $\Psi^{-1}\in C^{\infty}(\overline{Q^{d}})$, $\Psi(U^{d}\cap T(\Omega))=Q^{d}_{+}$, $\Psi(U^{d}\cap T(\partial\Omega))=Q^{d}_{0}$ and $\det D\Psi=\det D\Psi^{-1}=1$.

Let $\bu\in H_q(\Omega)$ and let $\theta\in C^{\infty}_{0}(\BR^2,\BR)$ be a scalar function such that $0\le\theta\le1$ and $\supp\theta\subseteq T^{-1}(U^{d})$. We now show that the function $\theta\bu\in H_q(\Omega)$ can be approximated by elements of $\dot{H}(\Omega)$. Let us define the vector fields $\bv(y)=(v_1(y),v_2(y))$ on $T(\Omega)$ and $\widehat{\bv}(z)=(\widehat{v}_1(z),\widehat{v}_2(z))$ on $Q^{d}_{+}$ by
\be
\bv(y)\equiv R\theta\bu(T^{-1}(y)),\enskip y\in T(\Omega),
\ee
and
\be
\widehat{v}_1(z)\equiv v_1(\Psi^{-1}(z)),\enskip\widehat{v}_2(z)\equiv v_2(\Psi^{-1}(z))-\frac{d\zeta}{dz_1}(z_1)v_1(\Psi^{-1}(z)),\enskip z\in Q^{d}_{+},
\ee
respectively. Then $\widehat{v}_1,\widehat{v}_2\in W^{1,q}(Q^{d}_{+})$. Since $\supp\bv\subseteq U^{d}\cap T(\overline{\Omega})$, the transformed functions $\widehat{v}_1$ and $\widehat{v}_2$ are zero in the neighborhood of $z_2=d$ and of $\{-d,d\}\times[0,d]$. More precisely, there exists a constant $\delta>0$ such that $\supp\widehat{v}_1,\supp\widehat{v}_2\subseteq\{z=(z_1,z_2):\abs{z_1}\le d-\delta, 0\le z_2\le d-\delta\}$. Therefore, extending the transformed functions by zero outside $Q^{d}_{+}$, we may assume that $\widehat{v}_1, \widehat{v}_2\in W^{1,q}(\BR^2_{+})$. Also, the boundary condition $\gamma(\bu)\cdot\bn=0$ on $\partial\Omega$ implies $\gamma(\widehat{v}_2)=0$ on $\partial\BR^2_{+}=\BR$,\footnote{The outward unit normal $\bn_y$ to $T(\partial\Omega)$ at $(y_1,\zeta(y_1))\in U^{d}\cap T(\partial\Omega)$ is given by 
\be
\bn_y(y_1,\zeta(y_1))=(1+\frac{d\zeta}{dy_1}(y_1)^2)^{-\frac{1}{2}}(\frac{d\zeta}{dy_1}(y_1),-1).
\ee} and hence $\widehat{v}_2\in W^{1,q}_{0}(\BR^2_{+})$. Thus, $\widehat{v}_1$ can be approximated in $W^{1,q}(\BR^2_{+})$ by $\{\hat{\varphi}_j^1\}_{j\in\BN}\subseteq C^{\infty}_{0}(\BR^2)$ such that $\supp\hat{\varphi}_j^1\subseteq\{z=(z_1,z_2):\abs{z_1},\abs{z_2}\le d-\frac{\delta}{2}\}$ and $\widehat{v}_2$ can be approximated in $W^{1,q}(\BR^2_{+})$ by $\{\hat{\varphi}_j^2\}_{j\in\BN}\subseteq C^{\infty}_{0}(\BR^2)$ satisfying $\supp\hat{\varphi}_j^2\subseteq\{z=(z_1,z_2):\abs{z_1}\le d-\frac{\delta}{2},0\le z_2\le d-\frac{\delta}{2}\}$ and $\hat{\varphi}_j^2(z_1,0)=0$ for any $z_1\in\BR$. We now retransfer $\hat{\varphi}_j^1$ and $\hat{\varphi}_j^2$ to $U^{d}$ using $\Psi$ and denote them by $\varphi_j^1$ and $\varphi_j^2$, i.e.,
\be
\varphi_j^1(y)=\hat{\varphi}_j^1(\Psi(y)),\enskip\varphi_j^2(y)=\hat{\varphi}_j^2(\Psi(y))+\frac{d\zeta}{dy_1}(y_1)\hat{\varphi}_j^1(\Psi(y))
\ee
for $y\in U^{d}$. We then have $\bfvarphi_j=(\varphi_j^1,\varphi_j^2)\in C^{\infty}_{0}(U^{d})$. Extending $\bfvarphi_j$ by zero outside $U^{d}$, we have $\bfvarphi_j\in C^{\infty}_{0}(\BR^2)$ and $\bfvarphi_j\cdot\bn_y=0$ on $T(\partial\Omega)$, where $\bn_y=\bn_y(y)$ is the unit outer normal to $T(\partial\Omega)$ in the coordinate system $\{y\}$. Finally, put $\bfeta_j(x)\equiv R^{\top}\bfvarphi_j(T(x))$ for $x\in\BR^2$. It is then immediately verified that $\bfeta_j\in\dot{H}(\Omega)$ and $\norm{\theta\bu-\bfeta_j|_{\Omega}}_{W^{1,q}(\Omega)}\to0$ as $j\to\infty$. 

Since $\partial\Omega$ is compact, the statement of Lemma \ref{lem:density} will then follow via a partition of unity subordinate to the finite open covering of $\partial\Omega$. 
\end{proof}

We shall next prove the existence of a weak solution to the following auxiliary problem:
\begin{align} \label{eq:vorticity}
\left\{
\ba
\curl\nabla^{\perp}\omega& =\curl \bff \quad & & \text{in} \enskip \Omega,\\
\omega & = g \quad & & \text{on} \enskip \partial \Omega.
\ea 
\right. 
\end{align}
Recall that $\nabla^{\perp}\omega=(-\partial_{2}\omega,\partial_{1}\omega)$ for a scalar function $w$, and $\curl\bff=\partial_{2}f_1-\partial_{1}f_2$ for a vector-valued function $\bff=(f_1,f_2)$.
\begin{defi}Let $\bff\in L^q(\Omega)$ and $g\in W^{1-1/q,q}(\partial\Omega)$, $1<q<\infty$. A scalar function $\omega\in W^{1,q}(\Omega)$ is called a $q$-weak (or $q$-generalized) solution to the problem \eqref{eq:vorticity} if it satisfies the integral identity
\be
\int_{\Omega}\nabla^{\perp}\omega\cdot\nabla^{\perp}\varphi\,dx=\int_{\Omega}\bff\cdot\nabla^{\perp}\varphi\,dx
\ee
for all scalar functions $\varphi\in W^{1,q'}_{0}(\Omega)$, $q'=\frac{q}{q-1}$, and $\gamma(\omega)=g$. Here $\gamma\colon W^{1,q}(\Omega)\to L^q(\partial\Omega)$ is the trace operator. 
\end{defi}
\begin{prop}\label{prop:vorticity}
Let $\Omega\subseteq\BR^2$ be a smooth bounded domain, and let $1<q<\infty$. Then for any $\bff\in L^q(\Omega)$ and $g\in W^{1-1/q,q}(\partial\Omega)$, there exists a unique $q$-weak solution $\omega\in W^{1,q}(\Omega)$ to the problem \eqref{eq:vorticity}. Moreover, the following estimate holds:
\be\label{202409201252}
\norm{\omega}_{W^{1,q}(\Omega)}\le c(\norm{\bff}_{L^q(\Omega)}+\norm{g}_{W^{1-1/q,q}(\partial\Omega)}),
\ee
where $c=c(q,\Omega)$.
\end{prop}
\begin{proof}From \cite{MR2808162}*{Theorem I\hspace{-1.2pt}I.4.3} there exists $G\in W^{1,q}(\Omega)$ with $\gamma(G)=g$ such that 
\be\label{202409201253}
\norm{G}_{W^{1,q}(\Omega)}\le c\norm{g}_{W^{1-1/q,q}(\partial\Omega)},
\ee
where $c=c(q,\Omega)>0$. We then look for a $q$-weak solution to the problem \eqref{eq:vorticity} of the form $\omega=\widehat{\omega}+G$, where $\widehat{\omega}\in W^{1,q}_{0}(\Omega)$ satisfies the integral identity
\be
\int_{\Omega}\nabla^{\perp}\widehat{\omega}\cdot\nabla^{\perp}\varphi\,dx=\int_{\Omega}(\bff-\nabla^{\perp}G)\cdot\nabla^{\perp}\varphi\,dx
\ee 
for all scalar functions $\varphi\in W^{1,q'}_{0}(\Omega)$, $q'=\frac{q}{q-1}$. The above identity is equivalent to the following one:
\be\label{202409201254}
\int_{\Omega}\nabla\widehat{\omega}\cdot\nabla\varphi\,dx=-\int_{\Omega}(\bff^{\perp}+\nabla G)\cdot\nabla\varphi\,dx,
\ee 
where $(x_1,x_2)^{\perp}=(-x_2,x_1)$. Since $\bff^{\perp}+\nabla G\in L^q(\Omega)$, the right-hand side of \eqref{202409201254} defines a linear functional on $W^{1,q'}_{0}(\Omega)$. Thus, according to \cite{MR1454361}*{Chapter I\hspace{-1.2pt}I, Theorem 1.2}, there exists a unique function $\widehat{\omega}\in W^{1,q}_{0}(\Omega)$ satisfying the identity \eqref{202409201254} for all $\varphi\in W^{1,q'}_{0}(\Omega)$. Moreover, such a $\widehat{\omega}$ satisfies the following estimate
\be\label{202409201255}
\norm{\widehat{\omega}}_{W^{1,q}(\Omega)}\le c\norm{\bff^{\perp}+\nabla G}_{L^q(\Omega)}\le c\left(\norm{\bff}_{L^q(\Omega)}+\norm{G}_{W^{1,q}(\Omega)}\right),
\ee
where $c=c(q,\Omega)>0$. The desired $q$-weak solution to the problem \eqref{eq:vorticity} is then given by $\omega=\widehat{\omega}+G\in W^{1,q}(\Omega)$, and the estimate \eqref{202409201252} follows from \eqref{202409201253} and \eqref{202409201255}. It remains to prove the uniqueness of the $q$-weak solution. To this end, denote by $\omega_1$ another $q$-weak solution corresponding to the same data. Then $\omega-\omega_1\in W^{1,q}_{0}(\Omega)$ satisfies the identity
\be
\int_{\Omega}\nabla^{\perp}(\omega-\omega_1)\cdot\nabla^{\perp}\varphi\,dx=\int_{\Omega}\nabla(\omega-\omega_1)\cdot\nabla\varphi\,dx=0
\ee
for all $\varphi\in W^{1,q'}_{0}(\Omega)$. By \cite{MR1454361}*{Chapter I\hspace{-1.2pt}I, Theorem 3.1}, it follows that $\omega-\omega_1=0$ for almost all $x\in\Omega$. The proof of the proposition is therefore completed. 
\end{proof}
Let $1<q<\infty$. We recall that for $\bu=(u_1,u_2)\in L^q(\Omega)$ with $\dv\bu\in L^q(\Omega)$ the normal component $\gamma_{\bn}\bu=\bn\cdot\bu|_{\partial\Omega}$ is well defined as a member of $W^{-1/q,q}(\partial\Omega)$ and the following inequality holds
\be
\norm{\gamma_{\bn}\bu}_{W^{-1/q,q}(\partial\Omega)}\le c(\norm{\bu}_{L^q(\Omega)}+\norm{\dv\bu}_{L^q(\Omega)})
\ee
with $c=c(q, \Omega)$; see \cite{MR2808162}*{Theorem I\hspace{-1.2pt}I\hspace{-1.2pt}I.2.2}, \cite{MR1928881}*{pp.\,50-51}. The following proposition is useful to prove the regularity of solutions to the Stokes system \eqref{S}.
\begin{prop}\label{prop:regularity}Let $\Omega\subseteq\BR^2$ be a smooth bounded domain. Let $m\ge1$ and $1<q<\infty$. Then
\be
W^{m,q}(\Omega)=\{\bu\in L^q(\Omega):\curl\bu\in W^{m-1,q}(\Omega), \dv\bu\in W^{m-1,q}(\Omega), \gamma_{\bn}\bu\in W^{m-1/q,q}(\partial\Omega)\}
\ee
and there exists $c=c(m, q, \Omega)>0$ such that
\be
\norm{\bu}_{W^{m,q}(\Omega)}\le c\left(\norm{\bu}_{L^q(\Omega)}+\norm{\curl\bu}_{W^{m-1,q}(\Omega)}+\norm{\dv\bu}_{W^{m-1,q}(\Omega)}+\norm{\gamma_{\bn}\bu}_{W^{m-1/q,q}(\partial\Omega)}\right)
\ee
for all $\bu\in W^{m,q}(\Omega)$.
\end{prop}
\begin{proof}
See Temam \cite{MR1846644}*{Proposition 1.4 in Appendix I} for the case $q=2$. For the general case $1<q<\infty$, see Kozono-Yanagisawa \cite{MR2542982}*{Lemmas 4.5 and 4.6}\footnote{In \cite{MR2542982}, the result has been proved in the case where the space dimension $n=3$, but the two-dimensional case may be dealt with analogously.}.
\end{proof}
\begin{prop}\label{weak_curl} Let $\Omega\subseteq\BR^2$ be a smooth bounded domain, let $\beta\equiv0$, and let $1<q<\infty$, $q'=\frac{q}{q-1}$. Suppose that $\bu\in J_q(\Omega)$ is a $q$-weak solution to the Stokes system \eqref{S} corresponding to $\bff\in H_{q'}(\Omega)'$, $a_{\ast}\equiv0$ and $\bb_{\ast}\in W^{-1/q,q}(\partial\Omega)$. Denote by $p$ the pressure associated with $\bu$ by Lemma $\ref{pressure_q}$. Then the pair $(\bu,p)$ satisfies the following integral identity
\be
\nu\int_{\Omega}\curl\bu\,\curl\bfvarphi\,dx-\langle {\bb_{\ast}},\bfvarphi \rangle_{\partial\Omega}+2\nu\int_{\partial\Omega}W^\top\gamma(\bu)\cdot\bfvarphi\,ds=\langle \bff,\bfvarphi \rangle_{\Omega}+\int_{\Omega}p\,\dv\bfvarphi\,dx
\ee
for all $\bfvarphi\in H_{q'}(\Omega)$. Here $W$ is the Weingarten map of $\partial\Omega$ defined in Proposition $\ref{prop:Weingarten}$.
\end{prop}
\begin{proof}
Since $\dot{H}(\Omega)=\{\bfvarphi\in C^{\infty}(\overline{\Omega}):\bfvarphi\cdot\bn=0\enskip\text{on }\partial\Omega\}$ is dense in $H_q(\Omega)$, $1<q<\infty$ (see Lemma \ref{lem:density}), it suffices to show that the following identity holds for all $\bu,\bfvarphi\in\dot{H}(\Omega)$:
\be\label{09011714}
\frac{\nu}{2}\int_{\Omega}S(\bu):S(\bfvarphi)\,dx=\nu\int_{\Omega}\curl\bu\,\curl\bfvarphi\,dx+2\nu\int_{\Omega}\dv\bu\,\dv\bfvarphi\,dx+2\nu\int_{\partial\Omega}W^\top\bu\cdot\bfvarphi\,ds.
\ee
Let $\bu,\bfvarphi\in \dot{H}(\Omega)$. Since
\be
\dv(\nu S(\bu)\bfvarphi)=\frac{\nu}{2}S(\bu):S(\bfvarphi)+\nu\Delta\bu\cdot\bfvarphi+\nu\nabla\dv\bu\cdot\bfvarphi\quad\text{in}\enskip\Omega,
\ee
integrating this relation over $\Omega$, we have
\be\label{09011715}
\nu\int_{\partial\Omega}[S(\bu)\bn]_{\tau}\cdot\bfvarphi\,ds=\frac{\nu}{2}\int_{\Omega}S(\bu):S(\bfvarphi)\,dx+\nu\int_{\Omega}\Delta\bu\cdot\bfvarphi\,dx+\nu\int_{\Omega}\nabla\dv\bu\cdot\bfvarphi\,dx.
\ee
On the other hand, since
\be
\nu\Delta\bu+\nu\nabla\dv\bu=2\nu\nabla\dv\bu-\nu\nabla^{\perp}\curl\bu\quad\text{in}\enskip\Omega,
\ee
multiplying this relation by $\bfvarphi\in\dot{H}(\Omega)$ and integrating by parts over $\Omega$, we obtain
\be\label{09011716}
\ba
&\nu\int_{\Omega}\Delta\bu\cdot\bfvarphi\,dx+\nu\int_{\Omega}\nabla\dv\bu\cdot\bfvarphi\,dx=2\nu\int_{\Omega}\nabla\dv\bu\cdot\bfvarphi\,dx-\nu\int_{\Omega}\nabla^{\perp}\curl\bu\cdot\bfvarphi\,dx\\
=&-2\nu\int_{\Omega}\dv\bu\,\dv\bfvarphi\,dx-\nu\int_{\Omega}\curl\bu\,\curl\bfvarphi\,dx+\nu\int_{\partial\Omega}\curl\bu(n_2,-n_1)^\top\cdot\bfvarphi\,ds.
\ea
\ee
Since $\bu\cdot\bn=0$ on $\partial\Omega$, it follows from Lemma \ref{lem_def-curl} that
\be\label{09011717}
\curl\bu(n_2,-n_1)^\top=[S(\bu)\bn]_{\tau}-2W^\top\bu\quad\text{on}\enskip\partial\Omega.
\ee
Then the desired identity \eqref{09011714} follows from \eqref{09011715}, \eqref{09011716} and \eqref{09011717}. This proves Proposition \ref{weak_curl}.
\end{proof}
We are now ready to prove Theorem \ref{thm:grS}.
\begin{proof}[Proof of Theorem $\ref{thm:grS}$]The proof will be achieved in three steps. In the first two steps, we show that $(\bu,p)\in W^{2,q}(\Omega)\times W^{1,q}(\Omega)$ together with the following estimate
\be\label{est:grS1}
\norm{\bu}_{W^{2,q}(\Omega)}+\norm{p}_{W^{1,q}(\Omega)}\le c\Bigl(\norm{\bff}_{L^{q}(\Omega)}+\norm{a_{\ast}}_{W^{2-1/q,q}(\partial\Omega)}+\norm{\bb_{\ast}}_{W^{1-1/q,q}(\partial\Omega)}+\norm{\bu}_{W^{1,q}(\Omega)}\Bigr),
\ee
where $c=c(q,\Omega,\nu,\beta)$.

\textbf{Step 1:} Let us first consider the case $\beta\equiv0$. According to Lemma \ref{solenoidalextension} (more precisely, its proof), for the boundary datum $a_{\ast}\in W^{2-1/q,q}(\partial\Omega)$ with $\int_{\partial\Omega} a_{\ast}\,ds=0$, there exists a solenoidal extension $\bA\in W^{2,q}(\Omega)$ of $a_{\ast}$ such that
\be\label{09020733}
\norm{\bA}_{W^{2,q}(\Omega)}\le c\norm{a_{\ast}}_{W^{2-1/q,q}(\partial\Omega)},
\ee
with $c=c(q, \Omega)>0$. Define the vector field $\bv\in J_q(\Omega)$ by $\bv\equiv\bu-\bA$. Then $\bv\in J_q(\Omega)$ can be regard as a $q$-weak solution to the following system
\begin{align} \label{S_hom}
\left\{
\ba
-\nu \Delta \bv+ \nabla p & =\bff+\nu\Delta\bA \quad & & \text{in} \enskip \Omega,\\
\dv \bv & =0 \quad & & \text{in} \enskip \Omega,\\
\bv \cdot\bn & =0 \quad & & \text{on} \enskip \partial \Omega,\\
[T(\bv,p)\bn]_{\tau}& =\bb_{\ast}-[\nu S(\bA)\bn]_{\tau}  \quad & &\text{on} \enskip \partial \Omega,
\ea 
\right. 
\end{align}
and the $W^{2,q}$ regularity of $\bu=\bv+\bA$ follows from that of $\bv$. Applying the operator curl to both sides of $\eqref{S_hom}_1$, and using Lemma \ref{lem_def-curl} along with $\eqref{S_hom}_{3,4}$ we observe that $\omega\equiv\curl \bv$ formally satisfies the following boundary value problem:
\begin{align} \label{vorticity}
\left\{
\ba
\nu\,\curl\nabla^{\perp}\omega& =\curl(\bff+\nu\Delta\bA) \quad & & \text{in} \enskip \Omega,\\
\omega& =\big(\frac{1}{\nu}\bb_{\ast}-[S(\bA)\bn]_{\tau}-2W^\top\bv\big)\cdot(n_2,-n_1)^\top \quad & &\text{on} \enskip \partial \Omega.
\ea 
\right. 
\end{align}
Since $\bff+\nu\Delta\bA\in L^q(\Omega)$ and $\big(\frac{1}{\nu}\bb_{\ast}-[S(\bA)\bn]_{\tau}-2W^\top\bv\big)\cdot(n_2,-n_1)^\top\in W^{1-1/q,q}(\Omega)$, according to Proposition \ref{prop:vorticity}, the problem \eqref{vorticity} has a solution $z\in W^{1,q}(\Omega)$ such that
\be\label{09020723}
\nu\int_{\Omega}\nabla^{\perp}z\cdot\nabla^{\perp}\varphi\,dx=\int_{\Omega}(\bff+\nu\Delta\bA)\cdot\nabla^{\perp}\varphi\,dx
\ee
for all scalar functions $\varphi\in W^{1,q'}_{0}(\Omega)$ and $\gamma(z)=\big(\frac{1}{\nu}\bb_{\ast}-[S(\bA)\bn]_{\tau}-2W^\top\bv\big)\cdot(n_2,-n_1)^\top$. Moreover, the solution $z$ satisfies the following estimate:
\be\label{09020724}
\norm{z}_{W^{1,q}(\Omega)}\le c(\norm{\bff+\nu\Delta\bA}_{L^q(\Omega)}+\norm{b_{\ast}}_{W^{1-1/q,q}(\partial\Omega)}+\norm{\nabla\bA}_{W^{1,q}(\Omega)}+\norm{\bv}_{W^{1,q}(\Omega)})
\ee
with $c=c(q,\Omega,\nu)$. In order to prove the $W^{2,q}$ regularity of $\bv$ with the aid of Proposition \ref{prop:regularity}, we shall show that $\curl\bv\in W^{1,q}(\Omega)$. More precisely, one can prove that $\nabla^{\perp}\curl\bv=\nabla^{\perp}z$. For any $\bfeta=(\eta_1,\eta_2)\in C^{\infty}_{0}(\Omega)$ we apply the Helmholtz-Weyl decomposition (\cite{MR2550139}*{Theorem 3.20}) to obtain 
\be\label{09020422}
\bfeta=\bh+\nabla^{\perp}w+\nabla r, 
\ee
where $\bh\in\tilde{X}_{\text{har}}(\Omega)$, $w\in W^{1,q'}_{0}(\Omega)\cap W^{2,q'}(\Omega)$ and $r\in W^{2,q'}(\Omega)$. Here
\be
\tilde{X}_{\text{har}}(\Omega)\equiv\{\bh=(h_1,h_2)\in C^{\infty}(\overline{\Omega}):\dv \bh=0,\,\curl \bh=0\enskip\text{in }\Omega,\,\bh\cdot\bn=0\enskip\text{on }\partial\Omega\}.
\ee
For a vector-valued function $\bh=(h_1,h_2)$ we set $\bh\wedge\bn\equiv h_2n_1-h_1n_2$. Since $\curl\nabla r=0$ in $\Omega$ and $\bfeta|_{\partial\Omega}=\boldsymbol{0}$, by integration by parts we have
\be\label{09020720}
\ba
&\int_{\Omega}\nabla^{\perp}z\cdot\bfeta\,dx=\int_{\Omega}\nabla^{\perp}z\cdot \bh\,dx+\int_{\Omega}\nabla^{\perp}z\cdot\nabla^{\perp}w\,dx+\int_{\Omega}\nabla^{\perp}z\cdot\nabla r\,dx\\
=&\int_{\Omega}z\,\curl \bh\,dx+\int_{\partial\Omega}z\bh\wedge\bn\,ds+\int_{\Omega}\nabla^{\perp}z\cdot\nabla^{\perp}w\,dx+\int_{\Omega}z\,\curl\nabla r\,dx+\int_{\partial\Omega}z\nabla r\wedge\bn\,ds\\
=&\int_{\Omega}\nabla^{\perp}z\cdot\nabla^{\perp}w\,dx+\int_{\partial\Omega}z(\bfeta-\nabla^{\perp}w)\wedge\bn\,ds\\
=&\frac{1}{\nu}\int_{\Omega}(\bff+\nu\Delta\bA)\cdot\nabla^{\perp}w\,dx-\int_{\partial\Omega}\big(\frac{1}{\nu}\bb_{\ast}-[S(\bA)\bn]_{\tau}-2W^\top\bv\big)\cdot(n_2,-n_1)^\top(\nabla^{\perp}w\wedge\bn)\,ds.
\ea
\ee
On the other hand, since $\bv\in J_q(\Omega)$ is a $q$-weak solution to the system \eqref{S_hom}, from Proposition \ref{weak_curl} it follows that the pair $(\bv,p)$ satisfies the following identity
\be\label{iden:weak_curl}
\ba
&\nu\int_{\Omega}\curl\bv\,\curl\bfvarphi\,dx-\int_{\partial\Omega}(\bb_{\ast}-[\nu S(\bA)\bn]_{\tau})\cdot\bfvarphi\,ds+2\nu\int_{\partial\Omega}W^\top\bv\cdot\bfvarphi\,ds\\
=&\int_{\Omega}(\bff+\nu\Delta\bA)\cdot\bfvarphi\,dx+\int_{\Omega}p\,\dv\bfvarphi\,dx
\ea
\ee
for all $\bfvarphi\in H_{q'}(\Omega)$. Since the tangential gradient of $w\in W^{1,q'}_{0}(\Omega)\cap W^{2,q'}(\Omega)$ appearing in the decomposition \eqref{09020422} is equal to zero, we have $\gamma(\nabla^{\perp}w)\cdot\bn=0$ on $\partial\Omega$. This implies that $\nabla^{\perp}w\in H_{q'}(\Omega)$ and $\nabla^{\perp}w$ can be written as $\nabla^{\perp}w=(I_2-\bn\otimes\bn)\nabla^{\perp}w+\bn(\bn\cdot\nabla^{\perp}w)=-(\nabla^{\perp}w\wedge\bn)(n_2,-n_1)^\top$ on $\partial\Omega$. Taking into account that $\nabla^{\perp}w\in H_{q'}(\Omega)$, we have from \eqref{09020422} and \eqref{iden:weak_curl} that
\be
\ba
&\int_{\Omega}\curl\bv\,\curl\bfeta\,dx=\int_{\Omega}\curl\bv\,\curl\nabla^{\perp}w\,dx\\
&=\frac{1}{\nu}\int_{\Omega}(\bff+\nu\Delta\bA)\cdot\nabla^{\perp}w\,dx+\int_{\partial\Omega}\big(\frac{1}{\nu}\bb_{\ast}-[S(\bA)\bn]_{\tau}-2W^\top\bv\big)\cdot\nabla^{\perp}w\,ds\\
&=\frac{1}{\nu}\int_{\Omega}(\bff+\nu\Delta\bA)\cdot\nabla^{\perp}w\,dx-\int_{\partial\Omega}\big(\frac{1}{\nu}\bb_{\ast}-[S(\bA)\bn]_{\tau}-2W^\top\bv\big)\cdot(n_2,-n_1)^\top(\nabla^{\perp}w\wedge\bn)\,ds,
\ea
\ee
from which and \eqref{09020720} we conclude that
\be
\int_{\Omega}\nabla^{\perp}z\cdot\bfeta\,dx=\int_{\Omega}\curl\bv\,\curl\bfeta\,dx
\ee
for all $\bfeta=(\eta_1,\eta_2)\in C^{\infty}_{0}(\Omega)$. This relation implies that $\nabla^{\perp}\curl\bv=\nabla^{\perp}z$ in $\Omega$, and hence $\curl\bv\in W^{1,q}(\Omega)$. Using this fact, $\eqref{S_hom}_{2,3}$ and Proposition \ref{prop:regularity} together with \eqref{09020724}, we find that $\bv\in W^{2,q}(\Omega)$ with the estimate
\be
\ba
\norm{\bv}_{W^{2,q}(\Omega)}&\le c(\norm{\bv}_{L^{q}(\Omega)}+\norm{\curl\bv}_{W^{1,q}(\Omega)})\le c(\norm{\bv}_{W^{1,q}(\Omega)}+\norm{\nabla z}_{L^q(\Omega)})\\
&\le c(\norm{\bv}_{W^{1,q}(\Omega)}+\norm{\bff}_{L^q(\Omega)}+\norm{\bA}_{W^{2,q}(\Omega)}+\norm{\bb_{\ast}}_{W^{1-1/q,q}(\partial\Omega)}),
\ea
\ee
where $c=c(q,\Omega,\nu)$. Recalling that $\bu=\bv+\bA$, we then discover the estimate above and \eqref{09020733} yield that
\be\label{202409300716}
\norm{\bu}_{W^{2,q}(\Omega)}\le c(\norm{\bu}_{W^{1,q}(\Omega)}+\norm{\bff}_{L^q(\Omega)}+\norm{a_{\ast}}_{W^{2-1/q,q}(\partial\Omega)}+\norm{\bb_{\ast}}_{W^{1-1/q,q}(\partial\Omega)}).
\ee
The $W^{1,q}$ regularity of $p$ follows from the equation $\nabla p=\nu\Delta\bu+\bff\in L^q(\Omega)$ with the estimate $\norm{\nabla p}_{L^q(\Omega)}\le c(\nu)\left(\norm{\bu}_{W^{2,q}(\Omega)}+\norm{\bff}_{L^q(\Omega)}\right)$. This estimate, \eqref{202409300718} and \eqref{202409300716} yield the desired estimate \eqref{est:grS1}.

\textbf{Step 2:} Let us next consider the case $\beta\not\equiv0$. Since $\bu\in W^{1,q}(\Omega)$ and $\beta\in C^1(\partial\Omega)$, we see that $\norm{\beta\bu_{\tau}}_{W^{1-1/q,q}(\partial\Omega)}\le c\norm{\bu}_{W^{1,q}(\Omega)}$ with $c=c(q,\Omega,\beta)$. Thus we may regard $(\bu,p)$ as a solution corresponding to the boundary datum $\bb_{\ast}-\beta\bu_{\tau}\in W^{1-1/q,q}(\partial\Omega)$. The regularity result for the case $\beta\equiv0$ proved in the first step then implies that $(\bu,p)\in W^{2,q}(\Omega)\times W^{1,q}(\Omega)$ along with the estimate \eqref{est:grS1}. 

\textbf{Step 3:} To prove Theorem \ref{thm:grS} completely, it remains to show that the term $\norm{\bu}_{W^{1,q}(\Omega)}$ on the right-hand side of \eqref{est:grS1} can be dropped. To this end, it is enough to prove that there exists a constant $c=c(q,\Omega,\nu,\beta)>0$, which is independent of the given data and the particular solution, such that
\be
\norm{\bu}_{W^{1,q}(\Omega)}\le c\Bigl(\norm{\bff}_{L^{q}(\Omega)}+\norm{a_{\ast}}_{W^{2-1/q,q}(\partial\Omega)}+\norm{\bb_{\ast}}_{W^{1-1/q,q}(\partial\Omega)}\Bigr).
\ee
This estimate can be shown by employing the same contradiction argument used in \cite{MR2808162}*{p.\,280} together with Lemma \ref{lem:unique_q} given below. The proof of the theorem is therefore completed, once we prove the following uniqueness result. 
\end{proof}
\begin{lemm}\label{lem:unique_q}Let $\Omega\subseteq\BR^2$ be a smooth bounded domain. Let $\beta\in C^1(\partial\Omega,\BR)$ be nonnegative, and let $1<q<\infty$. Suppose that $\bu$ is a $q$-weak solution to the Stokes problem \eqref{S} with $\bff\equiv\bb_{\ast}\equiv\boldsymbol{0}$ and $a_{\ast}\equiv0$. Assume further that $\int_{\Omega} \bu\cdot\bu_0\,dx=0$ in the case when $\beta\equiv0$ and $H(\Omega)\cap\SR(\Omega)=\mathrm{Span}\{\bu_{0}\}$,\footnote{See footnote \ref{foot:kernel}.} where $\bu_{0}$ is a nonzero rigid displacement. Then, it follows that $\bu\equiv\boldsymbol{0}$, $p\equiv0$ for almost all $x\in\Omega$, where $p\in L^q(\Omega)$ is the pressure associated with $\bu$ by Lemma $\ref{pressure_q}$.
\end{lemm}
\begin{rema} Remark \ref{rem:Unique_S} equally applies to Lemma \ref{lem:unique_q}.
\end{rema}
\begin{proof}[Proof of Lemma $\ref{lem:unique_q}$]If $q\ge2$, the $q$-weak solution $\bu$ is a (2-)weak solution. Thus, the result follows from the uniqueness of a weak solution; see Theorem \ref{thm:exist_unique_S}. If $1<q<2$, by using an iterative argument based on the first part of Theorem \ref{thm:grS} and the Sobolev embedding theorem, we can show that $\bu\in W^{1,2}(\Omega)$ and $p\in L^2(\Omega)$; see \cite{MR2808162}*{pp.\,280-281} for details. Again, Theorem \ref{thm:exist_unique_S} implies the desired result.
\end{proof}

\noindent
{\bf Acknowledgements.} We would like to express our deepest gratitude to Mikhail V. Korobkov for providing us with Example \ref{exam:Hamel} and for his valuable comments on the proof of Lemma \ref{lem:keyestimate} and detailed explanation of the argument used in \cite{MR3275850}. We would also like to thank Matthias K\"ohne for fruitful discussions on the proof of Theorem \ref{thm:Amick_symmetric} while T. Yamamoto was visiting Heinrich-Heine-Uni\-ver\-sit\"at D\"usseldorf in March 2023. Finally, we would like to thank the referees for their valuable comments and suggestions.

The work of G.P. Galdi is partially supported by the National Science Foundation Grant DMS--2307811. The work of T. Yamamoto was supported by Grant-in-Aid for JSPS Research Fellow, Grant Number JP22KJ2953.

\bigskip
\noindent
{\bf Data Availability Statement} Data sharing not applicable to this article as no datasets were generated or analysed during the current study.
\bigskip

\noindent
{\bf Declarations}

\noindent
{\bf Conflict of interest}
The authors state that there is no conflict of interest.


\end{document}